\documentclass[envcountsect,runningheads,a4paper]{llncs}

\usepackage{stmaryrd}
\usepackage{amssymb}
\usepackage{amsmath}
\usepackage{graphicx}
\usepackage{proof}
\usepackage{enumerate}
\usepackage{appendix}
\usepackage{enumitem}
\usepackage{multicol}
\usepackage{ifthen}
\usepackage{subcaption}
\usepackage{hyperref}
\hypersetup{
	colorlinks=true,
	linkcolor=blue,
	citecolor=blue
}
\usepackage{epigraph}

\usepackage{tikz}
\usetikzlibrary{arrows, topaths, calc, positioning,shapes,intersections}

\tikzstyle{node} = [circle, minimum size = 1.2mm, inner sep = 0mm, color=black, fill]
\tikzstyle{hyperedge} = [rectangle, minimum width = 6mm, minimum height = 6mm, draw, inner sep = 0mm, color = black]
\tikzstyle{HG} = [align = center]

\newcommand{\downsquigarrow}{\mathbin{\rotatebox[origin=c]{-90}{$\rightsquigarrow$}}}

\begin{document}
	\title{Hypergraph Lambek Calculus}
	\author{Tikhon Pshenitsyn \orcidID{0000-0003-4779-3143}\thanks{The study was funded by RFBR, project number 20-01-00670.}\\ tpshenitsyn@lpcs.math.msu.su}
	\authorrunning{T. Pshenitsyn}
	\institute{Department of Mathematical Logic and Theory of Algorithms\\Faculty of Mathematics and Mechanics\\Lomonosov Moscow State University\\GSP-1, Leninskie Gory, Moscow, 119991, Russian Federation}
	\maketitle
	
\begin{abstract}
	It is known that context-free grammars can be extended to generating graphs resulting in graph grammars; one of such fundamental approaches is hyperedge replacement grammars. On the other hand there are type-logical grammars which also serve to describe string languages. In this paper, we investigate how to extend the Lambek calculus ($\mathrm{L}$) and grammars based on it to graphs. The resulting approach is called hypergraph Lambek calculus ($\mathrm{HL}$). It is a logical sequential calculus whose sequents are graphs; it naturally extends the Lambek calculus and also allows one to embed its variants (commutative $\mathrm{L}$, $\mathrm{NL\diamondsuit}$, $\mathrm{L}_{\mathbf{1}}^\ast$). Besides, many properties of the Lambek calculus (cut elimination, counters, models) can be lifted to $\mathrm{HL}$. However, while Lambek grammars are equivalent to context-free grammars in the string case, hypergraph Lambek grammars are much more powerful than hyperedge replacement grammars. Particularly, the former can generate the language of all graphs without isolated nodes; the language of all bipartite graphs; finite intersections of languages generated by hyperedge replacement grammars. Nevertheless, the derivability problem in $\mathrm{HL}$ and the membership problem for grammars based on $\mathrm{HL}$ are NP-complete as well as the membership problem for hyperedge replacement grammars.
\end{abstract}
	\tableofcontents
	\newpage
\section{Introduction}\label{sec_intr}
	The Lambek calculus ($\mathrm{L}$) was firstly introduced in \cite{Lambek} by Joachim Lambek. It is a sequential calculus which appears to be useful in describing natural languages. In the standard variant of the Lambek calculus types are built using two divisions and product; the calculus itself includes one axiom and six inference rules.
	
	The Lambek calculus ($\mathrm{L}$) is the basis of Lambek grammars, which describe string languages in a categorial way. Namely, a grammar contains a correspondence between symbols of an alphabet and types of $\mathrm{L}$. A string is generated by the grammar if a sequent composed of types corresponding to symbols of the string can be proved in $\mathrm{L}$. Such grammars are called type-logical since they operate with types based on a logical calculus. They are opposed to context-free grammars, which generate strings using productions and not by means of a logical system.
	
	Since 1958 until nowadays the Lambek calculus has been significantly improved, its different extensions have been presented in a number of works. For example, $\mathrm{L}$ with conjuction and disjunction is considered, see \cite{Kanasawa}; $\mathrm{L}$ with modalities is presented in the work of Michael Moortgat \cite{Moortgat}; $\mathrm{L}$ with the permutation rule $\mathrm{LP}$ is studied, etc. Many fundamental properties of the Lambek calculus and of its extensions have been discovered. Mati Pentus proved that the derivability problem in $\mathrm{L}$ is NP-complete (see \cite{Pentus_complexity}); L-models and R-models were introduced, completeness was established by Pentus in \cite{Pentus_models}. Regarding Lambek grammars, Pentus proved \cite{Pentus_cfg} that the class of languages generated by Lambek grammars equals the class of context-free languages.
	
	The second field of research that should be mentioned in this work is the theory of graph grammars. Generalizing context-free grammars (CFGs), graph grammars produce graph languages using rewriting rules. An overview of graph grammars is given in the handbook \cite{Rosenberg}; a wide variety of mechanisms generating graphs is presented there. We focus on a particular approach called \emph{hyperedge replacement grammar} (HRG in short) introduced by Feder \cite{Feder} and Pavlidis \cite{Pavlidis} since it is very close to context-free grammars in terms of definitions and structural properties. Hyperedge replacement grammars derive hypergraphs by means of productions: a production allows one to replace an edge of a hypergraph with another hypergraph. Hyperedge replacement grammars (HRGs) have a number of properties in common with CFGs such as the pumping lemma, the fixed-point theorem, the Parikh theorem, the Greibach normal form etc. An overview of hyperedge replacement grammars can be found in \cite{Drewes}.
	
	Being impressed by many similarities between HRGs and CFGs, we were curious whether it is possible to generalize type-logical grammars to hypergraphs. We started with basic categorial grammars and introduced hypergraph basic categorial grammars (see \cite{Pshenitsyn}); we showed their duality with HRGs and also a number of similarities with basic categorial grammars. Our goal now is to do the same with the Lambek calculus. We wish to construct a generalization of the Lambek calculus that has types, sequents, axioms, inference rules; but now hypergraphs instead of strings are to be involved in this mechanism. We wish to preserve fundamental features of the Lambek calculus, e.g. the cut elimination, L- and R-models as well as to extend Lambek grammars in such a way that there will be duality between them and hyperedge replacement grammars.
	
\section{Preliminaries: string formalisms}
$\mathbb{N}$ includes $0$. The set $\Sigma^*$ is the set of all strings over the alphabet $\Sigma$ including the empty string $\Lambda$. The length $|w|$ of the word $w$ is the number of symbols in $w$. $\Sigma^+$ denotes the set of all nonempty strings. The set $\Sigma^\circledast$ is the set of all strings consisting of distinct symbols. The set of all symbols contained in the word $w$ is denoted by $[w]$. If $f:\Sigma\to\Delta$ is a function from one set to another, then it is naturally extended to a function $f:\Sigma^*\to\Delta^*$ ($f(\sigma_1\dots\sigma_k)=f(\sigma_1)\dots f(\sigma_k)$).

We start with a very brief introduction of well-known context-free string grammars.

\begin{definition}
	A \emph{context-free grammar} is a tuple $\langle N, \Sigma, P, S\rangle$, where $N$ is a finite alphabet of nonterminal symbols, $\Sigma$ is a finite alphabet of terminal symbols ($N\cap\Sigma=\emptyset$), $P$ is a set of productions, and $S\in N$. Each production is of the form $A\to \alpha$ where $A\in N$ and $\alpha\in(N\cup \Sigma)^\ast$.
\end{definition}
The language generated by a context-free grammar is the set of all strings that can be obtained from $S$ by applying productions from $P$. We say that two grammars are equivalent if they generate the same language.
\begin{example}
	Let $N=\{S\},\Sigma=\{a,b\}$ and let $P$ contain two productions:
	$$
	S\to aSb, S\to ab
	$$
	Then the language generated by this grammar is $\{a^nb^n|n\ge 1\}$. E.g. $S\Rightarrow aSb\Rightarrow aaSbb\Rightarrow aaabbb$ justifies that $aaabbb$ belongs to this language.
\end{example}
This approach has a number of extensions, in particular, to graphs; one of such generalizations will be shown later. One of important features of all context-free systems is, as their name says, independence of the context: each production is applied to a nonterminal symbol without regard to its enviroment.
\subsection{Lambek calculus}\label{sec_lambek}
In this section, we provide basic definitions and examples regarding the Lambek calculus; concepts behind these definitions form the basis for the idea of the hypergraph Lambek calculus.

Let us fix a countable set $Pr=\{p_i\}_{i=1}^\infty$ of \emph{primitive types}. 
\begin{definition}
	The set $Tp(\mathrm{L})$ of types in the Lambek calculus is the least set such that:
	\begin{itemize}
		\item $Pr\subseteq Tp(\mathrm{L})$;
		\item If $A,B\in Tp(\mathrm{L})$ are types, then $(B\backslash A), (A/B), (A\cdot B)$ are also types, i.e. belong to $Tp(\mathrm{L})$ (brackets are often omitted).
	\end{itemize}
\end{definition}
\begin{definition}
	A sequent is of the form $T_1,\dots,T_n\to T$ where $T_i, T$ are types ($n>0$). $T_1,\dots,T_n$ is called an antecedent, and $T$ is called a succedent.
\end{definition}
When talking about the Lambek calculus small letters $p, q, \dots$ and strings composed of them (e.g. $np, cp$) range over primitive types. Capital letters $A,B,\dots$ range over types. Capital Greek letters $\Gamma,\Delta,\dots$ range over finite (possibly empty) sequences of types. Sequents thus can be represented as $\Gamma\to A$, where $\Gamma$ is nonempty. 

The Lambek calculus $\mathrm{L}$ is a logical system with one axiom and six inference rules:

$$\infer[]{p\to p}{}
$$
$$
\infer[(\backslash\to)]{\Gamma, \Pi, A \backslash B, \Delta \to C}{\Pi \to A & \Gamma, B, \Delta \to C}
\qquad
\infer[(\to\backslash)]{\Pi \to A \backslash B}{A, \Pi \to B}
\qquad
\infer[(\cdot\to)]{\Gamma, A \cdot B, \Delta \to C}{\Gamma, A, B, \Delta \to C}
$$
$$
\infer[(/\to)]{\Gamma, B / A, \Pi, \Delta \to C}{\Pi \to A & \Gamma, B, \Delta \to C}
\qquad
\infer[(\to/)]{\Pi \to B / A}{\Pi, A \to B}
\qquad
\infer[(\to\cdot)]{\Gamma, \Delta \to A \cdot B}{\Gamma \to A & \Delta \to B}
$$
A sequent $\Gamma\to A$ is derivable ($\mathrm{L}\vdash \Gamma\to A$) iff it can be obtained from axioms applying rules. A corresponding sequence of rule applications is called a derivation.
\begin{example}
	This is a derivation in $\mathrm{L}$:
	$$
	\infer[(\to\backslash)]{q\backslash r\to (p\cdot q)\backslash (p\cdot r)}{
	\infer[(\cdot\to)]{p\cdot q,q\backslash r\to (p\cdot r)}{
		\infer[(\backslash\to)]{p, q,q\backslash r\to (p\cdot r)}{
			\infer[]{q\to q}{} & \infer[(\to\cdot)]{p, r\to (p\cdot r)}{
				\infer[]{p\to p}{} & \infer[]{r\to r}{}
			}	
		}
	}
	}
	$$
\end{example}
The Lambek calculus itself is of interest as a calculus: questions of structural properties, of derivability problem complexity, of models arise. Besides, it forms a basis for a class of type-logical grammars:
\begin{definition}
	A Lambek grammar is a tuple $G=\langle \Sigma, S, \triangleright\rangle$ where $\Sigma$ is a finite set (alphabet), $S\in Tp_L$ is a distinguished type, and $\triangleright\subseteq\Sigma\times Tp(\mathrm{L})$ is a finite binary relation, i.e. it assigns a finite number of types to each symbol in the aplhabet. The language $L(G)$ generated by $G$ is the set of all nonempty strings $a_1\dots a_n$ over $\Sigma$ for which there are types $T_1,\dots,T_n$ such that $a_i\triangleright T_i$, and $\mathrm{L}\vdash T_1,\dots,T_n\to S$.
\end{definition}
\begin{example}\label{ex_lan_anbn}
	Consider the following Lambek grammar $\langle\{a,b\}, s, \triangleright\rangle$:
	\begin{enumerate}
		\item $s\in Pr$;
		\item The relation $\triangleright$ is as follows: $a\triangleright s/p$, $b\triangleright p$, $b\triangleright s\backslash p$.
	\end{enumerate}
	This grammar generates the language $\{a^nb^n|n\ge 1\}$. E.g. the string $aaabbb$ corresponds to the following derivable sequence of types: $s/p,s/p,s/p,p,s\backslash p,s\backslash p\to s$.
\end{example}
The most famous result regarding Lambek grammars is the fact that classes of languages generated by Lambek grammars and by context-free grammars without the empty word coincide. This nontrivial result shows that these two approaches are in some sense equivalent. We will discuss this result later.

From the practical point of view, Lambek grammars can serve to describe natural languages; for instance, the sentence {\it Tim thinks Helen is smart} corresponds to the derivable sequent $np,  (np\backslash s)/s, np, (np\backslash S)/adj,adj\to s$. However, it is known that Lambek grammars generate only context-free languages while natural language phenomena include non-context-free ones.

\section{Preliminaries: hyperedge replacement grammars}\label{sec_prelim}
Now we turn to graph grammars. They are developed as an extension of context-free grammars that is used to produce graphs and hypergraphs nstead of just strings. This task is of interest since graph structures are widely used in programming and in linguistics, and one expects that certain graph languages can be described in the same way as string languages. We focus on a particular approach called hyperedge replacement grammar since it is very close to context-free grammars in sense of definitions. Below all required definitions are introduced. They are taken from \cite{Drewes}. Note that we use a slightly different notation from that in \cite{Drewes}.
\subsection{Hypergraphs, Sub-hypergraphs}\label{hyp_def}
Let $C$ be some fixed set of labels for whom the function $type: C\to \mathbb{N}$ is considered. 
\begin{definition}\label{def_hypergraph}
	\emph{A hypergraph $G$ over $C$} is a tuple $G=\langle V, E, att, lab, ext \rangle$ where $V$ is the set of \emph{nodes}, $E$ is the set of \emph{hyperedges}, $att: E\to V^\circledast$ assigns a string (i.e. an ordered set) of \emph{attachment} nodes to each edge, $lab: E \to C$ labels each edge by some element of $C$ in such a way that $type(lab(e))=|att(e)|$ whenever $e\in E$, and $ext\in V^\circledast$ is a string of \emph{external} nodes. 
	
	Components of a hypergraph $G$ are denoted by $V_G, E_G, att_G, lab_G, ext_G$ resp.
\end{definition}
According to this definition in hypergraphs labels are put on hyperedges; a particular label appears on hyperedges with the same number of attachment nodes.

In the remainder of the paper, hypergraphs are simply called graphs, and hyperedges are simply called edges. Usual graphs (with type 2 hyperedges) are called 2-graphs. The set of all hypergraphs with labels from $C$ is denoted by $\mathcal{H}(C)$. Graphs are usually named by letters $G$ and $H$. 

In drawings of graphs black dots correspond to nodes, labeled squares correspond to edges, $att$ is represented with numbered lines, and external nodes are depicted by numbers in brackets. If an edge has exactly two attachment nodes, it can be denoted by an arrow (which goes from the first attachment node to the second one). 
\begin{definition}\label{type}
	The function $type$ (or $type_G$ to be exact) returns the number of nodes attached to some edge in a graph $G$: $type_G(e):=|att_G(e)|$.
	If $G$ is a graph, then $type(G):=|ext_G|$.
\end{definition}
\begin{example}
	The following picture represents a graph $G$:
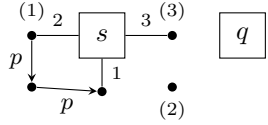
	\begin{center}
		{\tikz[baseline=.1ex]{
				\node[] (R) {};
				\node[node,above right=2mm and 0mm of R,label=above:{\scriptsize $(1)$}] (N1) {};
				\node[node,below=5.5mm of N1] (N2) {};
				\node[hyperedge,right=5.5mm of N1] (E) {$s$};
				\node[node,below=3.7mm of E] (N3) {};
				\node[node,right=5.5mm of E,label=above:{\scriptsize $(3)$}] (N4) {};
				\node[node,below=5.5mm of N4,label=below:{\scriptsize $(2)$}] (N5) {};
				\node[hyperedge,right=5.5mm of N4] (E2) {$q$};

				\draw[>=stealth,->,black] (N1) -- node[left] {\small $p$} (N2);
				\draw[>=stealth,->,black] (N2) -- node[below] {\small $p$} (N3);
				\draw[-,black] (N3) -- node[right] {\scriptsize 1} (E);
				\draw[-,black] (N1) -- node[above] {\scriptsize 2} (E);
				\draw[-,black] (N4) -- node[above] {\scriptsize 3} (E);
		}}
	\end{center}
	Here $type(p)=2, type(q)=0, type(s)=3$; $type(G)=3$.
\end{example}
\begin{definition}
	A sub-hypergraph (or just subgraph) $H$ of a graph $G$ is a hypergraph such that $V_H\subseteq V_G$, $E_H\subseteq E_G$, and for all $e\in E_H$ $att_H(e)=att_G(e)$, $lab_H(e)=lab_G(e)$.
\end{definition}
%Note that each subgraph $H$ of the graph $G$ can be uniquely defined by $E_H$ and $ext_H$ (particularly, $V_H$ equals the set of all the attachment nodes of edges from $E_H$). Let us then denote $H$ by $\langle E_H, ext_H \rangle_G$ or even $\langle E_H \rangle_G$ if the specific form of $ext_H$ is not important.
\begin{definition}
	If $H=\langle \{v_i\}_{i=1}^n,\{e_0\},att,lab,v_1\dots v_n\rangle$, $att(e_0)=v_1\dots v_n$ and $lab(e_0)=a$, then $H$ is called \emph{a handle}. It is denoted by $\circledcirc(a)$.
\end{definition}
\begin{definition}
	\emph{An isomorphism} between graphs $G$ and $H$ is a pair of bijective functions $\mathcal{E}: E_G\to E_H$, $\mathcal{V}: V_G\to V_H$ such that $att_H\circ\mathcal{E}=\mathcal{V}\circ att_G$, $lab_G=lab_H\circ\mathcal{E}$, $\mathcal{V}(ext_G)=ext_H$. 
\end{definition}
In this work, we do not distinguish between isomorphic graphs.

Strings can be considered as graphs with the string structure. This is formalized in
\begin{definition}\label{def_str_gr}
	A string graph induced by a string $w=a_1\dots a_n, n>0$ is a graph of the form $\langle \{v_i\}_{i=0}^n,\{e_i\}_{i=1}^n,att,lab,v_0v_n \rangle$ where $att(e_i)=v_{i-1}v_i$, $lab(e_i)=a_i$. It is denoted by $w^\bullet$.
\end{definition}

We additionally introduce the following definitions and notations:
\begin{definition}
	Let $H\in \mathcal{H}(C)$ be a graph, and let $f:E_H\to C$ be a relabeling function. Then $f(H)=\langle V_H, E_H, att_H, lab_{f(H)}, ext_H\rangle$ where $lab_{f(H)}(e)=f(e)$ for all $e$ in $E_H$. It is required that $type(lab_H(e))=type(f(e))$ for $e\in E_H$.
\end{definition}
If one wants to relabel only one edge $e_0$ within $H$ with a label $a$, then the result can be denoted by $H[e_0:=a]$
\subsection{Operations on Graphs}
In graph formalisms certain graph transformation are in use. To generalize the Lambek calculus we present the following operation called compression.
\subsubsection{Compression.}\label{com}
Let $G$ be a graph, and let $H$ be a subgraph of $G$. Compression of $H$ into an $a$-labeled edge within $G$ is a procedure of transformation of $G$, which can be done under the following conditions:
\begin{enumerate}[label=(\alph*)]
	\item\label{com_retain_att} For each $v\in V_H$, if $v$ is attached to some edge $e\in E_G\setminus E_H$ (i.e. $v\in [att(e)]$), then $v$ has to be external in $H$ ($v\in [ext_H]$).
	\item\label{com_retain_ext} If $v\in V_H$ is external in $G$, then it is external in $H$ ($[ext_G]\cap V_H \subseteq [ext_H]$).
	\item $type(H)=type(a)$.
\end{enumerate}
Then the procedure is the following:
\begin{enumerate}
	\item Remove all nodes of $V_H$ except for those of $ext_H$ from $V_G$;
	\item Remove $E_H$ from $E_G$;
	\item Add a new edge $\widetilde{e}$;
	\item Set $att(\widetilde{e})=ext_H$, $lab(\widetilde{e})=a$.
\end{enumerate}
Let $G\llbracket a/H\rrbracket$ (or $G\llbracket a,\widetilde{e}/H\rrbracket$) denote the resulting graph.

Formally, $G\llbracket a/H\rrbracket=$ $\langle V^\prime, E^\prime, att^\prime, lab^\prime, ext_G\rangle$, where $V^\prime=V_G\setminus (V_H\setminus ext_H)$, $E^\prime=\{\widetilde{e}\}\cup (E_G\setminus E_H)$, $att^\prime(e)=att_G(e), lab^\prime(e)=lab_G(e)$ for $e\not=\widetilde{e}$, and $att^\prime(\widetilde{e})=ext_H$, $lab^\prime(\widetilde{e})=a$.
\subsubsection{Replacement.}\label{sec_repl}
This procedure is defined in \cite{Drewes} and it plays a fundamental role in hyperedge replacement grammars. The replacement of an edge $e_0$ in $G$ with a graph $H$ can be done if $type(e_0)=type(H)$ as follows:
\begin{enumerate}
	\item Remove $e_0$;
	\item Insert an isomorphic copy of $H$ (namely, $H$ and $G$ have to consist of disjoint sets of nodes and edges);
	\item For each $i$, fuse the $i$-th external node of $H$ with the $i$-th attachement node of $e_0$.
\end{enumerate}
To be more precise, the set of edges in the resulting graph is $(E_G\setminus\{e_0\})\cup E_H$, and the set of nodes is $V_G\cup (V_H\setminus ext_H)$. The result is denoted by $G[H/e_0]$.

It is known that if several edges of a graph are replaced by another graphs, then the result does not depend on order of replacements; moreover the result is not changed if replacements are done simultaneously. In \cite{Drewes} this is called sequentialization and parallelization. The following notation is in use: if $e_1,\dots,e_k$ are distinct edges of a graph $H$ and they are simultaneously replaced by graphs $H_1,\dots,H_k$ resp. (this means that $type(H_i)=type(e_i)$), then the result is denoted $H[H_1/e_1,\dots,H_k/e_k]$.

Note that compression and replacement are opposite to each other. It is stated in
\begin{proposition}\label{opposite_com_rep}
	Compression and replacement are opposite: 
	\begin{enumerate}
		\item $G\llbracket a,e/H\rrbracket[H/e]\equiv G$ (for a subgraph $H$ of $G$ satisfying conditions \ref{com_retain_att} and \ref{com_retain_ext}; $a$ is an arbitrary label);
		\item $G[H/e]\llbracket a,e/H\rrbracket\equiv G$ (provided $e: type(e)=type(H)$ and $a=lab_G(e)$).
	\end{enumerate}
\end{proposition}
Notation used for compression and replacement also reflects that they are opposite.

\subsection{Hyperedge Replacement Grammars}
\begin{definition}
	A \emph{hyperedge replacement grammar (HRG)} is a tuple $HGr=\langle N, \Sigma, P, S\rangle$, where $N$ is a finite alphabet of nonterminal symbols, $\Sigma$ is a finite alphabet of terminal symbols ($N\cap\Sigma=\emptyset$), $P$ is a set of productions, and $S\in N$. Each production is of the form $A\to H$ where $A\in N$, $H\in\mathcal{H}(N\cup \Sigma)$ and $type(A)=type(H)$.
\end{definition}

In contrast to graphs, particular graph grammars are denoted by letters combinations like $HGr$.

Edges labeled by terminal (nonterminal) symbols are called \emph{terminal (nonterminal) edges}. 

One observes that this definition is very close to the definition of context-free grammars: a production replaces a nonterminal symbol by a graph labeled by terminal and nonterminal symbols. The only difference is that we additionally control types of involved objects.

If $G$ is a graph, $e_0\in E_G$, $lab(e_0)=A$ and $A\to H\in P$, then $G$ directly derives $G[H/e_0]$ (denote $G\Rightarrow G[H/e_0]$). The transitive reflexive closure of $\Rightarrow$ is denoted by $\overset{\ast}{\Rightarrow}$. If $G \overset{\ast}{\Rightarrow} H$, then $G$ is said to derive $H$. The corresponding sequence of production applications is called a derivation. We write $G \overset{k}{\Rightarrow} H$ if $G$ derives $H$ in $k$ steps.
\begin{definition}
	The \emph{language generated by an HRG $\langle N, \Sigma, P, S\rangle$} is the set of graphs $H\in \mathcal{H}(\Sigma)$ such that $\circledcirc(S)\overset{\ast}{\Rightarrow} H$. A language generated by an HRG is also called a (hyper)graph context-free language (denote HCFL).\\
	Two grammars are said to be \emph{equivalent} iff they generate the same language.
\end{definition}
Extending properties of context-free grammars one obtains in particular the following results for HRGs: the context-freeness lemma, the pumping lemma, the Parikh theorem. They can be found in \cite{Drewes}. This shows that HRGs are closely related to context-free grammars; proofs of the aforementioned results directly generalize corresponding ones for strings.

\section{Hypergraph Lambek Calculus}
As we emphasized above, HRGs naturally generalize context-free grammars. On the other hand, there is the Lambek calculus whose grammars are equivalent to context-free grammars while it works in a competely different way than context-free grammars do. Then a natural question arises: is it possible to generalize the Lambek calculus to graphs in a natural way? Analogously, we expect that such a generalization would preserve fundamental features of the Lambek calculus, e.g. the cut elimination, the subformula property, existence of partial semigroup models (including L- and R- models). We would also like to define hypergraph Lambek grammars, and we expect that they have to be equivalent to HRGs. Looking ahead, the last expectation was proved wrong.

In this section we introduce the hypergraph Lambek calculus: we define types, sequents, axioms and rules of this formalism. Our goal is to introduce logic on graphs. This means \emph{literally} a logic on graphs: while such calculi as the Lambek calculus, the first-order predicate calculus, the propositional calculus deal with objects of string nature, we desire a new formalism to work with objects of graph nature. Thus sequents are supposed to be graphs rather than strings and types are assumed to label edges of graphs. Definitions presented below meet these requirements.
\subsection{Types and sequents}
We fix a countable set $Pr$ of primitive types and a function $type:Pr\to\mathbb{N}$ such that for each $n\in\mathbb{N}$ there are infinitely many $p\in Pr$ for which $type(p)=n$. Types are constructed from primitive types using division and multiplication. Simultaneously, we define the function $type$ on types (sorry for the tautology): it is obligatory since we are going to label edges by types. 

Let us fix some symbol $\$$ that is not included in all the sets considered. {\bf NB!} This symbol is allowed to label edges with different number of attachment nodes. To be consistent with Definition \ref{def_hypergraph} one can assume that there are countably many symbols $\$_n,n\ge 0$ such that $type(\$_n)=n$.
\begin{definition}
The set $Tp(\mathrm{HL})$ of types is defined inductively as the least set satisfying the following conditions:
	\begin{enumerate}
		\item $Pr\subseteq Tp(\mathrm{HL})$.
		\item Let $N$ (``numerator'') be in $Tp(\mathrm{HL})$. Let $D$ (``denominator'') be a graph such that exactly one of its edges (call it $e_0$) is labeled by $\$$, and other edges (possibly, there are none of them) are labeled by elements of $Tp(\mathrm{HL})$; let also $type(N)=type(D)$. Then $T=\div(N/D)$ also belongs to $Tp(\mathrm{HL})$, and $type(T):=type_D(e_0)$.
		\item Let $M$ be a graph such that all its edges are labeled by types from $Tp(\mathrm{HL})$ (possibly, there are no edges). Then $T=\times(M)$ is a type as well (i.e. it belongs to $Tp(\mathrm{HL})$), and $type(T)=type(M)$.
	\end{enumerate}
\end{definition}
For a type $A$ we can define the set of its subtypes in a natural way considering $A$ as a term.
\begin{example}\label{ex_types}
	The following structures are types:
	\begin{itemize}
		\item $A_1=\div\left(q\middle/\mbox{
			{\tikz[baseline=1ex]{
					\node[hyperedge] (E1) {$s$};
					\node[node,right=5mm of E1] (N1) {};
					\node[hyperedge,right=5mm of N1] (E2) {$\$$};
					\node[node,right=5mm of E2] (N2) {};
					\node[hyperedge,right=5mm of N2] (E3) {$r$};
					\draw[-,black] (E1) -- node[above] {\small 1} (N1);
					\draw[-,black] (N1) -- node[above] {\small 1} (E2);
					\draw[-,black] (E2) -- node[above] {\small 2} (N2);
					\draw[-,black] (N2) -- node[above] {\small 1} (E3);
			}}
		}\right)$;
		\item $A_2=\div\left(t\middle/\mbox{
			{\tikz[baseline=1ex]{
					\node[node,label=above:{\small $(1)$}] (N1) {};
					\node[hyperedge, right=5mm of N1] (E1) {$r$};
					\node[node, right=5mm of E1,label=above:{\small $(2)$}] (N2) {};
					\node[hyperedge,right=5mm of N2] (E2) {$\$$};
					\node[node,right=5mm of E2] (N3) {};
					\node[hyperedge,right=5mm of N3] (E3) {$s$};
					\draw[-,black] (N1) -- node[above] {\small 1} (E1);
					\draw[-,black] (N2) -- node[above] {\small 2} (E2);
					\draw[-,black] (E2) -- node[above] {\small 1} (N3);
					\draw[-,black] (N3) -- node[above] {\small 1} (E3);
			}}
		}\right)$;
		\item $A_3=\div\left(q\middle/\mbox{
			{\tikz[baseline=1ex]{
					\node[node] (N1) {};
					\node[hyperedge, right=5mm of N1] (E1) {$\$$};
					\node[node, above right= 4mm and 9mm of E1] (N2) {};
					\node[node,below right=4mm and 9mm of E1] (N3) {};
					\node[hyperedge,below=4mm of N2] (E2) {$t$};
					\draw[-,black] (N1) -- node[above] {\small 3} (E1);
					\draw[-,black] (E1) -- node[above] {\small 1} (N2);
					\draw[-,black] (E1) -- node[below] {\small 2} (N3);
					\draw[-,black] (N2) -- node[right] {\small 2} (E2);
					\draw[-,black] (N3) -- node[right] {\small 1} (E2);
			}}
		}\right)$;
	\item $A_4=\times\left(\mbox{
		{\tikz[baseline=.1ex]{
				\node[hyperedge] (E1) {$A_1$};
				\node[node,above right=5mm and 4mm of E1,label=below:{\small $(1)$}] (N1) {};
				\node[node,below right=5mm and 4mm of E1,label=above:{\small $(2)$}] (N2) {};
				\node[hyperedge,below right=5mm and 4mm of N1] (E2) {$p$};
				\draw[-,black] (E1) to[bend left] node[left] {\small 1} (N1);
				\draw[-,black] (E1) to[bend right] node[left] {\small 2} (N2);
				\draw[-,black] (E2) to[bend right] node[right] {\small 1} (N1);
				\draw[-,black] (E2) to[bend left] node[right] {\small 2} (N2);
		}}
	}\right)$.
	\end{itemize}
	Here $type(p)=2, type(q)=0, type(r)=type(s)=1, type(t)=2$; $type(A_1)=type(A_2)=2,type(A_3)=3,type(A_4)=2$. Note that the denominator of $A_2$ is not connected --- this is allowed.
\end{example}
\begin{example}\label{ex_subtypes}
	$A_4$ from the previous example has 6 subtypes: $A_4,p,A_1,q,s,r$.
\end{example}
Sequents in the graph case are defined similarly to sequents in the string case with the difference that antecedents are graphs instead of strings.

\begin{definition}
	\emph{A graph sequent} is a structure of the form $H\to A$, where $A\in Tp(\mathrm{HL})$ is a type, $H\in\mathcal{H}(Tp(\mathrm{HL}))$ is a graph labeled by types and $type(H)=type(A)$. $H$ is called the antecedent of the sequent, and $A$ is called the succedent of the sequent.
\end{definition}
Let $\mathcal{T}$ be a subset of $Tp(\mathrm{HL})$. We say that $H\to A$ is over $\mathcal{T}$ if $G\in\mathcal{H}(\mathcal{T})$ and $A\in\mathcal{T}$.
\begin{example}\label{ex_sequent}
	The following structure is a sequent:
	\\
	$$\mbox{	
		{\tikz[baseline=.1ex]{
				\node[hyperedge] (E1) {$A_2$};
				\node[node,above right=3mm and 5mm of E1,label=above:{\small $(1)$}] (N1) {};
				\node[node,below right=3mm and 5mm of E1] (N2) {};
				\node[hyperedge,right=5mm of N1] (E2) {$p$};
				\node[hyperedge,right=5mm of N2] (E3) {$A_3$};
				\node[node,above right=3mm and 5mm of E3,label=right:{\small $(2)$}] (N3) {};
				\node[node,right=5mm of E3] (N4) {};
				\draw[-,black] (E1) -- node[above] {\small 1} (N1);
				\draw[-,black] (E1) -- node[below] {\small 2} (N2);
				\draw[-,black] (N1) -- node[above] {\small 1} (E2);
				\draw[-,black] (E2) -- node[above] {\small 2} (N3);
				\draw[-,black] (N2) -- node[below] {\small 1} (E3);
				\draw[-,black] (E3) -- node[above] {\small 2} (N3);
				\draw[-,black] (E3) -- node[below] {\small 3} (N4);
	}}}\quad\to\quad\times\left(\mbox{
		{\tikz[baseline=.1ex]{
				\node[hyperedge] (E1) {$A_1$};
				\node[node,above right=5mm and 4mm of E1,label=below:{\small $(1)$}] (N1) {};
				\node[node,below right=5mm and 4mm of E1,label=above:{\small $(2)$}] (N2) {};
				\node[hyperedge,below right=5mm and 4mm of N1] (E2) {$p$};
				\draw[-,black] (E1) to[bend left] node[left] {\small 1} (N1);
				\draw[-,black] (E1) to[bend right] node[left] {\small 2} (N2);
				\draw[-,black] (E2) to[bend right] node[right] {\small 1} (N1);
				\draw[-,black] (E2) to[bend left] node[right] {\small 2} (N2);
		}}
	}\right)$$
	\\
	Here $A_1,A_2,A_3$ are from Example \ref{ex_types}.
\end{example}
\subsection{Axiom and rules}\label{sec_axioms_rules}
The hypergraph Lambek calculus (denoted $\mathrm{HL}$) we introduce here is a logical system that defines what graph sequents are derivable (=provable) in sense of axioms and rules. $\mathrm{HL}$ includes one axiom and four rules. They are introduced below.
\subsubsection{Axiom.} $\circledcirc(p)\to p,\quad p\in Pr$.
\subsubsection{Rule $(\div\to)$.}\label{subsec_div_to}
Let $\div(N/D)$ be a type and let $E_D=\{d_0,d_1,\dots,d_k\}$ where $lab(d_0)=\$$. Assume that $D$ is a concrete graph such that its nodes and edges are distinct from those in other involved graphs. Let $H\to A$ be a graph sequent and let $e\in E_H$ be labeled by $N$. Let finally $H_1,\dots,H_k$ be graphs labeled by types. Then the rule $(\div\to)$ is the following:
$$
\infer[(\div\to)]{H[D/e][d_0:=\div(N/D)][H_1/d_1,\dots,H_k/d_k]\to A}{H\to A & H_1\to lab(d_1) &\dots & H_k\to lab(d_k)}
$$
This rule is technically the hardest one; it explains how a type with division can appear in an antecedent. It can be also considered from bottom to top: there is a type $\div(N/D)$ in an antecedent; it is ``overlaid'' on subgraphs $H_1,\dots,H_k$; then some kind of reduction of $D$ and of $H_1,\dots,H_k$ happens, and the whole sequent splits into $(k+1)$ new ones. 
\subsubsection{Rule $(\to\div)$.}
Let $H$ be a graph, and let $F$ be its subgraph; let $A$ be a type. The rule is of the form
$$
\infer[(\to\div)]{F\to \div(A/H\llbracket \$/F\rrbracket)}{H\to A}
$$
This means that if one obtains a sequent $H\to N$, then he can compress some its subgraph into a single \$-labeled edge and to put this new graph in the denominator of a succedent; $F$ then becomes an antecedent.
\subsubsection{Rule $(\times\to)$.}
Let $H$ be a graph, and let $F$ be its subgraph; let $A$ be a type.
$$
\infer[(\times\to)]{H\llbracket \times(F)/F\rrbracket\to A}{H\to A}
$$
That is, a subgraph $F$ in the antecedent can be compressed into a single $\times(F)$-labeled edge.
\subsubsection{Rule $(\to\times)$.}
Let $\times(M)$ be a type and let $E_M=\{m_1,\dots,m_l\}$. Assume that $M$ is a concrete graph such that its nodes and edges are distinct from those in other involved graphs. Let $H_1,\dots,H_l$ be graphs. Then
$$
\infer[(\to\times)]{M[H_1/m_1,\dots,H_l/m_l]\to\times(M)}{H_1\to lab(m_1) & \dots & H_l\to lab(m_l)}
$$
This rule is quite intuitive: several sequents can be combined into a single one via some graph structure $M$.
\begin{definition}
	A graph sequent $H\to A$ is derivable in $\mathrm{HL}$ ($\mathrm{HL}\vdash H\to A$) if it can be obtained from axioms using rules of $\mathrm{HL}$. A corresponding sequence of rule applications is called a derivation and its graphical representation is called a derivation tree.
\end{definition}
\begin{remark}
	If a graph $M$ in the rule $(\to\times)$ does not have edges, then there are zero premises in this rule ($l=0$); hence formally $M\to\times(M)$ is derivable, and this sequent can be considered as an axiom (though this looks strange).
\end{remark}
\subsection{Examples}
We proceed with some examples that illustrate how these rules work.
\begin{example}\label{ex_simple_der}
Firstly, we provide four simple examples of rule applications with $T_1,\dots,T_5$ being some types and $A_1$ being from Example \ref{ex_types}.
\\
	\begin{center}
	\begin{tikzpicture}
		\node[node] (N1) {};
		\node[hyperedge,above=4mm of N1] (E1) {$T_1$};
		\node[hyperedge,right=4mm of N1] (E2) {$T_2$};
		\node[node,right=4mm of E2] (N2) {};
		\node[hyperedge,right=4mm of N2] (E3) {$A_1$};
		\node[node,right=8mm of E3] (N3) {};
		\node[hyperedge,above left=4.2mm and 3.5mm of N3] (E4) {$T_3$};
		\node[hyperedge,above=4mm of N3] (E5) {$T_4$};
		\node[above right=1mm and 1.5mm of N3] (To) {$\to T_5$};
		
		\draw[-,black] (N1) -- node[left] {\small 1} (E1);
		\draw[-,black] (N1) -- node[below] {\small 1} (E2);
		\draw[-,black] (E2) -- node[below] {\small 2} (N2);
		\draw[-,black] (N2) -- node[below] {\small 1} (E3);
		\draw[-,black] (E3) -- node[below] {\small 2} (N3);
		\draw[-,black] (N3) -- node[left] {\small 1} (E4);
		\draw[-,black] (N3) -- node[right] {\small 1} (E5);

		\node[hyperedge,above=9mm of E1] (E31) {$q$};
		\node[right=0.5mm of E31] (To4) {$\to T_5$};
		
		\node[right=18.5mm of N1] (R){};
		\node[node, above=18mm of R] (N21) {};
		\node[hyperedge,above=4mm of N21] (E21) {$T_1$};
		\node[hyperedge,right=4mm of N21] (E22) {$T_2$};
		\node[node,above=4mm of E22,label=above:{\small $(1)$}] (N22) {};
		\node[above right=-2mm and 0mm of E22] (To2) {$\to s$};
		
		\draw[-,black] (N21) -- node[left] {\small 1} (E21);
		\draw[-,black] (N21) -- node[below] {\small 1} (E22);
		\draw[-,black] (E22) -- node[right] {\small 2} (N22);

		\node[node,right=19mm of E22,label=below:{\small $(1)$}] (N23) {};
		\node[hyperedge,above left=4mm and 0mm of N23] (E24) {$T_3$};
		\node[hyperedge,above right=4mm and 0mm of N23] (E25) {$T_4$};
		\node[above right=0.5mm and 6mm of N23] (To3) {$\to r$};
		
		\draw[-,black] (N23) -- node[left] {\small 1} (E24);
		\draw[-,black] (N23) -- node[right] {\small 1} (E25);
		
		\node[above left = 11mm and 3mm of N1] (l1) {};
		\node[above right = 11mm and 30mm of N3] (l2) {};
		\draw[-,black] (l1) -- (l2);
		\node[right=2mm of l2] {$(\div\to)$};
	\end{tikzpicture}
\vspace{5mm}
\\
	\begin{tikzpicture}
		\node[node] (N1) {};
		\node[hyperedge,above right=3mm and 1.5mm of N1] (E1) {$T_1$};
		\node[node,right=3.5mm of E1] (N2) {};
		\node[hyperedge,right=3.5mm of N2] (E2) {$T_2$};
		\node[node,below right=3mm and 1.5mm of E2,label=below:{\small $(1)$}] (N3) {};
		\node[above right=1mm and 1mm of N3] (To) {$\to\div$};
		\node[right=-2.5mm of To] (Left) {\huge (};
		\node[right=-2mm of Left] (T3) {$T_3$};
		\node[right=-3mm of T3] (Div) {\huge /};
		\node[right=20mm of Left] (Right) {\huge )};
		
		\draw[-,black] (N1) -- node[left] {\small 1} (E1);
		\draw[-,black] (E1) -- node[above] {\small 2} (N2);
		\draw[-,black] (N2) -- node[above] {\small 2} (E2);
		\draw[-,black] (E2) -- node[right] {\small 1} (N3);
		
		\node[hyperedge,below right=-5mm and -2mm of Div] (E21) {$\$$};
		\node[node,above right=3mm and 1mm of E21] (N21) {};
		\node[hyperedge,below right=3mm and 1mm of N21] (E22) {$T_4$};
		
		\draw[-,black] (E21) -- node[left] {\small 1} (N21);
		\draw[-,black] (N21) -- node[right] {\small 1} (E22);
		
		\node[above left = 10mm and 1mm of N1] (l1) {};
		\node[above right = 10mm and 37mm of N3] (l2) {};
		\draw[-,black] (l1) -- (l2);
		
		\node[node, above=12mm of N1] (N21) {};
		\node[hyperedge,above right=2mm and 4mm of N21] (E21) {$T_1$};
		\node[node,above right=2mm and 5mm of E21] (N22) {};
		\node[hyperedge,below right=2mm and 5mm  of N22] (E22) {$T_2$};
		\node[node,below right=2mm and 5mm of E22] (N23) {};
		\node[hyperedge,above right=2mm and 5mm of N23] (E23) {$T_4$};
		
		\node[right=1mm of E23] (To2) {$\to T_3$};
		
		\draw[-,black] (N21) -- node[above] {\small 1} (E21);
		\draw[-,black] (E21) -- node[above] {\small 2} (N22);
		\draw[-,black] (N22) -- node[above] {\small 2} (E22);
		\draw[-,black] (E22) -- node[above] {\small 1} (N23);
		\draw[-,black] (N23) -- node[above] {\small 1} (E23);
		
		\node[right=2mm of l2] {$(\to\div)$};
	\end{tikzpicture}
	\vspace{5mm}
	\\
		\begin{tikzpicture}
	\node[node] (N1) {};
	\node[hyperedge,above=4mm of N1] (E1) {$T_1$};
	\node[rectangle, minimum width = 32.3mm, minimum height = 14mm, draw, inner sep = 0mm, color = black,right=4mm of N1] (E2) {};
	\node[node,right=4mm of E2,label=right:{\small $(1)$}] (N2) {};
	\node[hyperedge,above=4mm of N2] (E3) {$T_4$};
	\node[node,above left=4mm and 1mm of E3] (N3) {};
	\node[node,above right=4mm and 1mm of E3] (N4) {};
	\node[above right=1mm and 6mm of N2] (To) {$\to T_5$};
	
	\node[right=4mm of N1] (Times) {$\times$};
	\node[right=-3mm of Times] (Left) {\Huge $($};
	\node[right=22mm of Times] (Right) {\Huge $)$};
	
	\node[hyperedge,above right=-6mm and -1mm of Left] (T2) {$T_3$};
	\node[node,below=3.2mm of T2,label=right:{\small $(1)$}](V2) {};
	
	\draw[-,black] (T2) -- node[left] {\small 1} (V2);
	
	\node[hyperedge,right=6mm of T2] (T3) {$T_2$};
	\node[node,below=3.2mm of T3,label=right:{\small $(2)$}](V3) {};
	
	\draw[-,black] (T3) -- node[left] {\small 1} (V3);

	\draw[-,black] (N1) -- node[left] {\small 1} (E1);
	\draw[-,black] (N1) -- node[above] {\small 1} (E2);
	\draw[-,black] (E2) -- node[above] {\small 2} (N2);
	\draw[-,black] (N2) -- node[right] {\small 3} (E3);
	\draw[-,black] (E3) -- node[right] {\small 1} (N3);
	\draw[-,black] (E3) -- node[right] {\small 2} (N4);
	
	\node[hyperedge,above=22mm of N1] (E21) {$T_1$};
	\node[node,right=4mm of E21] (N21) {};
	\node[hyperedge,right=4mm of N21] (E22) {$T_3$};
	\node[hyperedge,right=4mm of E22] (E23) {$T_2$};
	\node[node,right=4mm of E23,label=below:{\small $(1)$}] (N22) {};
	\node[hyperedge,right=4mm of N22] (E24) {$T_4$};
	\node[node,above right=1mm and 3mm of E24] (N23) {};
	\node[node,below right=1mm and 3mm of E24] (N24) {};
	\node[right=5mm of E24] (To2) {$\to T_5$};

	\draw[-,black] (E21) -- node[above] {\small 1} (N21);
	\draw[-,black] (N21) -- node[above] {\small 1} (E22);
	\draw[-,black] (E23) -- node[above] {\small 1} (N22);
	\draw[-,black] (N22) -- node[above] {\small 3} (E24);
	\draw[-,black] (E24) -- node[above] {\small 1} (N23);
	\draw[-,black] (E24) -- node[below] {\small 2} (N24);

	\node[above left = 17mm and 2.5mm of N1] (l1) {};
	\node[above right = 17mm and 17mm of N2] (l2) {};
	\draw[-,black] (l1) -- (l2);
	\node[right=2mm of l2] {$(\times\to)$};
	\end{tikzpicture}
	\vspace{5mm}
	\\
	\begin{tikzpicture}
		\node[node,label=below:{\small $(1)$}] (N1) {};
		\node[hyperedge,right=8mm of N1] (E1) {$T_1$};
		\node[node,right=8mm of E1,label=below:{\small $(2)$}] (N2) {};
		\node[node,above=1.7mm of E1] (N3) {};
		\node[hyperedge,above left=-1.5mm and 4mm of N3] (E2) {$T_2$};
		\node[hyperedge,above right=-1.5mm and 4mm of N3] (E3) {$T_3$};
		\node[above right=1mm and 1mm of N2] (To) {$\to\times$};
		\node[right=-2.5mm of To] (Left) {\Huge (};
		
		\draw[-,black] (N1) -- node[left] {\small 1} (E2);
		\draw[-,black] (E2) -- node[above] {\small 2} (N3);
		\draw[-,black] (N3) -- node[above] {\small 1} (E3);
		\draw[-,black] (E3) -- node[right] {\small 2} (N2);
		\draw[-,black] (N1) -- node[below] {\small 2} (E1);
		\draw[-,black] (E1) -- node[below] {\small 1} (N2);
		
		\node[node,right=12.5mm of N2, label=below:{\small $(1)$}] (N21) {};
		\node[hyperedge,right=4mm of N21] (E21) {$T_5$};
		\node[node,right=4mm of E21,label=below:{\small $(2)$}] (N22) {};
		\node[hyperedge,above=1.5mm of E21] (E22) {$T_4$};
		\node[right=16.5mm of Left] (Right) {\Huge )};
		
		\draw[-,black] (N21) -- node[above] {\small 1} (E21);
		\draw[-,black] (E21) -- node[above] {\small 2} (N22);
		\draw[-,black] (N21) to[bend left=25] node[above] {\small 1} (E22);
		\draw[-,black] (E22) to[bend left=25] node[above] {\small 2} (N22);
		
		\node[above left = 10mm and 1mm of N1] (l1) {};
		\node[node,above right = 5mm and 1mm of l1,label=below:{\small $(1)$}] (N31) {};
		\node[hyperedge,right=4mm of N31] (E31) {$T_1$};
		\node[node,right=4mm of E31,label=below:{\small $(2)$}] (N32) {};
		
		\draw[-,black] (N31) -- node[above] {\small 2} (E31);
		\draw[-,black] (E31) -- node[above] {\small 1} (N32);
		\node[right=0.5mm of N32] (To2) {$\to T_5$};
		
		\node[node,below right = 0.5mm and 6mm of To2,label=above:{\small $(1)$}] (N41) {};
		
		\node[hyperedge,above right=5mm and 2mm of N41] (E41) {$T_2$};
		\node[node,below right=5mm and 2mm of E41] (N42) {};
		\node[hyperedge,above right=5mm and 2mm of N42] (E42) {$T_3$};
		\node[node,below right=5mm and 2mm of E42,label=above:{\small $(2)$}] (N43) {};
		
		\draw[-,black] (N41) -- node[right] {\small 1} (E41);
		\draw[-,black] (E41) -- node[left] {\small 2} (N42);
		\draw[-,black] (N42) -- node[right] {\small 1} (E42);
		\draw[-,black] (E42) -- node[left] {\small 2} (N43);
		\node[above right=2.5mm and 0.5mm of N43] (To3) {$\to T_4$};
		
		\node[above right = 10mm and 13mm of N22] (l2) {};
		\draw[-,black] (l1) -- (l2);
		\node[right=2mm of l2] {$(\to\times)$};
	\end{tikzpicture}
\end{center}
\end{example}
\begin{example}\label{ex_derivation}
	The sequent from Example \ref{ex_sequent} is derivable in $\mathrm{HL}$; here is its derivation:
	$$
	\infer[(\to\times)]
	{
		\mbox{	
			{\tikz[baseline=.1ex]{
					\node[hyperedge] (E1) {$A_2$};
					\node[node,above right=3mm and 8mm of E1,label=above:{\small $(1)$}] (N1) {};
					\node[node,below right=3mm and 8mm of E1] (N2) {};
					\node[hyperedge,right=of N1] (E2) {$p$};
					\node[hyperedge,right=of N2] (E3) {$A_3$};
					\node[node,above right=3mm and 8mm of E3,label=right:{\small $(2)$}] (N3) {};
					\node[node,right=8mm of E3] (N4) {};
					\draw[-,black] (E1) -- node[above] {\small 1} (N1);
					\draw[-,black] (E1) -- node[below] {\small 2} (N2);
					\draw[-,black] (N1) -- node[above] {\small 1} (E2);
					\draw[-,black] (E2) -- node[above] {\small 2} (N3);
					\draw[-,black] (N2) -- node[below] {\small 1} (E3);
					\draw[-,black] (E3) -- node[above] {\small 2} (N3);
					\draw[-,black] (E3) -- node[below] {\small 3} (N4);
		}}}\quad\to\;\times\left(\mbox{
			{\tikz[baseline=.1ex]{
					\node[hyperedge] (E1) {$A_1$};
					\node[node,above right=5mm and 4mm of E1,label=below:{\small $(1)$}] (N1) {};
					\node[node,below right=5mm and 4mm of E1,label=above:{\small $(2)$}] (N2) {};
					\node[hyperedge,below right=5mm and 4mm of N1] (E2) {$p$};
					\draw[-,black] (E1) to[bend left] node[left] {\small 1} (N1);
					\draw[-,black] (E1) to[bend right] node[left] {\small 2} (N2);
					\draw[-,black] (E2) to[bend right] node[right] {\small 1} (N1);
					\draw[-,black] (E2) to[bend left] node[right] {\small 2} (N2);
			}}
		}\right)
	}{
		\infer[(\to\div)]
		{
			\mbox{	
				{\tikz[baseline=.1ex]{
						\node[hyperedge] (E1) {$A_2$};
						\node[node,above=5mm of E1,label=right:{\small $(1)$}] (N1) {};
						\node[node,below =5mm of E1] (N2) {};
						\node[hyperedge,right=5mm of N2] (E3) {$A_3$};
						\node[node,above=5mm of E3,label=right:{\small $(2)$}] (N3) {};
						\node[node,right=5mm of E3] (N4) {};
						\draw[-,black] (E1) -- node[left] {\small 1} (N1);
						\draw[-,black] (E1) -- node[left] {\small 2} (N2);
						\draw[-,black] (N2) -- node[below] {\small 1} (E3);
						\draw[-,black] (E3) -- node[left] {\small 2} (N3);
						\draw[-,black] (E3) -- node[below] {\small 3} (N4);
			}}}\quad\to\; \div\left(q\middle/\mbox{
				{\tikz[baseline=1ex]{
						\node[hyperedge] (E1) {$s$};
						\node[node,above=5mm of E1] (N1) {};
						\node[hyperedge,right=5mm of N1] (E2) {$\$$};
						\node[node,right=5mm of E2] (N2) {};
						\node[hyperedge,below=5mm of N2] (E3) {$r$};
						\draw[-,black] (E1) -- node[left] {\small 1} (N1);
						\draw[-,black] (N1) -- node[above] {\small 1} (E2);
						\draw[-,black] (E2) -- node[above] {\small 2} (N2);
						\draw[-,black] (N2) -- node[right] {\small 1} (E3);
				}}
			}\right)
		}
		{
			\infer[(\div\to)]{
				\mbox{	
					{\tikz[baseline=.1ex]{
							\node[node] (N2) {};
							\node[hyperedge,left=4mm and 4mm  of N2] (E2) {$A_3$};
							\node[hyperedge,right=4mm and 4mm of N2] (E3) {$A_2$};
							\node[node,above right=4mm and 4mm of E2] (N3) {};
							\node[node,left=4mm of E2] (N1) {};
							\node[hyperedge,left=4mm of N1] (E1) {$r$};
							\node[node,above right=4mm of E3] (N4) {};
							\node[hyperedge,below right=4mm of N4] (E4) {$s$};
							\node[right=1.5mm of E4] (Q) {$\;\to\; q$};
							\draw[-,black] (E1) -- node[above] {\small 1} (N1);
							\draw[-,black] (N1) -- node[above] {\small 2} (E2);
							\draw[-,black] (E2) -- node[above] {\small 3} (N3);
							\draw[-,black] (E2) -- node[above] {\small 1} (N2);
							\draw[-,black] (N2) -- node[above]{\small 2} (E3);
							\draw[-,black] (E3) -- node[above] {\small 1} (N4);
							\draw[-,black] (N4) -- node[above] {\small 1} (E4);
				}}}
			}
			{
				\infer{\circledcirc(r)\to r}{}
				&
				\infer{\circledcirc(s)\to s}{}
				&
				\quad\infer[(\div\to)]
				{
					\mbox{	
						{\tikz[baseline=.1ex]{
								\node[hyperedge] (E1) {$t$};
								\node[node,above =4mm of E1] (N1) {};
								\node[node,below =4mm of E1] (N2) {};
								\node[hyperedge,below right=4mm and 8mm of N1] (E2) {$A_3$};
								\node[node,right=7mm of E2] (N3) {};
								\draw[-,black] (E1) -- node[left] {\small 1} (N1);
								\draw[-,black] (E1) -- node[left] {\small 2} (N2);
								\draw[-,black] (N1) -- node[above] {\small 2} (E2);
								\draw[-,black] (N2) -- node[below] {\small 1} (E2);
								\draw[-,black] (E2) -- node[above] {\small 3} (N3);
					}}}\;\to\; q
				}
				{
					\infer{\circledcirc(t)\to t}{}
					&
					\infer{\circledcirc(q)\to q}{}
				}
			}
		}
		&\infer{\circledcirc(p)\to p}{}
	}
	$$
\end{example}
\subsection{Some remarks regarding definitions}
\begin{remark}\label{rem_bot_to_top}
	All the rules are formulated in the ``top-to-bottom'' fashion: each rule says that if sequents above the line (premises) are derivable, then a sequent below the line (conclusion) is derivable as well. However, sometimes it is more convenient to consider these rules from bottom to top, e.g. when a sequent is given and you check whether it is derivable. For instance, rules $(\to\div)$ and $(\times\to)$ can be reformulated as follows:
	\begin{itemize}
		\item Let $F\to \div(A/D)$ be a graph sequent; let $e_0\in E_D$ be labeled by \$. Then
		$$
		\infer[(\to\div)]{F\to \div(A/D)}{D[F/e_0]\to A}
		$$
		\item Let $G\to A$ be a graph sequent and let $e\in E_G$ be labeled by $\times(F)$. Then
		$$
		\infer[(\times\to)]{G\to A}{G[F/e]\to A}
		$$
	\end{itemize}
	The remaining rules can also be reformulated in this way (e.g. we provided verbal explanation for $(\div\to)$ above).
\end{remark}
\begin{remark}
	Further we sometimes say: ``let us consider a derivation of a sequent $H\to A$ from bottom-to-top''; this implies that we consider $H\to A$ as the start of a derivation, all its premises as the first step of a derivation and so on; particularly, axioms are last steps of a derivation, if we focus on this point of view. For example, we may say the following: ``if we consider the derivation from Example \ref{ex_derivation} from bottom to top, then the rule $(\to\div)$ is applied after the rule $(\to\times)$''.
\end{remark}
\begin{remark}\label{rem_loops}
	We forbid cases where some external nodes of a graph or some attachment nodes of an edge coincide; that is, we forbid loops (in a general sense). This is done following \cite{Drewes} (to obtain more similarities) and due to our desire to shorten definitions and reasonings and not to consider extra cases. Besides, it is not clear how to define compression if one deals with loops. However, there is a way how to extend our definitions to cases where loops are allowed. In order to do this one has to change definitions as follows:
	\begin{itemize}
		\item Everywhere in Definition \ref{def_hypergraph} $V^\circledast$ is replaced by $V^\ast$.
		\item Replacement (Section \ref{sec_repl}) is defined in the same way but we need to clarify how we understand ``fusing'': namely, if, say, $i$-th and $j$-th attachment nodes of $e_0$ coincide ($att(e_0)(i)=att(e_0)(j),i\ne j$) and $j$-th and $k$-th external nodes of $H$ coincide as well ($k\ne i,j$), then after replacement all three nodes are fused into a single one. E.g.
		\begin{center}
			{\tikz[baseline=.1ex]{
					\node[] (H) {$H=$};
					\node[hyperedge, right=10mm of H] (E) {$a$};
					\node[node,left=7mm of E] (N1) {};
					\node[node,right=7mm of E] (N2) {};
					\draw[-,black] (E) -- node[above] {\small 1} (N1);
					\draw[-,black] (E) to[bend left=30] node[above] {\small 2} (N2);
					\draw[-,black] (E) to[bend right=30] node[below] {\small 3} (N2);
					
					\node[right=10mm of N2] (G) {$G=$};
					\node[node,right=2mm of G] (N3) {};
					\node[above=0mm of N3] {\scriptsize $(1)$};
					\node[below=0mm of N3] {\scriptsize $(2)$};
					\node[node,right=7mm of N3] (N4) {};
					\node[below=0mm of N4] {\scriptsize $(3)$};
					\draw[>=stealth,->,black] (N3) -- node[above] {$b$} (N4);
			}}
		
			{\tikz[baseline=.1ex]{
					\node[] (HG) {$H[G/e_0]=$};
					\node[node,right=1mm of HG] (N) {};
					\draw[>=stealth,->,black] (N) to [out=-30,in=30,looseness=30] node[right] {$b$} (N);
			}}
		\end{center}
		\item In definitions of the rules $(\times\to)$ and $(\to\div)$ of the hypergraph Lambek calculus we use formulations from Remark \ref{rem_bot_to_top} where these rules are defined through replacement. However, this way of definition is somewhat undesirable: it cannot be nicely reformulated in an equivalent top-to-bottom way. Consider, for instance the following two derivations with a type $L_0=\times\bigg(
		\mbox{{\tikz[baseline=.1ex]{
					\node[node] (N) {};
					\draw[>=stealth,->,black] (N) to [out=-30,in=30,looseness=30] node[right] {$p$} (N);
					\node[above=0mm of N] {\scriptsize $(1)$};
					\node[below=0mm of N] {\scriptsize $(2)$};
		}}}\bigg)$ and some type $T$:
		$$
		\infer[(\times\to)]{
		\mbox{{\tikz[baseline=.1ex]{
					\node[node] (N1) {};
					\node[node,right=5mm of N1] (N) {};
					\draw[>=stealth,->,black] (N) to [out=-30,in=30,looseness=30] node[right] {$L_0$} (N);
					\draw[>=stealth,->,black] (N1) -- node[above] {$q$} (N);
					\node[right=10mm of N] {$\to T$};
		}}}
		}{
		\mbox{{\tikz[baseline=.1ex]{
					\node[node] (N1) {};
					\node[node,right=5mm of N1] (N) {};
					\draw[>=stealth,->,black] (N) to [out=-30,in=30,looseness=30] node[right] {$p$} (N);
					\draw[>=stealth,->,black] (N1) -- node[above] {$q$} (N);
					\node[right=10mm of N] {$\to T$};
		}}}
		}
		\qquad
		\infer[(\times\to)]{
			\mbox{{\tikz[baseline=.1ex]{
						\node[node] (N1) {};
						\node[node,right=5mm of N1] (N) {};
						\node[node,right=5mm of N] (N3) {};
						\draw[>=stealth,->,black] (N) -- node[above] {$L_0$} (N3);
						\draw[>=stealth,->,black] (N1) -- node[above] {$q$} (N);
						\node[right=10mm of N] {$\to T$};
			}}}
		}{
			\mbox{{\tikz[baseline=.1ex]{
						\node[node] (N1) {};
						\node[node,right=5mm of N1] (N) {};
						\draw[>=stealth,->,black] (N) to [out=-30,in=30,looseness=30] node[right] {$p$} (N);
						\draw[>=stealth,->,black] (N1) -- node[above] {$q$} (N);
						\node[right=10mm of N] {$\to T$};
			}}}
		}
		$$
	\end{itemize}
	In the case when we forbid coincidences within external or attachment nodes the application of $(\times\to)$ is completely defined by $H$ and $F$ (see notation in Section \ref{sec_axioms_rules}); here, however, this is not true. This is one of reasons why we decided to reject the idea of allowing coincidences of external or attachment nodes. In the remainder of the work we stick to definitions given in Sections \ref{sec_prelim} and \ref{sec_axioms_rules}, and return to the issue of this remark only once in Section \ref{embed_l*1}.
\end{remark}
\begin{remark}
	$\mathrm{HL}$ denotes a ``standard'' variant of the hypergraph Lambek calculus with operations $\div$ and $\times$; however, sometimes we will be interested in reducing or extending this set of operations; then a relevant set of operations will be listed in brackets after $\mathrm{HL}$. E.g. if we want to consider the hypergraph Lambek calculus with $\div$ only, we denote it as $\mathrm{HL}(\div)$.
\end{remark}

\section{Embedding of the Lambek calculus and of its variants}\label{sec_embed}
As promised, $\mathrm{HL}$ naturally generalizes $\mathrm{L}$: we will show how to embed the Lambek calculus into the hypergraph Lambek calculus considering strings as string graphs. Besides, surprisingly $\mathrm{HL}$ can model several extensions of the Lambek calculus, which are discussed below.
\subsection{Embedding of $\mathrm{L}$}\label{sec_embed_lambek}
Types of the Lambek calculus are embedded in $\mathrm{HL}$ by means of a function $tr:Tp(\mathrm{L})\to Tp(\mathrm{HL})$ presented below:
\begin{itemize}
	\item $tr(p):=p, p\in Pr, type(p)=2$;
	\item $tr(A/B):=\div\left(tr(A)\middle/\mbox{
		{\tikz[baseline=.1ex]{
				\node[] (R1) {};
				\node[node,above=-2mm of R1, label=left:{\small $(1)$}] (N1) {};
				\node[node,right=10mm of N1] (N2) {};
				\node[node,right=10mm of N2, label=right:{\small $(2)$}] (N3) {};
				\draw[>=stealth,->,black] (N1) -- node[above] {$\$$} (N2);
				\draw[>=stealth,->,black] (N2) -- node[above] {$tr(B)$} (N3);
		}}
	}\right)$
	\item $tr(B\backslash A):=\div\left(tr(A)\middle/\mbox{
		{\tikz[baseline=.1ex]{
				\node[] (R1) {};
				\node[node,above=-2mm of R1, label=left:{\small $(1)$}] (N1) {};
				\node[node,right=10mm of N1] (N2) {};
				\node[node,right=10mm of N2, label=right:{\small $(2)$}] (N3) {};
				\draw[>=stealth,->,black] (N1) -- node[above] {$tr(B)$} (N2);
				\draw[>=stealth,->,black] (N2) -- node[above] {$\$$} (N3);
		}}
	}\right)$
	\item $tr(A\cdot B):=\times\left(\mbox{
		{\tikz[baseline=.1ex]{
			\node[] (R1) {};
			\node[node,above=-2mm of R1, label=left:{\small $(1)$}] (N1) {};
			\node[node,right=10mm of N1] (N2) {};
			\node[node,right=10mm of N2, label=right:{\small $(2)$}] (N3) {};
			\draw[>=stealth,->,black] (N1) -- node[above] {$tr(A)$} (N2);
			\draw[>=stealth,->,black] (N2) -- node[above] {$tr(B)$} (N3);
		}}
	}\right)$
\end{itemize}
String sequents $\Gamma\to A$ are translated into graph sequents as follows: $tr(\Gamma\to A):=tr(\Gamma)^\bullet\to tr(A)$. Let $tr(Tp(\mathrm{L}))$ be the image of $tr$.

\begin{theorem}\label{embed_lambek}\leavevmode
	\begin{enumerate}
		\item If $\mathrm{L}\vdash \Gamma\to C$, then $\mathrm{HL}\vdash tr(\Gamma\to C)$;
		\item If $\mathrm{HL}\vdash G\to T$ is a derivable graph sequent over $tr(Tp(\mathrm{L}))$, then for some $\Gamma$ and $C$ we have $G\to T=tr(\Gamma\to C)$ (particularly $G$ has to be a string graph) and $\mathrm{L}\vdash \Gamma\to C$.
	\end{enumerate}
\end{theorem}
Proofs of both statements are done by induction on the size of the derivation. See details in \ref{embed_lambek_proof}.

\subsection{Embedding of $\mathrm{NL\diamondsuit}$}
$\mathrm{NL\diamondsuit}$ is presented in \cite{Moortgat}. In this calculus, the set of types is denoted $Tp(\mathrm{NL\diamondsuit})$; types are built from primitive types using $\backslash,/,\cdot$ and two unary operators $\diamondsuit,\square$ (i.e., if $A$ belongs to $Tp(\mathrm{NL\diamondsuit})$, then $\diamondsuit(A)$ and $\square(A)$ are also types of $\mathrm{NL\diamondsuit}$). This variant of the Lambek calculus is nonassociative: antecedents of sequents are considered to be bracketed structures, defined inductively as follows: $\mathcal{A}:=Tp(\mathrm{NL\diamondsuit})\,|\,(\mathcal{A},\mathcal{A})\,|\,(\mathcal{A})^\diamond$. Sequents then are of the form $\Gamma\to A$ where $\Gamma\in\mathcal{A}$ and $A\in Tp(\mathrm{NL\diamondsuit})$.

Rules for $\backslash,\cdot,/$ are formulated as follows (here $\Gamma[\Delta]$ denotes the term $\Gamma$ containing a distinguished occurence of the subterm $\Delta$):
$$
\infer[(\backslash\to)]{\Gamma[(\Pi, A \backslash B)] \to C}{\Pi \to A & \Gamma[B]\to C}
\qquad
\infer[(\to\backslash)]{\Pi \to A \backslash B}{(A, \Pi) \to B}
\qquad
\infer[(\cdot\to)]{\Gamma[A \cdot B] \to C}{\Gamma[(A,B)] \to C}
$$
$$
\infer[(/\to)]{\Gamma[(B / A, \Pi)]\to C}{\Pi \to A & \Gamma[B] \to C}
\qquad
\infer[(\to/)]{\Pi \to B / A}{(\Pi, A) \to B}
\qquad
\infer[(\to\cdot)]{(\Gamma, \Delta) \to A \cdot B}{\Gamma \to A & \Delta \to B}
$$

The following rules for $\square,\diamondsuit$ are added:
$$
\infer[(\diamondsuit\to)]{\Gamma[\diamondsuit A]\to B}{\Gamma[(A)^\diamond]\to B}
\qquad
\infer[(\square\to)]{\Gamma[(\square A)^\diamond]\to B}{\Gamma[A]\to B}
$$
$$
\infer[(\to\diamondsuit)]{(\Gamma)^\diamond\to \diamondsuit A}{\Gamma\to A}
\qquad
\infer[(\to\square)]{\Gamma\to \square A}{(\Gamma)^\diamond\to A}
$$
Moortgat notices in \cite{Moortgat} that $\diamondsuit$ and $\square$ are ``truncated forms of product and implication''. This is evidenced by the way we embed $\mathrm{NL\diamondsuit}$ in $\mathrm{HL}$.

Let us fix the primitive types $p_\diamond,p_{br}\in\mathcal{P}_2$. Consider the following graphs with $X,Y$ being parameters:
\begin{itemize}
	\item $Br(X,Y)={\tikz[baseline=.1ex]{
			\node[node,label=below:{\small $(1)$}] (N1) {};
			\node[node,right=10mm of N1] (N2) {};
			\node[node,right=10mm of N2,label=below:{\small $(2)$}] (N3) {};
			
			\draw[>=stealth,->,black] (N1) -- node[above] {$X$} (N2);
			\draw[>=stealth,->,black] (N2) -- node[above] {$Y$} (N3);
			\draw[>=stealth,->,black] (N1) to[bend left=90] node[above] {$p_{br}$} (N3);
	}}$;
	\item $Diam(X)={\tikz[baseline=.1ex]{
			\node[node,label=below:{\small $(1)$}] (N1) {};
			\node[node,right=12mm of N1,label=below:{\small $(2)$}] (N3) {};
			
			\draw[>=stealth,->,black] (N1) -- node[above] {$X$} (N3);
			\draw[>=stealth,->,black] (N1) to[bend left=90] node[above] {$p_{\diamond}$} (N3);
	}}$
\end{itemize}
Then we introduce the following translation function $tr_{\diamond}$:
\begin{itemize}
	\item $tr_\diamond(p)=p, p\in Pr,type(p)=2;$
	\item $tr_\diamond((A/B))=\div\left(tr_\diamond(A)\middle/ Br(\$,tr_\diamond(B))\right)$;
	\item $tr_\diamond((B\backslash A))=\div\left(tr_\diamond(A)\middle/ Br(tr_\diamond(B),\$)\right)$;
	
	\item $tr_\diamond(\square (A))=\div\left(tr_\diamond(A)\middle/Diam(\$)\right)$;
	\item $tr_\diamond((A\cdot B))=\times\left(Br(tr_\diamond(A),tr_\diamond(B))\right)$;
	\item $tr_\diamond(\diamondsuit (A))=\times\left(Diam(tr_\diamond(A))\right)$;
\end{itemize}
If $\Gamma,\Delta$ are sequences of types, then $tr_{\diamond}((\Gamma,\Delta)):=Br(tr_{\diamond}(\Gamma),tr_{\diamond}(\Delta))$ where $Br(tr_{\diamond}(\Gamma),tr_{\diamond}(\Delta))$ is understood as the replacement of corresponding edges by graphs. Similarly, $tr_{\diamond}((\Gamma)^\diamond):=Diam(tr_{\diamond}(\Gamma))$. Finally, $tr_{\diamond}(\Gamma\to A):=tr_{\diamond}(\Gamma)\to tr_{\diamond}(A)$.
\begin{theorem}\label{embed_nld}\leavevmode
	\begin{enumerate}
		\item If $\mathrm{NL\diamondsuit}\vdash \Gamma\to C$, then we have $\mathrm{HL}\vdash tr_{\diamond}(\Gamma\to C)$;
		\item If $\mathrm{HL}\vdash G\to T$ is a derivable graph sequent over $tr_\diamond(Tp(\mathrm{NL\diamondsuit}))\cup\{p_{br},p_\diamond\}$, then $G\to T=tr_{\diamond}(\Gamma\to C)$ for some $\Gamma$ and $C$ and $\mathrm{NL\diamondsuit}\vdash \Gamma\to C$.
	\end{enumerate}
\end{theorem}
This theorem is proved in a straightforward way similar to the case of $\mathrm{L}$ with few more technicalities. Namely, one has to explain why $p_\diamond$ and $p_{br}$ indeed serve as $\diamondsuit,\square$ and $()$. A sketch of the proof is given in Appendix \ref{embed_nld_proof}.
\begin{remark}
	The nonassociative Lambek calculus $\mathrm{NL}$ can be embedded in $\mathrm{HL}$ as well: it suffices not to consider $\diamondsuit,\square$ and graphs with $p_\diamond$-labeled edges in the above construction.
\end{remark}

\subsection{Embedding of $\mathrm{LP}$}\label{sec_embed_lp}
$\mathrm{LP}$ is $\mathrm{L}$ enriched with the additional premutation rule:
$$
\infer[(P).]{\Gamma\;A\;B\;\Delta\to C}{\Gamma\;B\;A\;\Delta\to C}
$$

One of the ways of modeling this formalism in $\mathrm{HL}$ is by using edges of type 1. The translation function $tr_P$ is the following:
\begin{itemize}
	\item $tr_{\mathrm{P}}(p)=p$;
	\item $tr_{\mathrm{P}}(A/B)=tr_{\mathrm{P}}(B\backslash A)=\div\left(tr_{\mathrm{P}}(A)\middle/{\tikz[baseline=.1ex]{
			\node[hyperedge] (E1) {$\$$};
			\node[node,right= of E1,label=below:{\small $(1)$}] (N1) {};
			\node[hyperedge,right= of N1] (E2) {$\;tr_P(B)\;$};
			\draw[-,black] (E1) -- node[above] {\small 1} (N1);
			\draw[-,black] (N1) -- node[above] {\small 1} (E2);
	}}\right)$;
	\item $tr_{\mathrm{P}}(A\cdot B)=\times\left(
	{\tikz[baseline=.1ex]{
			\node[hyperedge] (E1) {$\;tr_P(A)\;$};
			\node[node,right= of E1,label=below:{\small $(1)$}] (N1) {};
			\node[hyperedge,right= of N1] (E2) {$\;tr_P(B)\;$};
			\draw[-,black] (E1) -- node[above] {\small 1} (N1);
			\draw[-,black] (N1) -- node[above] {\small 1} (E2);
	}}
	\right)$.	
\end{itemize}
If $\Gamma=T_1,\dots,T_n$ is a sequence of types, then $tr_{\mathrm{P}}(\Gamma):=\langle \{v_0\},\{e_i\}_{i=1}^n,att,lab,v_0\rangle$ where $att(e_i)=v_0$, $lab(e_i)=tr_{\mathrm{P}}(T_i)$. As before, $tr_{\mathrm{P}}(\Gamma\to A):=tr_{\mathrm{P}}(\Gamma)\to tr_{\mathrm{P}}(A)$.
\begin{theorem}\label{embed_lp}\leavevmode
	\begin{enumerate}
		\item If $\mathrm{LP}\vdash \Gamma\to C$, then $\mathrm{HL}\vdash tr_{\mathrm{P}}(\Gamma\to C)$;
		\item If $\mathrm{HL}\vdash G\to T$ is a derivable graph sequent over $tr_P(Tp(\mathrm{L}))$, then $G\to T=tr_{\mathrm{P}}(\Gamma\to C)$ for some $\Gamma$ an $C$ and $\mathrm{LP}\vdash \Gamma\to C$.
	\end{enumerate}
\end{theorem}
Proof of this theorem is similar to that of Theorem \ref{embed_lambek} (see \ref{embed_lambek_proof}).

\subsection{Embedding of $\mathrm{L}_{\mathbf{1}}^\ast$}\label{embed_l*1}
	In the string case there is a variant of L where empty antecedents are allowed, and there is an additional type $\mathbf{1}$ called the unit. Then one more axiom and one inference rule are added:
	
	Axiom: $\to\mathbf{1}$.
	
	Rule:
	$$
	\infer[(\mathbf{1}\to)]{\Gamma,\mathbf{1},\Delta\to A}{\Gamma,\Delta\to A}
	$$
	This extension of L is called the Lambek calculus with the unit and it is denoted by $\mathrm{L}_{\mathbf{1}}^\ast$; a corresponding set of types is denoted $Tp_{\mathbf{1}}$. Certainly, we wish this calculus to be embedded in HL as well. Our definition of a string graph (Definition \ref{def_str_gr}), though, does not include the case of an empty string so we cannot use $tr$ from Section \ref{sec_embed_lambek}. This can be done if we allow coincidences of external nodes or of attachment nodes in the way proposed in Remark \ref{rem_loops}. If so, then we define a function $tr_{\mathbf{1}}$ as follows:
	\begin{itemize}
		\item $tr_{\mathbf{1}}(p)=p,p\in Pr, type(p)=2$;
		\item $tr_{\mathbf{1}}(A/B)=\div(tr_{\mathbf{1}}(A)/(\$\,tr_{\mathbf{1}}(B))^\bullet)$;
		\item $tr_{\mathbf{1}}(B\backslash A)=\div(tr_{\mathbf{1}}(A)/(tr_{\mathbf{1}}(B)\,\$)^\bullet)$;
		\item $tr_{\mathbf{1}}(A\cdot B)=\times((tr_{\mathbf{1}}(A)\;tr_{\mathbf{1}}(B))^\bullet)$;
		\item $tr_{\mathbf{1}}(\mathbf{1}):=\times(\langle\{v_0\},\emptyset,\emptyset,\emptyset,v_0v_0\rangle)=\times\left(
		\mbox{{\tikz[baseline=.1ex]{
					\node (R) {};
					\node[node, above=-1mm of R] (N) {};
					\node[left=0mm of N] {\scriptsize $(1)$};
					\node[right=0mm of N] {\scriptsize $(2)$};
		}}}\right)$
	\end{itemize}
	The inductive definition of $tr_{\mathbf{1}}$ does not differ from that of $tr$; the only difference is in how $tr_{\mathbf{1}}$ treats $\mathbf{1}$. Let us also extend $tr_{\mathbf{1}}$ to sequents as follows: $tr_{\mathbf{1}}(\Gamma\to A)=tr_{\mathbf{1}}(\Gamma)^\bullet\to tr_{\mathbf{1}}(A)$. If $\Gamma$ is empty, we put $(\Lambda)^\bullet:=\langle\{v_0\},\emptyset,\emptyset,\emptyset,v_0v_0\rangle$.
	\begin{theorem}
		Let $\Gamma\to C$ be a sequent over $Tp_{\mathbf{1}}$.
		$\mathrm{L}_{\mathbf{1}}^\ast\vdash \Gamma\to C$ if and only if
		$\mathrm{HL}\vdash tr_{\mathbf{1}}(\Gamma\to C)$. 
	\end{theorem}
	\begin{proof}
		The proof is similar to that of Theorem \ref{embed_lambek}. In both directions we need to use induction on length of a derivation. Let us consider cases where $\mathbf{1}$ participates in each direction.
		\\
		Let $\mathrm{L}_{\mathbf{1}}^\ast\vdash \Gamma\to C$. 
		
		If it is the axiom $\to\mathbf{1}$, then it is translated into a graph sequent $(\Lambda)^\bullet\to\times((\Lambda)^\bullet)$.
		
		If this sequent is obtained after the rule application
		$$
		\infer[(\mathbf{1}\to)]{\Phi,\mathbf{1},\Psi\to A}{\Phi,\Psi\to A}
		$$
		then this rule can be remodeled by $(\times\to)$ with the type $tr_{\mathbf{1}}(\mathbf{1})$.
		\\
		Let $\mathrm{HL}\vdash tr_{\mathbf{1}}(\Gamma\to C)$.
		
		If $tr_{\mathbf{1}}(C)=tr_{\mathbf{1}}(\mathbf{1})$ and the last rule is $(\to\times)$, then $tr_{\mathbf{1}}(\Gamma)=(\Lambda)^\bullet$, and this corresponds to the axiom case $\to\mathbf{1}$.
		
		If the last rule is $(\times\to)$, and it is applied to a type of the form $tp_{\mathbf{1}}(\mathbf{1})$, then it can be remodeled using the rule $(\mathbf{1}\to)$.
		\qed
	\end{proof}
	This theorem is weaker than previous embedding theorems: we restrict our consideration to string graphs only. Generally, if a graph sequent $G\to A$ is over the set $tp_{\mathbf{1}}(Tp_{\mathbf{1}}^\ast)$, then $G$ is not necessarily a string graph: it may contain loops.
\subsection{$\mathrm{HL}$ as a Source of Extensions of $\mathrm{L}$}
Now it is plausible that different variants of the Lambek calculus can be considered as certain fragments of the hypergraph Lambek calculus. After noticing this, we came up with the thought that it would be interesting to do the opposite: to consider certain ``natural'' fragments of $\mathrm{HL}$ and then to try to interpret them as variants of $\mathrm{L}$. 

One of experiments in this direction led us to an extension of the Lambek calculus, which is presented below. We call it \emph{Lambek calculus with weights}: its types are enriched with non-negative integers. Though we firstly developed it as a fragment of $\mathrm{HL}$ and then as an autonomous formalism, here we present the latter before the former.
\subsection{Lambek Calculus with Weights}
As usually, we fix a countable set $Pr$ of primitive types.
\begin{definition}
	The set $Tp(\mathrm{LW})$ is the least set such that
	\begin{itemize}
		\item If $p\in Pr$ and $n\in\mathbb{N}$, then $(p;n)$ is in $Tp(\mathrm{LW})$;
		\item If $A,B$ are in $Tp(\mathrm{LW})$ and $n$ is a natural number, then $(B\backslash A;n), (A/B;n), (A\cdot B;n)$ are also in $Tp(\mathrm{LW})$.
	\end{itemize}
\end{definition}
Sequents in this calculus are also enriched with numbers: they are of the form $$\langle n\rangle T_1,\dots,T_k\to T$$ where $T_i, T$ are types ($k>0$) and $n$ belongs to $\mathbb{N}$.
Axioms and rules of the Lambek calculus with weights are similar to those in $\mathrm{L}$ but they are applied with regard to numbers in types. There are two more axioms regarding weights of primitive types. Below $m$ is less than or equal to $n$.
$$\infer[]{\langle 0 \rangle(p;0)\to (p;0)}{}
$$
$$
\infer[(\backslash\to)]{\langle n_1+n_2+k \rangle \Gamma, \Pi, (A \backslash B;k), \Delta \to C}{\langle n_1 \rangle \Pi \to A & \langle n_2 \rangle \Gamma, B, \Delta \to C}
\qquad
\infer[(\to\backslash)]{\langle n-m \rangle \Pi \to (A \backslash B;m)}{\langle n \rangle A, \Pi \to B}
$$
$$
\infer[(/\to)]{\langle n_1+n_2+k \rangle \Gamma, (B / A;k), \Pi, \Delta \to C}{\langle n_1 \rangle\Pi \to A & \langle n_2 \rangle \Gamma, B, \Delta \to C}
\qquad
\infer[(\to/)]{\langle n-m \rangle\Pi \to (B / A;m)}{\langle n \rangle \Pi, A \to B}
$$
$$
\infer[(\cdot\to)]{\langle n-m \rangle \Gamma, (A \cdot B;m), \Delta \to C}{\langle n \rangle\Gamma, A, B, \Delta \to C}
\qquad
\infer[(\to\cdot)]{\langle n_1+n_2+k \rangle \Gamma, \Delta \to (A \cdot B;k)}{\langle n_1 \rangle\Gamma \to A & \langle n_2 \rangle \Delta \to B}
$$
$$
\infer[(\mathrm{w}\to)]{\langle n-m \rangle\Gamma, (p;m), \Delta \to C}{\langle n \rangle\Gamma, (p;0), \Delta \to C}
\qquad
\infer[(\to\mathrm{w})]{\langle n+k \rangle \Gamma \to (p;k)}{\langle n \rangle \Gamma \to (p;0)}
$$
In order to shorten notation we denote a type of the form $(T;0)$ as just $T$.
\begin{remark}
	Let us inductively define an unweighting function $\mathrm{unw}:Tp(\mathrm{LW})\to Tp(\mathrm{L})$ in a natural way:
	\begin{itemize}
		\item $\mathrm{unw}((p;n))=p$;
		\item $\mathrm{unw}((A\circ B;n))=\mathrm{unw}(A)\circ \mathrm{unw}(B)$ for $\circ\in\{\backslash,/,\cdot\}$.
	\end{itemize}
	We also say that $\mathrm{unw}(\langle n \rangle \Gamma\to C)=\mathrm{unw}(\Gamma) \to \mathrm{unw}(C)$.
	\\
	For such a function, if $\mathrm{LW}\vdash \Gamma\to C$, then $\mathrm{L}\vdash \mathrm{unw}(\Gamma) \to \mathrm{unw}(C)$. This can be seen from the fact that axioms and rules in $\mathrm{LW}$ disregarding numbers coincide with those of $\mathrm{L}$ (and rules $(\mathrm{w}\to)$ and $(\to\mathrm{w})$ turn into rules where a premise equals a conclusion).
\end{remark}
\begin{example}
	In $\mathrm{L}$ the following sequent is derivable: $\mathrm{L}\vdash p/q,q/r,r\to p$. Moreover, there are two derivation trees for it:
	$$
	\infer[(\div\to)]{p/q,q/r,r\to p}{\infer[(\div\to)]{p/q,q\to p}{p\to p & q\to q} & r\to r}
	\qquad
	\infer[(\div\to)]{p/q,q/r,r\to p}{p\to p & \infer[(\div\to)]{q/r,r\to q}{q\to q & r\to r}}
	$$
	Now we add weights to types as follows: $\langle 0 \rangle (p/q;1),(q;2)/r,r\to (p;1)$. This sequent is derivable in $\mathrm{LW}$ as well:
	$$
	\infer[(\div\to)]{\langle 0 \rangle (p/q;1),(q;2)/r,r\to (p;1)}{\infer[(\mathrm{w}\to)]{\langle 0 \rangle (p/q;1),(q;2)\to (p;1)}{\infer[(\div\to)]{\langle 2 \rangle (p/q;1),q\to (p;1)}{\infer[(\mathrm{w}\to)]{\langle 1 \rangle p\to (p;1)}{\langle 0 \rangle p\to p} & \langle 0 \rangle q\to q}} & \langle 0 \rangle r\to r}
	$$
	This derivation corresponds via the function $\mathrm{unw}$ to the first derivation of the sequent in $\mathrm{L}$. However, the second derivation cannot be recreated in $\mathrm{LW}$: $(q;2)$ has weight 2, which has to be ``unleashed'' before application of division within $(p/q;1)$.
\end{example}
\begin{example}
	We can also define grammars based on LW with the only difference that a language now consists of pairs $(w;n)$ where $w$ is a string, and $n\in\mathbb{N}$ is its weight. Consider e.g. the grammar $Gr=\langle\{a,b,c\},s,\triangleright\rangle$ where
	\begin{itemize}
		\item $a\triangleright (p/q;1)$;
		\item $b\triangleright q, p\backslash q$;
		\item $c\triangleright p\backslash (s;1),s\backslash (s;1)$.
	\end{itemize}
	If we disregard weights, this grammar generates the language $\{a^nb^nc^k|n,k>0\}$. Taking weights into account we obtain $L(Gr)=\{(a^nb^nc^k;m)|m-n+k=0,\,n,k>0\}$. If we rid of second components of elements in $L(Gr)$, we obtain a usual language $\{a^nb^nc^k|n\ge k>0\}$ which is not context-free due to Ogden's lemma.
\end{example}
As we announced at the beginning of this section, $\mathrm{LW}$ may be considered as a fragment of $\mathrm{HL}$. This is done using the  function $tr_W$ we define below. Firstly, let us accept the following notation: if $\times(M)$ is a type, then $\times(M)\oplus k$ denotes a type $\times(M^\prime)$ where $M^\prime=\langle V_M\sqcup\{u_1,\dots,u_k\},E_M,att_M,lab_M,ext_M\rangle$ for $u_1,\dots,u_k$ being new nodes; similarly, if $\div(N/D)$ is a type, then $\div(N/D)\ominus k$ denotes a type $\div(N/D^\prime)$ where $D^\prime=\langle V_D\sqcup\{u_1,\dots,u_k\},E_D,att_D,lab_D,ext_D\rangle$ for new nodes $u_1,\dots,u_k$.

$tr_W$ can be easily defined using this notation similarly to $tr$ defined in Section \ref{sec_embed_lambek}:
\begin{itemize}
	\item $tr_W((p;0)):=p, p\in Pr, type(p)=2$;
	\item $tr_W((p;n)):=\times(\circledcirc(p))\oplus n, p\in Pr$;
	\item $tr_W((A/B;n)):=\div\left(tr_W(A)\middle/\mbox{
		{\tikz[baseline=.1ex]{
				\node[] (R1) {};
				\node[node,above=-2mm of R1, label=left:{\small $(1)$}] (N1) {};
				\node[node,right=14mm of N1] (N2) {};
				\node[node,right=14mm of N2, label=right:{\small $(2)$}] (N3) {};
				\draw[>=stealth,->,black] (N1) -- node[above] {$\$$} (N2);
				\draw[>=stealth,->,black] (N2) -- node[above] {$tr_W(B)$} (N3);
		}}
	}\right)\ominus n;$
	\item $tr_W((B\backslash A;n)):=\div\left(tr_W(A)\middle/\mbox{
		{\tikz[baseline=.1ex]{
				\node[] (R1) {};
				\node[node,above=-2mm of R1, label=left:{\small $(1)$}] (N1) {};
				\node[node,right=14mm of N1] (N2) {};
				\node[node,right=14mm of N2, label=right:{\small $(2)$}] (N3) {};
				\draw[>=stealth,->,black] (N1) -- node[above] {$tr_W(B)$} (N2);
				\draw[>=stealth,->,black] (N2) -- node[above] {$\$$} (N3);
		}}
	}\right)\ominus n;$
	\item $tr_W((A\cdot B;n)):=\times\left(\mbox{
		{\tikz[baseline=.1ex]{
				\node[] (R1) {};
				\node[node,above=-2mm of R1, label=left:{\small $(1)$}] (N1) {};
				\node[node,right=14mm of N1] (N2) {};
				\node[node,right=14mm of N2, label=right:{\small $(2)$}] (N3) {};
				\draw[>=stealth,->,black] (N1) -- node[above] {$tr_W(A)$} (N2);
				\draw[>=stealth,->,black] (N2) -- node[above] {$tr_W(B)$} (N3);
		}}
	}\right)\oplus n.$
\end{itemize}
Thus we model weights by isolated nodes. Note that we firstly thought about the idea of adding isolated nodes to usual Lambek types, and only then came up with the construction of $\mathrm{LW}$.

Returning to formal reasonings, we extend $tr_W$ on sequents: $tr_W(\langle n\rangle \Gamma\to A):=tr_W(\Gamma)^\bullet\oplus n \to tr_W(A)$.
Then the following theorem can be proved using the same methods as in the proof of Theorem \ref{embed_lambek}:
\begin{theorem}\label{embed_lw}\leavevmode
	\begin{enumerate}
		\item If $\mathrm{LW}\vdash\langle n\rangle \Gamma\to C$, then $\mathrm{HL}\vdash tr_W(\langle n\rangle\Gamma\to C)$;
		\item If $\mathrm{HL}\vdash G\to T$ is a derivable graph sequent over $tr_W(Tp(\mathrm{LW}))$, then for some $\Gamma$, $C$ and $n$ we have $G\to T=tr_W(\langle n\rangle \Gamma\to C)$ and $\mathrm{LW}\vdash \langle n\rangle \Gamma\to C$.
	\end{enumerate}
\end{theorem}
This theorem along with Proposition \ref{differ_IN} immediately yields, e.g., the following
\begin{proposition}
	If $\mathrm{LW}\vdash \langle m\rangle\Gamma\to A$ and $\mathrm{LW}\vdash \langle n\rangle\Gamma\to A$, then $m=n$.
\end{proposition}
Certainly, the Lambek calculus with weights deserves a separate work. Many questions of theoretical and practical nature remain open. We have an intuitive feeling that $\mathrm{LW}$ could have applications in linguistics to model ``weights'' of sentences. This could be useful in syntactic disambiguation: namely, if a sentence can be interpreted in two ways, and one of them is more likely than the other one, then we would expect that a corresponding sequent for the second interpretation has more weight than that for the first one.

Summing up, there is a lot of work to do with $\mathrm{LW}$. However, the established connection between $\mathrm{LW}$ and $\mathrm{HL}$ reduces the amount of work; e.g. the cut elimination theorem for $\mathrm{LW}$ now immediately follows from Theorem \ref{cut} proved for $\mathrm{HL}$. 
\section{Structural Properties of $\mathrm{HL}$}
In this section we consider several properties of $\mathrm{HL}$. We start with an expected
\begin{proposition}\label{TtoT}
	$\mathrm{HL}\vdash \circledcirc(T)\to T$ for all types $T$.
\end{proposition}
\begin{proof}
	Induction on size of $T$. If $T$ is primitive, then $\circledcirc(T)\to T$ is an axiom.
	\\
	If $T=\div(N/D)$ and $E_D=\{e_0,\dots,e_k\}$ where $lab(e_0)=\$$, then
	$$
	\infer[(\to\div)]{\circledcirc(\div(N/D))\to \div(N/D)}{
		\infer[(\div\to)]{D[e_0:=\div(N/D)]\to N}{\circledcirc(N)\to N & \circledcirc(lab(e_1))\to lab(e_1) & \dots & \circledcirc(lab(e_k))\to lab(e_k)}
	}
	$$
	All the above sequents are derivable by induction hypothesis.
	\\
	If $T=\times(M)$ and $E_M=\{e_1,\dots,e_l\}$, then
	$$
	\infer[(\times\to)]{\circledcirc(\times(M))\to \times(M)}{
		\infer[(\to\times)]{M\to \times(M)}{\circledcirc(lab(e_1))\to lab(e_1) & \dots & \circledcirc(lab(e_l))\to lab(e_l)}
	}
	$$
	Again, we apply induction hypothesis.
	\qed
\end{proof}
Therefore, the axiom $\circledcirc(p)\to p$ where primitive types are considered can be replaced by more general one $\circledcirc(T)\to T$ for all types; this does not change the set of derivable sequents.

Below we introduce another important properties of $\mathrm{HL}$ that will be implicitly or explicitly applied throughout this work.

\subsection{Sizes of Types and The Subformula Property}
Looking at derivations of graph sequents, one observes that sequents within them become smaller and simpler going from bottom to top. This observation is well known for $\mathrm{L}$; our aim is to generalize it for $\mathrm{HL}$. 
\begin{definition}
	A size of a type $T\in\mathrm{HL}$ is the total number of primitive types and operators $\div$ and $\times$ within it. Formally,
	\begin{itemize}
		\item $T=p\Rightarrow |T|=1$;
		\item If $T=\div(N/D)$, $E_D=\{d_0,\dots,d_k\}$, $lab(d_0)=\$$, then $|T|=|N|+|lab(d_1)|+\dots+|lab(d_k)|+1$;
		\item If $T=\times(M)$, $E_M=\{m_1,\dots,m_k\}$, then $|T|=|lab(m_1)|+\dots+|lab(m_k)|+1$.
	\end{itemize}
	We also define $|H\to A|:=|lab(h_1)|+\dots+|lab(h_k)|+|A|$ for $E_H=\{h_1,\dots,h_k\}$.
\end{definition}
There are two observations:
\begin{enumerate}
	\item In each rule of $\mathrm{HL}$ the sum of sizes of all premises is less than size of a conclusion (more precise, one less);
	\item In a derivation of a sequent $H\to A$ only subtypes of $A$ or of labels of $H$ can occur.
\end{enumerate}
Both of them directly follow from structure of rules of $\mathrm{HL}$. They yield decidability of $\mathrm{HL}$: given a sequent $H\to A$, one can go through its all possible derivations (there are finitely many ones due to the above statements) and check whether one of them is correct. Of course, such an algorithm is VERY unefficient: one has to go through all possible graph structures of certain size. The algorithmic complexity of $\mathrm{HL}$ will be discussed in Section \ref{sec_complex}.

\subsection{The cut elimination}
One of fundamental properties of $\mathrm{L}$ is admissibility of the following rule called the cut rule:
$$
\infer[(\mathrm{cut})]{\Gamma\;\Pi\;\Delta\to B}{\Pi\to A & \Gamma\;A\;\Delta\to B}
$$
Admissibility means that each sequent that can be derived in $\mathrm{L}$ enriched with $(CUT)$ can be derived in $\mathrm{L}$ without this rule. For $\mathrm{L}$ this was proved by Lambek in \cite{Lambek}. 

This rule can be naturally extended to $\mathrm{HL}$ as follows. Let $H\to A, G\to B$ be graph sequents, $e_0\in E_G$ be an edge, and $lab(e_0)$ be equal to $A$. Then
$$
\infer[(\mathrm{cut})]{G[H/e_0]\to B}{H\to A & G\to B}
$$
\begin{theorem}[cut elimination]\label{cut}
	If $F\to B$ is derivable in $\mathrm{HL}$ enriched with $(\mathrm{cut})$, then it is derivable in $\mathrm{HL}$.
\end{theorem}
It suffices to prove that if $\mathrm{HL}\vdash H\to A$ and $\mathrm{HL}\vdash G\to B$, then $\mathrm{HL}\vdash G[H/e_0]\to B$ where $lab(e_0)=A$. It is done by induction on $|H\to A|+|G\to B|$. We consider different cases depending on structure of these sequents. See the proof in \ref{cut_proof}.

The cut elimination theorem implies reversibility of rules $(\times\to)$ and $(\to\div)$. This is stated in
\begin{proposition}\label{reversibility}\leavevmode
	\begin{enumerate}
		\item If $\mathrm{HL}\vdash H\to C$ and $e_0\in E_H$ is labeled by $\times(M)$, then $\mathrm{HL}\vdash H[M/e_0]\to C$;
		\item If $\mathrm{HL}\vdash H\to \div(N/D)$ and $e_0\in E_D$ is labeled by $\$$, then $\mathrm{HL}\vdash D[H/e_0]\to N$.
	\end{enumerate}
\end{proposition}
\begin{proof}\leavevmode
	\begin{enumerate}
		\item Use the cut rule as below:
		$$
		\infer[(\mathrm{cut})]{H[M/e_0]\to C}{M\to \times(M) & H\to C}
		$$
		Derivability of $M\to\times(M)$ is trivial.
		\item Use the cut rule as below:
		$$
		\infer[(\mathrm{cut})]{D[H/e_0]\to N}{H\to \div(N/D) & D[e_0:=\div(N/D)]\to N}
		$$
		Derivability of $D[e_0:=\div(N/D)]\to N$ is trivial as well (using $(\div\to)$ we come up with premises $\circledcirc(N)\to N$ and $\circledcirc(lab_D(e))\to lab_D(e),\; e\ne e_0$).
	\end{enumerate}
	\qed
\end{proof}	
	
\subsection{Counters}
One of features $\mathrm{HL}$ inherits from $\mathrm{L}$ is so-called counters.
\begin{definition}
	Let $f:Pr\to \mathbb{Z}$ be some function. An $f$-counter $\#_f:Tp(\mathrm{HL})\to\mathbb{Z}$ is defined as follows:
	\begin{itemize}
		\item $\#_f(p)=f(p)$;
		\item If $T=\div(N/D)$ and $E_D=\{e_0,e_1,\dots,e_n\}$ where $lab(e_0)=\$$, then 
		\begin{displaymath}
		\#_f(T)=\#_f(N)-\sum_{i=1}^n \#_f(lab(e_i)).
		\end{displaymath}
		\item If $T=\times(M)$ and $E_M=\{e_1,\dots,e_n\}$, then 
		\begin{displaymath}
		\#_f(T)=\sum_{i=1}^n \#_f(lab(e_i)).
		\end{displaymath}
	\end{itemize}
	If $G$ is labeled by types and $E_G=\{e_1,\dots,e_n\}$, then $\#_f(G):=\sum_{i=1}^n \#_f(lab(e_i))$.
\end{definition}
\begin{proposition}\label{counter}
	If $\mathrm{HL}\vdash H\to A$, then $\#_f(H)=\#_f(A)$ for each $f$.
\end{proposition}
The proof is done by a straightforward induction.

Counters can be used to prove that a sequent $H\to A$ is not derivable: it suffices to present such a counter $f$ that $f(H)\ne f(A)$.
\begin{example}
	We provide two specific examples of counters:
	\begin{itemize}
		\item $f=g_q, q\in Pr:\;g_q(p)=1$ whenever $p=q$ and $g_q(p)=0$ otherwise. We write $\#_p$ instead of $\#_{g_p}$.
		\\
		E.g. for types from Example \ref{ex_types} $\#_{q}(A_1)=1$, $\#_{s}(A_1)=-1$, $\#_{p}(A_4)=1$.
		\item $f=h_m, m\in \mathbb{N}:\;h_m(p)=1$ whenever $type(p)=m$ and $h_m(p)=0$ otherwise.
	\end{itemize}
\end{example}
Using counters we can prove, e.g., the following
\begin{proposition}\label{differ_IN}
	Let $H\to T$ be a graph sequent; let $H^\prime = \langle V_H\sqcup \{v_1,\dots,v_k\},E_H, \\att_H, lab_H, ext_H\rangle$ (that is, we add $k$ new isolated nodes $v_1,\dots,v_k$ to $H$) for $k>0$. Then at most one of sequents $H\to T$ and $H^\prime\to T$ is derivable.
\end{proposition}
I.e. if two sequents differ from each other only by a few nonexternal isolated nodes, then they cannot both be derivable.
\begin{proof}
	Let us fix a new primitive type $\iota$ ($type(\iota)=0$) and change all nonexternal isolated nodes within $H$, types of $H$ and $T$ by  edges labeled by $\iota$ (note that they are not attached to any node). Denote the result of such a procedure as $iota(H), iota(T)$. Now note that if $\mathrm{HL}\vdash H\to T$, then $\mathrm{HL}\vdash iota(H)\to iota(T)$ (it suffices to check correctness w.r.t. all rules); therefore, $\#_\iota(iota(H))=\#_\iota(iota(T))$. Besides, if $\mathrm{HL}\vdash H^\prime\to T$, then, similarly, $\#_\iota(iota(H^\prime))=\#_\iota(iota(T))$. However, $\#_\iota(iota(H))$ cannot be equal to $\#_\iota(iota(H^\prime))$ since $H$ and $H^\prime$ differ by exactly the number of nonexternal isolated nodes, and consequently $\#_\iota(iota(H^\prime))>\#_\iota(iota(H))$. This leads to contradiction.
\end{proof}
\subsection{Wolf lemma}\label{sec_wolf}
\epigraph{\itshape Such a dreadful flight!\\ Wolves eat wolves on sight.}{---Kornei Chukovsky, \textit{Cock-The-Roach}}
Kornei Chukovsky was a Russian children writer. These lines of his poem ``Cock-the-roach'' look different in Russian, and they are literally translated as ``Frightened wolves ate each other''. This, of course, sounds absurd (which was an author's intention). The statement and the proof of the lemma below essentially say that like wolves types cannot ``eat each other'' and disappear without a trace. 
\begin{definition}
	Let $A$ be a type, and let $B$ be its distinguished subtype. We say that $B$ is a top occurrence within $A$ if one of the following holds:
	\begin{enumerate}
		\item $A=B$;
		\item $A=\times(M)$ and $\exists e_0\in E_M$ such that $B$ is a top occurrence within $lab(e_0)$;
		\item $A=\div(N/D)$ and $B$ is a top occurrence within $N$. 
	\end{enumerate}
\end{definition}
\begin{example}
	In Example \ref{ex_types} $q$ and $p$ are top occurrences within $A_4$, and $r$ and $s$ are not.
\end{example}
\begin{definition}\label{def_lonely}
	A primitive type $p$ is said to be lonely in a type $A$ if for each top occurrence of $p$ within $A$ there is a subtype $\times(M)$ of $A$ such that $|E_M|\ge 2$ and for some $e_0\in E_M$ $lab(e_0)=p$ is that top occurrence.
\end{definition}
\begin{example}
	In Example \ref{ex_types} $p,s,r$ are lonely in $A_4$, and $q$ is not.
\end{example}
\begin{definition}
	A type $A$ is called skeleton if $A=\times(M)$, $E_M=\emptyset$ and $|ext_M|=|V_M|$.
\end{definition}
\begin{lemma}[wolf lemma]\label{wolf1}
	Let $p\in Pr$ be lonely in $\times(H)$ and let $\times(H)$ not contain skeleton subtypes. Then $\mathrm{HL}\not\vdash H\to p$.
\end{lemma}
\begin{proof}
	The proof is ex falso: assume that $\mathrm{HL}\vdash H\to p$. If a derivation contains an axiom only, then $H=\circledcirc(p)$, which contradicts loneliness of $\times(H)$.
	\\
	Let a derivation include more than one step. There has to be an axiom of the form $\circledcirc(p)\to p$ in this derivation where $p$ is the same as the succedent in $H\to p$. Now it suffices to notice that, however, no rule can be infered to $\circledcirc(p)\to p$ in this derivation.
	\\
	Indeed, if $(\div\to)$ is applied to it, then this step is of the form (notation is like in Section \ref{sec_axioms_rules})
	$$
	\infer[(\div\to)]{G\to p}{\circledcirc(p)\to p & H_1\to lab(d_1) & \dots & H_k\to lab(d_k)}
	$$
	where $G=\circledcirc(p)[D/e][d_0:=\div(N/D)][H_1/d_1,\dots,H_k/d_k]$ and $e$ is the only edge of $\circledcirc(p)$ labeled by $N$. This implies that $N$ has to equal $p$. Consequently, $\times(H)$ contains $\div(p/D)$ as a top occurrence, which contradicts that $p$ is lonely in $\times(H)$. Therefore, this is impossible.
	\\
	Let $(\times\to)$ be infered to $\circledcirc(p)\to p$. Then this step can be presented in the form
	$$
	\infer[(\times\to)]{\circledcirc(p)\llbracket \times(F)/F\rrbracket\to p}{\circledcirc(p)\to p}
	$$
	Here $F$ is a subgraph of $\circledcirc(p)$. If $F$ contains the $p$-labeled edge, then $p$ is not lonely since $\times(F)$ is a top occurrence within $\times(H)$. The remaining option is that $F$ does not contain edges. Note that all nodes in $F$ in such a case have to be external since they are attached to the $p$-labeled edge. Thus, $\times(F)$ is skeleton which contradicts the assumption of the lemma.
	\qed
\end{proof}
We will use its corollary, which we also call ``a wolf lemma'':
\begin{corollary}\label{wolf}
	Let $\mathcal{T}$ be a set of types such that for each $T\in\mathcal{T}$ $T$ does not have skeleton subtypes and $p$ is lonely in $T$. Let $\mathrm{HL}\vdash H\to p$ for $H\in\mathcal{H}(\mathcal{T})$. Then $H=\circledcirc(p)$.
\end{corollary}
\begin{proof}
	$\times(H)$ does not have skeleton subtypes. Thus, according to Lemma \ref{wolf1} $p$ is not lonely in $\times(H)$. This means that there is a top occurrence of $p$ within $\times(H)$ for which Definition \ref{def_lonely} does not hold. Let $E_H=\{e_1,\dots,e_n\}$. If this occurrence is a proper subtype of some type $T=lab(e_k)$, then $p$ is not lonely within $T$, which contradicts $T\in\mathcal{T}$. Thus for some $k$ $lab(e_k)=p$. In order for $p$ to be lonely within $\times(H)$, $|E_H|$ necessarily equals 1. This implies that $H\to p$ does not contain $\div$ or $\times$ types, which allows us to draw a conclusion that $H\to p$ is an axiom and $H=\circledcirc(p)$. 
	\qed
\end{proof}
\subsection{Simple types}
This section presents a simple but very useful technical result regarding derivability in $\mathrm{HL}$. It is developed in order to show a connection between the hypergraph Lambek calculus and hyperedge replacement grammars and to reason some examples regarding hypergraph Lambek grammars. The below theorem says that if types in an antecedent of a derivable sequent do not have denominators containing types with division, then we can derive this sequent by simply ``reducing'' denominators with subgraphs of the antecedent and by applying $(\times\to)$.
\begin{definition}
	A type $A$ is called simple if one of the following holds:
	\begin{itemize}
		\item $A$ is primitive;
		\item $A=\times(M)$, $E_M=\{m_1,\dots,m_l\}$ and $lab(m_1),\dots, lab(m_l)$ are simple;
		\item $A=\div(N/D)$, $E_D=\{d_0,\dots,d_k\}$, $lab(d_0)=\$$, $N$ is simple, and $lab(d_1)$, $\dots$, $lab(d_k)$ are \emph{primitive}.
	\end{itemize} 
\end{definition}
\begin{theorem}\label{simple_der}
	Let $\mathrm{HL}\vdash H\to P$ where $H$ is labeled by simple types and $P$ is either primitive or is of the form $\times(K)$ where all edge labels in $K$ are primitive. Then there exists a \emph{simple derivation} of $H\to P$, i.e. such a derivation that 
	\begin{enumerate}
		\item\label{simple_der_to_times} The rule $(\to\times)$ either does not appear or is applied once at the first step of a derivation.
		\item\label{simple_der_div_to} In each application of $(\div\to)$ all the premises except for the first one are of the form $\circledcirc(q)\to q,\;q\in Pr$.
		\item\label{simple_der_times_to} If a sequent $H^\prime\to p$ within the derivation tree of this derivation contains a type of the form $\times(M)$ in the antecedent, then the rule, after which $H^\prime\to p$ appears, must be $(\times\to)$.
	\end{enumerate}
\end{theorem}
\begin{proof}
	Firstly, note that if $P=\times(K)$, then the rule $(\to\times)$ has to be applied one time in a derivation. Let this be as follows (where $E_K=\{k_1,\dots,k_m\}$):
	$$
	\infer[(\to\times)]{K[G_1/k_1,\dots,G_m/k_m]\to \times(K)}{G_1\to lab_K(k_1) & \dots & G_m\to lab_K(k_m)}
	$$
	A derivation of $G_i\to lab_K(k_i),\;i=1,\dots, m$ is a sequence of applications of rules $(\times\to)$ and $(\div\to)$ only since $lab_K(k_i)$ is primitive. Thus we can repeat this sequence of derivations within $K[G_1/k_1,\dots,G_m/k_m]\to \times(K)$ from bottom to top for $i=1$, $i=2$, ..., $i=m$. After this we obtain the sequent $K\to \times(K)$ and now apply the rule $(\to\times)$. Therefore, each derivation of $H\to P$ can be remodeled in such a way that the condition \ref{simple_der_times_to} is met. Let us further consider such a derivation (name it $\Delta$).
	
	Now let us prove that $\Delta$ can be remodeled in such a way that a new one will satisfy conditions \ref{simple_der_div_to} and \ref{simple_der_times_to} as well. This is done by induction on length of $\Delta$. 
	
	If $H\to P$ is an axiom (particularly, $P$ is primitive), then both requirements are satisfied. 
	
	If $H$ contains an edge $e_0$ labeled by a type $\times(M)$, then we can derive a sequent $H[M/e_0]\to P$ (see Proposition \ref{reversibility}). Since length of a derivation equals the total number of symbols $\times$ and $\div$ included in types of an antecedent and a succedent, length of derivation of $H[M/e_0]\to P$ is less than that of $H\to P$; thus we can apply the induction hypothesis and obtain a simple derivation for $H[M/e_0]\to P$. The it suffices to apply the rule $(\times\to)$ to this sequent:
	$$
	\infer[(\times\to)]{H\to P}{H[M/e_0]\to P}
	$$
	Hence we obtained a simple derivation for $H\to P$.
	
	Let $H$ not contain types of the form $\times(M)$. Then the last step of ane derivation must be of the form
	$$
	\infer[(\div\to)]{G[D/e][d_0:=\div(N/D)][H_1/d_1,\dots,H_k/d_k]\to P}{G\to P & H_1\to lab(d_1) &\dots & H_k\to lab(d_k)}
	$$
	where $H=G[D/e][d_0:=\div(N/D)][H_1/d_1,\dots,H_k/d_k]$ (otherwise, if the last step is not $(\div\to)$, we have $H\to P = K\to \times(K)$, and this sequent obviously has a simple derivation). Applying the induction hypothesis, we obtain that there are simple derivations for sequents $H_i\to lab(d_i)$; each of such derivations is a sequence of applications of the rules $(\div\to)$ and $(\times\to)$. Now we construct a derivation of $H\to P$ from bottom to top as follows: firstly, we repeat the simple derivation of $H_1\to lab(d_1)$ (but now we consider $H_1$ to be a subgraph of $H$ and disregard $lab(d_i)$ in the succedent), then we repeat the simple derivation of $H_2\to lab(d_2)$ within $H$ and so on until $H_k\to lab(d_k)$. Now we have a sequent of the form $G[D/e][d_0:=\div(N/D)]\to P$ as a premise. Then we apply $(\div\to)$ to $\div(N/D)$ by ``overlaying'' the denominator on edges of $D$; thus, each premise except for the first one is of the form $\circledcirc(lab(d_i))\to lab(d_i)$, and the first one is $G\to P$. Now we can apply the induction hypothesis to $G\to P$, which shows that $G\to P$ can also be derived in the fashion stated in the lemma.
	\qed
\end{proof}
This theorem will be often used in less general cases, for instanse, when $P$ is primitive or there are no types with $\times$ in the antecedent.
\subsection{Equivalence of types}\label{sec_equiv}
In the string case, we say that types $A$ and $B$ are equivalent if $\mathrm{L}\vdash A\to B$ and $\mathrm{L}\vdash B\to A$. Similar definition can be introduced in $\mathrm{HL}$:
\begin{definition}
	Types $A$ and $B$ for which $type(A)=type(B)$ are equivalent ($A\sim B$) if $\mathrm{HL}\vdash \circledcirc(A)\to B$ and $\mathrm{HL}\vdash \circledcirc(B)\to A$.
\end{definition}
There are two simple observations regarding $\sim$:
\begin{proposition}\leavevmode
	\begin{enumerate}
		\item $\sim$ is an equivalence relation.
		\item If $C$ is a type with a primitive subtype $p$ which occurs in $C$ exactly once, and $A\sim B$, then $C[p\leftarrow A]\sim C[p\leftarrow B]$ ($C[p\leftarrow A]$ denotes substitution of $p$ in $C$ by $A$).
	\end{enumerate}
\end{proposition}
The second proposition says that we can replace equivalent subtypes by each other within a type preserving their equivalence.
\begin{proof}\leavevmode
	\begin{enumerate}
		\item Reflexivity follows from Proposition \ref{TtoT}; symmetry is obvious; transitivity can be proved using the cut rule, which is admissible according to Theorem \ref{cut}.
		\item This is proved by induction on depth of $p$ within $C$; the proof is similar to that of Proposition \ref{TtoT} with the only difference that we now apply Proposition \ref{TtoT} itself to all the premises except for one where we apply the induction hypothesis.
	\end{enumerate}
	\qed
\end{proof}
Both these propositions allow us to conclude that equivalent types are indistinguishable from the point of view of derivability in $\mathrm{HL}$.

Now we consider three simple but curious equivalences. Their proof is straightforward.
\begin{proposition}\label{elim_mult}\leavevmode
	\begin{enumerate}
		\item Let $\times(M)$ be a type and let $e_0\in E_M$ be labeled by $\times(L)$. Then $\times(M)\sim \times(M[L/e_0])$.
		\item Let $\div(N/D)$ be a type and let $e_0\in E_D$ be labeled by $\times(M)$. Then $\div(N/D)\sim \div(N/D[M/e_0])$.
		\item Let $\div(\div(N/D_1)/D_2)$ be a type and let $e_0\in D_1$ be labeled by $\$$. Then $\div(\div(N/D_1)/D_2)\sim \div(N/D_1[D_2/e_0])$.
	\end{enumerate}
\end{proposition}
Using these equivalences, we can eliminate subtypes of the form $\times(M)$ from denominators of divisions and from other multiplications, as well as simplify types constructed using two divisions in a row.
\begin{example}
	$\div\left(s\middle/\mbox{{\tikz[baseline=.1ex]{
				\node[] (R) {};
				\node[node,above=4.5mm of R] (N1) {};
				\node[node,below left=6mm and 6mm of N1] (N2) {};
				\node[node,below right=6mm and 6mm of N1] (N3) {};
				
				\draw[>=stealth,->,black] (N1) -- node[above left] {\small \$} (N2);
				\draw[>=stealth,->,black] (N2) -- node[below] {\small $\times((pq)^\bullet)$} (N3);
				\draw[>=stealth,->,black] (N3) -- node[above right] {\small $r$} (N1);
	}}}\right)
	\sim
	\div\left(s\middle/\mbox{{\tikz[baseline=.1ex]{
				\node[] (R) {};
				\node[node,above=4.5mm of R] (N1) {};
				\node[node,below left=6mm and 6mm of N1] (N2) {};
				\node[node,below right=6mm and 6mm of N1] (N3) {};
				\node[node,below right=6mm and 6mm of N2] (N4) {};
				
				\draw[>=stealth,->,black] (N1) -- node[above left] {\small \$} (N2);
				\draw[>=stealth,->,black] (N2) -- node[below left] {\small $p$} (N4);
				\draw[>=stealth,->,black] (N4) -- node[below right] {\small $q$} (N3);
				\draw[>=stealth,->,black] (N3) -- node[above right] {\small $r$} (N1);
	}}}\right).
	$
\end{example}
\section{Hypergraph Lambek Grammars}
Now we can define  notion of grammars based on the hypergraph Lambek calculus. As in the case of HRGs we consider an alphabet $\Sigma$ with a function $type:\Sigma\to \mathbb{N}$.
\begin{definition}
	A \emph{hypergraph Lambek grammar (HL-grammar, HLG)} is a tuple $HGr=\langle \Sigma, S, \triangleright\rangle$ where $\Sigma$ is a finite set (alphabet), $S\in Tp_{\mathrm{HL}}$ is a distinguished type, and $\triangleright\subseteq\Sigma\times Tp_{\mathrm{HL}}$ is a finite binary relation. Additionally, we require that $a\triangleright T$ implies $type(a)=type(T)$.
\end{definition}
We call the set $dict(HGr)=\{T\in Tp(\mathrm{HL}):\exists a: a\triangleright T\}$ a dictionary of $HGr$.
\begin{definition}
	\emph{The language $L(HGr)$ generated by a hypergraph Lambek grammar} $HGr=\langle \Sigma, S, \triangleright\rangle$ is the set of all hypergraphs $G\in\mathcal{H}(\Sigma)$ for which a function $f_G:E_G\to Tp(\mathrm{HL})$ exists such that:
	\begin{enumerate}
		\item $lab_G(e)\triangleright f_G(e)$ whenever $e\in E_G$;
		\item $\mathrm{HL}\vdash f_G(G)\to S$.
	\end{enumerate}
\end{definition}

\begin{example}
Consider an HLG $\langle\{a,b,c\},s,\triangleright\rangle$ where
\begin{itemize}
	\item $a\:\triangleright\: E_0=\div\left(s\middle/\mbox{{\tikz[baseline=.1ex]{
				\node[] (R) {};
				\node[node,above right=2mm and 0mm of R,label=above:{\scriptsize $(1)$}] (N1) {};
				\node[node,below=5.95mm of N1] (N2) {};
				\node[hyperedge,right=5.5mm of N1] (E) {$s$};
				\node[node,below=3.7mm of E] (N3) {};
				\node[node,right=5.5mm of E,label=above:{\scriptsize $(3)$}] (N4) {};
				\node[node,below=5.5mm of N4,label=below:{\scriptsize $(2)$}] (N5) {};

				\draw[>=stealth,->,black] (N1) -- node[left] {\small \$} (N2);
				\draw[>=stealth,->,black] (N2) -- node[below] {\small $p$} (N3);
				\draw[-,black] (N3) -- node[right] {\scriptsize 1} (E);
				\draw[-,black] (N1) -- node[above] {\scriptsize 2} (E);
				\draw[-,black] (N4) -- node[above] {\scriptsize 3} (E);
	}}}\right)$;
	\item $b\:\triangleright\: s$;
	\item $c\:\triangleright\: p$.
	\end{itemize}
	Then the following hypergraph belongs to the language generated by this grammar:
	
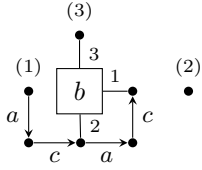
\begin{center}
	{\tikz[baseline=.1ex]{
			\node[node,label=above:{\scriptsize $(1)$}] (N1) {};
			\node[node,below=5.5mm of N1] (N2) {};
			\node[node,right=5.5mm of N2] (N3) {};
			\node[node,right=5.5mm of N3] (N5) {};
			\node[node,above=5.5mm of N5] (N4) {};
			\node[hyperedge,left=3.25mm of N4] (E) {$b$};
			\node[node,above=3.7mm of E,label=above:{\scriptsize $(3)$}] (N6) {};
			\node[node,right=6mm of N4,label=above:{\scriptsize $(2)$}] (N7) {};
			
			\draw[>=stealth,->,black] (N1) -- node[left] {\small $a$} (N2);
			\draw[>=stealth,->,black] (N2) -- node[below] {\small $c$} (N3);
			\draw[>=stealth,->,black] (N5) -- node[right] {\small $c$} (N4);
			\draw[>=stealth,->,black] (N3) -- node[below] {\small $a$} (N5);
			
			\draw[-,black] (N3) -- node[right] {\scriptsize 2} (E);
			\draw[-,black] (N6) -- node[right] {\scriptsize 3} (E);
			\draw[-,black] (N4) -- node[above] {\scriptsize 1} (E);
	}}
\end{center}
In order to show this we change labels by types corresponding to them (in this grammar this can be done uniquely), add a succedent $s$, and derive the resulting sequent:
$$
\infer[(\div\to)]{
\mbox{
	{\tikz[baseline=.1ex]{
		\node[node,label=above:{\scriptsize $(1)$}] (N1) {};
		\node[node,below=5.5mm of N1] (N2) {};
		\node[node,right=5.5mm of N2] (N3) {};
		\node[node,right=5.5mm of N3] (N5) {};
		\node[node,above=5.5mm of N5] (N4) {};
		\node[hyperedge,left=3.25mm of N4] (E) {$s$};
		\node[node,above=3.7mm of E,label=above:{\scriptsize $(3)$}] (N6) {};
		\node[node,right=6mm of N4,label=above:{\scriptsize $(2)$}] (N7) {};		
		\draw[>=stealth,->,black] (N1) -- node[left] {\small $E_0$} (N2);
		\draw[>=stealth,->,black] (N2) -- node[below] {\small $p$} (N3);
		\draw[>=stealth,->,black] (N5) -- node[right] {\small $p$} (N4);
		\draw[>=stealth,->,black] (N3) -- node[below] {\small $E_0$} (N5);	
		\draw[-,black] (N3) -- node[right] {\scriptsize 2} (E);
		\draw[-,black] (N6) -- node[right] {\scriptsize 3} (E);
		\draw[-,black] (N4) -- node[above] {\scriptsize 1} (E);
	}}
}\to s}{
\infer[(\div\to)]{
\mbox{
	{\tikz[baseline=.1ex]{
			\node[] (R) {};
			\node[node,above right=2mm and 0mm of R,label=above:{\scriptsize $(1)$}] (N1) {};
			\node[node,below=5.95mm of N1] (N2) {};
			\node[hyperedge,right=5.5mm of N1] (E) {$s$};
			\node[node,below=3.7mm of E] (N3) {};
			\node[node,right=5.5mm of E,label=above:{\scriptsize $(3)$}] (N4) {};
			\node[node,below=5.5mm of N4,label=below:{\scriptsize $(2)$}] (N5) {};
			\draw[>=stealth,->,black] (N1) -- node[left] {\small $E_0$} (N2);
			\draw[>=stealth,->,black] (N2) -- node[below] {\small $p$} (N3);
			\draw[-,black] (N3) -- node[right] {\scriptsize 1} (E);
			\draw[-,black] (N1) -- node[above] {\scriptsize 2} (E);
			\draw[-,black] (N4) -- node[above] {\scriptsize 3} (E);
	}}
}\to s}
{\circledcirc(s)\to s & \circledcirc(s)\to s & \circledcirc(p) \to p} & \circledcirc(s)\to s & \circledcirc(p) \to p
}
$$
\end{example}
Hypegraph Lambek grammars are graph grammars that generate hypergraph languages; thus they represent an alternative tool to HRGs. The most important issue regarding HLGs is describing the class of languages generated by them and comparing it with the class of languages generated by HRGs. Recall that in the string case the following theorem holds:
\begin{theorem}
	The class of languages generated by Lambek grammars coincides with the class of context-free languages without the empty word.
\end{theorem}
This theorem has two directions; the first one ($\mathrm{CFGs}\subseteq \mathrm{LGs}$) was proved by Gaifman in 1960 \cite{Bar_Hillel_Gaifman} while the other one ($\mathrm{LGs}\subseteq \mathrm{CFGs}$) was proved by Pentus in 1993 \cite{Pentus_cfg}. The first part is more simple; its proof is based on the Greibach normal form for context-free grammars. The second part appeared to be a hard problem; Pentus proved it using so-called free group interpretations and interpolants. 

Summing up, in the string case these two approaches are equivalent if we disregard such a nonsubstantive word as the empty word. Regarding the graph case, of course, our first expectation was that similar things happen: HRGs and HLGs are equivalent disregarding, possibly, some nonsubstantive cases. As in the string case, we introduced the analogue of the Greibach normal form for HRGs and studied how to convert these grammars into HLGs. However, this was not clear at all whether it is possible to perform the convertion of HLGs into equivalent HRGs: the proof of Pentus exploits free group interpretation, which is hard to generalize to graphs (we have no idea how to do this). Surprisingly, this convertion cannot be done at all! We figured out that hypergraph Lambek grammars generate a wider class of languages than HRGs. Moreover, for HLGs even the pumping lemma and the Parikh theorem do not hold. In Section \ref{sec_power} we study recognizing power of HLGs in detail and show that they are more powerful than HRGs.

Meanwhile, in the rest of this section we present and prove some closure properties of HLGs that will be used later. 

\begin{definition}
	Let $f:\Sigma\to\Delta$ be a relabeling function. If $H\in\mathcal{H}(\Sigma)$ is a graph, then we denote by $f(H)$ a graph obtained from $H$ by changing each label $a\in\Sigma$ by $f(a)$.
\end{definition}
\begin{proposition}\label{closure_relabeling}
	Languages generated by HLGs are closed under relabelings, i.e. if $L$ is a language over $\Sigma$ generated by an HLG, and $f:\Sigma\to \Delta$ is a relabeling function, then $\{f(H)|H\in L\}$ can be also generated by an HLG. 
\end{proposition}
\begin{proof}
	Let $HGr=\langle\Sigma, S,\triangleright\rangle$ be a grammar such that $L(HGr)=L$. Then it suffices to replace each relation of the form $a\triangleright T$ by a relation $f(a)\triangleright T$.
	\qed
\end{proof}
\begin{definition}
	Let $f:\Sigma\to\mathcal{H}(\Delta)$ be such a function that $type(a)=type(f(a))$ whenever $a\in\Sigma$ (we call it a graph-for-symbol substitution). If $H\in\mathcal{H}(\Sigma)$ is a graph such that $E_H=\{e_1,\dots,e_n\}$, then we denote by $f(H)$ a graph $H[f(lab(e_1))/e_1,\dots,f(lab(e_n))/e_n]$.
\end{definition}
We call such a substitution $f$ edgeful if for each $a\in\Sigma$ $f(a)$ contains at least one edge. Now we can formulate a closure property regarding such substitutions.
\begin{theorem}\label{closure_EGfSS}
	Let $HGr$ be an HLG such that types in its dictionary do not have skeleton subtypes; let also $f:\Sigma\to \mathcal{H}(\Delta)$ be an edgeful graph-for-symbol substitution. Then $\{f(H)|H\in L(HGr)\}$ can be generated by an HLG. 
\end{theorem}
\begin{proof}
	Denote components of $HGr$ as $HGr=\langle\Sigma,S,\triangleright\rangle$. For each graph $H\in f(\Sigma)$ we choose an arbitrary edge $e(H)$ within $E_H$ (note that this set is not empty). Besides, for each edge $e\in E_H$ except for $e(H)$ we introduce a new primitive type $p(e)$. Let a function $r_H:E_H\to Tp(\mathrm{HL})\cup\{\$\}$ be defined as follows: $r_H(e(H)):=\$$, and $r_H(e):=p(e)$ whenever $e\ne e(H)$. Now we present a new correspondence $\triangleright^\prime$. Let $b$ be in $\Sigma$, $H=f(b)$, $e$ belong to $E_H$, and $lab_H(e)=a$. 
	\begin{enumerate}
		\item If $e\ne e(H)$, then we say that $a \triangleright^\prime p(e)$;
		\item If $e=e(H)$, then for all types $T$ such that $b\triangleright T$ we say that $a\triangleright^\prime \div(T/r_H(H))$.
	\end{enumerate}
	Our aim is to prove that $HGr^\prime=\langle\Delta,S,\triangleright^\prime\rangle$ generates $\{f(H)|H\in L(HGr)\}$. Note that all new primitive types $p(e)$ are lonely in the dictionary of $HGr^\prime$, and note also that types in this dictionary do not have skeleton subtypes. Thus we can apply the wolf lemma (Corollary \ref{wolf}) and obtain that for $H$ being labeled by types of the dictionary of $HGr^\prime$ a sequent $H\to p(e)$ is derivable if and only if $H=\circledcirc(p(e))$.
	
	Let $G\to S$ be a derivable sequent where $G$ is over the dictionary of $HGr^\prime$. Let a type of the form $R=\div(T/r_H(H))$ be a label of $G$ where $H=f(b)$. Consider a part of a derivation of $G\to S$ where $R$ appears the first time:
	$$
	\infer[(\div\to)]{K\to S^\prime}{L\to S^\prime & H_1\to p(e_1) & \dots & H_k\to p(e_k)}
	$$
	Here $L$ is a graph with an edge $e^\prime$ labeled by $T$ and $K=L[r_H(H)/e^\prime][e(H):=R][H_1/e_1,\dots, H_k/e_k]$ where $E_H=E_{r_H(H)}=\{e(H),e_1,\dots,e_k\}$. Due to above reasonings $H_i=\circledcirc(p(e_i))$; hence $K=L[r_H(H)/e^\prime][e(H):=R]$. The latter means that $K$ is obtained from $L$ by inserting a relabeling of $H$ of the form $r_H(H)[e(H):=R]$. 
	
	In this step of a derivation new primitive types $p(e_1),\dots, p(e_k)$ appear within $K$. Observe that they cannot actively participate in further rules because in types of the grammar they either occur as separate types or label edges in denominators of types; the same holds with $R$ due to construction of types in the dictionary of $HGr^\prime$. Therefore, a subgraph $r_H(H)[e(H):=R]$ of a graph $K$, which appears on this step of a derivation, has to retain until the last step of a derivation. Finally, note that this subgraph can appear in $G$ only as a result of substitution of $H$ instead of $b$. Since $e^\prime$ is labeled by an old type $T$, after finding and eliminating all such substitutions we can derive a sequent in the old grammar. This yields that $L(HGr^\prime)\subseteq L(HGr)$.
	
	The other direction is more simple. Consider again a type $R=\div(T/r_H(H))$ as above. Note that the following derivation takes place:
	$$
	\infer[(\div\to)]{r_H(H)[e(H):=R]\to T}{
		\circledcirc(T)\to T & \circledcirc(p(e_1))\to p(e_1) & \dots & \circledcirc(p(e_k))\to p(e_k)}
	$$
	Now if $G\to S$ is a derivation in the old grammar, we can use the cut rule and replace each label $T$ of $G$ with a corresponding graph $r_H(H)[e(H):=R]$. A new sequent is considered to be a sequent in the new grammar as desired.
	\qed
\end{proof}
Note that the property of being edgeful is of importance since otherwise one would subsitute edgeless graphs for all symbols in a language generated by an HLG and obtain a language with infinitely many edgeless graphs; however, languages generated by HLGs can contain at most one edgeless graph. 

\section{Power of Hypergraph Lambek Grammars}\label{sec_power}
We start with showing that languages generated by HRGs can be generated by HLGs as well except for some nonsubstantive cases.

\subsection{Isolated-Node Boundedness}
Denote by $isize(H)$ the number of isolated nodes in $H$.
\begin{definition}
	A hypergraph language $L$ is isolated-node bounded (IB) if there is a constant $M>0$ such that for each $H\in L$ $isize(H)<M\cdot |E_H|$.
\end{definition}
It appears that each language generated by an HLG is isolated-node bounded disregarding at most one member of the language.
\begin{theorem}\label{inb}
	Each language $L$ generated by an HLG is of the form $L=L_0\cup IL$ where $L_0$ is isolated-node bounded and either $IL=\emptyset$ or $IL=\{I\}$ where $I$ is an edgeless graph.
\end{theorem}
\begin{proof}(of Theorem \ref{inb})
	Let us denote by $||T||_I$ the total number of isolated nodes within a type $T$. Formally, $||p||_I=0$ for $p\in Pr$; for $T=\div(N/D)$ where $E_D=\{d_0,d_1,\dots,d_k\}$ and $lab(d_i)=T_i\in Tp(\mathrm{HL}), i>0$ we say that $||T||_I=||N||_I+||T_1||_I+\dots+||T_k||_I+m$ where $m$ is the number of isolated nodes in $D$; for $T=\times(M)$ where $E_M=\{m_1,\dots,m_k\}$ and $lab(m_i)=T_i$ we say that $||T||_I=||T_1||_I+\dots+||T_k||_I+m$ where $m$ is the number of isolated nodes in $M$. Note that $||T||_I\ge |\#_\iota(iota(T))|$ for each type $T$ where $\iota$ and $iota$ are as in Proposition \ref{differ_IN}.
	
	Let $HGr=\langle \Sigma, S,\triangleright\rangle$ be an HLG generating $L$. Note that according to Proposition \ref{differ_IN} there is at most one edgeless graph $I$ such that $\mathrm{HL}\vdash I\to S$: indeed, $I$ has to contain $type(S)$ external nodes, and the number of nonexternal ones can be determined uniquely. It remains to show that $L_0=L\setminus \{I\}$ (or $L_0=L$, if such an edgeless graph does not exist) is isolated-node bounded.
	
	We define $C$ as $\max\{||T||_I: T\in dict(HGr)\cup\{S\}\}+type(S)+1$. We check Definition \ref{inb} with the constant $M=3C+1$. Let $H\in L$; then there is a relabeling $f:E_H\to Tp(\mathrm{HL})$ such that $lab_H(e)\triangleright f(e)$ for all $e\in E_H$, and $\mathrm{HL}\vdash f(H)\to S$; denote $G=f(H)$. Applying the construction from Proposition \ref{differ_IN} we obtain that $\#_\iota(iota(G))=\#_\iota(iota(S))$. Accordingly to the definition of $iota$ we have $\#_\iota(iota(G))=\#_\iota(iota(T_1))+\dots+\#_\iota(iota(T_k))+m$ where $m$ is the number of nonexternal isolated nodes in $G$, and $T_1,\dots,T_k$ are all labels in $G$ ($k=|E_H|$). Note that the number of external isolated nodes in $G$ does not exceed $type(S)$. Therefore, $isize(G)\le m+type(S)=\#_\iota(iota(S))-\#_\iota(iota(T_1))-\dots-\#_\iota(iota(T_k))+type(S)\le ||S||_I+||T_1||_I+\dots+||T_k||_I+type(S)\le M\cdot(k+2)< (3M+1)\cdot k=C\cdot |E_G|$. This completes the proof.
	\qed
\end{proof}
\begin{example}
	The language consisting of all edgeless graphs of type 0 (i.e. of graphs of the form $\langle \{v_1,\dots,v_n\},\emptyset,\emptyset,\emptyset, \emptyset\rangle$) can be generated by no HLG. However, it is simple to construct an HRG generating this language.
\end{example}
\subsection{Convertion of HRGs into HLGs}
Our goal is to study how to transform HRGs into equivalent HLGs. In order to do this we use the weak Greibach normal form for HRGs introduced in \cite{Pshenitsyn_GNFHRG}:
\begin{definition}
	An HRG $HGr$ is in \emph{the weak Greibach normal form} if there is exactly one terminal edge in the right-hand side of each production. Formally, $\forall (X\to H)\in P_{HGr}$ $ \exists!e_0\in E_H:lab_H(e_0)\in \Sigma_{HGr}$.
\end{definition}
In the paper \cite{Pshenitsyn_GNFHRG} we prove the following
\begin{theorem}\label{WGNF}
	For each HRG generating an isolated-node bounded language there is an equivalent HRG in the weak Greibach normal form.
\end{theorem}
Using it, we can prove the following theorem applying standard techniques.
\begin{theorem}\label{hrg_hlg}
	For each HRG generating an isolated-node bounded language there is an equivalent hypergraph Lambek grammar.
\end{theorem}
\begin{proof}
	Let an HRG be of the form $HGr=\langle N,\Sigma,P,S\rangle$. Applying Theorem \ref{WGNF} we can assume that $HGr$ is in the weak Greibach normal form.
	
	Consider elements of $N$ as elements of $Pr$ with the same function $type$ defined on them. Since $HGr$ is in the weak Greibach normal form, each production in $P$ is of the form $\pi=X\to G$ where $G$ contains exactly one terminal edge $e_0$ (say $lab_G(e_0)=a\in\Sigma$). We convert this production into the type $T_\pi:=\div(X/G[e_0:=\$])$. Then we introduce the HLG $HGr^\prime=\langle \Sigma, S,\triangleright\rangle$ where $\triangleright$ is defined as follows: $a\triangleright T_\pi$ (note that if $G=\circledcirc(a)$, then we can simply write $a\triangleright X$).
	The main objective is to prove that $L(HGr)=L(HGr^\prime)$. 
	
	Firstly, we show that $L(HGr)\subseteq L(HGr^\prime)$ by induction on size of a derivation in $HGr$. To be more technically sound we do this thoroughly (while omitting some tedious details in the second part of the proof).
	\\
	Induction basis. Let $S\Rightarrow H$ where $H\in\mathcal{H}(\Sigma)$. Then $\pi=S\to H$ belongs to $P$ and $E_H=\{e_0\}$. In this case we can derive $\mathrm{HL}\vdash H[e_0:=\div(S/H[e_0:=\$])]\to S$ in one step (since $|E_H|=1$).
	\\
	Induction step. Let $S\overset{k}{\Rightarrow} H$ where $H\in\mathcal{H}(\Sigma)$ (in this notation, induction is on $k$). There has to be a branch (called $\beta$) in this derivation of the form $G\Rightarrow G[F_0/e_0]\overset{l}{\Rightarrow} G[F_0/e_0][F_1/f_1,\dots,F_l/f_l]$ where $f_1,\dots,f_l$ are all nonterminal edges of $F_0$, and $F_1,\dots,F_l$ are terminal graphs ($l>0$). That is, we apply a production $lab(e_0)\to F_0$ and then $l$ productions that change all nonterminal edges of $F_0$ with terminal graphs.
	
	Let us introduce a production $\pi=lab(e_0)\to F$ where $F=F_0[F_1/f_1,\dots,F_l/f_l]$ is a terminal graph. Now we change a grammar $HGr$ and the derivation $S\overset{k}{\Rightarrow} H$ a bit: we add to $\Sigma$ a new terminal symbol $a_0$ ($type(a_0)=type(e_0)$), add a production $lab(e_0)\to \circledcirc(a_0)$ to $P$ and apply this production in the derivation instead of the branch $\beta$. A new derivation yields a graph $H^\prime$, which is related to $H$ as follows: $H=H^\prime[F/e_0], lab_{H^\prime}(e_0)=a_0$. In the new grammar (call it $HGr(\pi)$) $S$ derives $H^\prime$ in $(k-l)$ steps; this allows us to apply the induction hypothesis and to obtain that $H^\prime$ belongs to the language generated by an HLG (call it $HGr^\prime(\pi)=\langle \Sigma,S,\triangleright_\pi\rangle$) constructed from $HGr(\pi)$ in the same way as $HGr^\prime$ from $HGr$. This means that for $H^\prime$ there is such a relabeling $f:E_{H^\prime}\to Tp(\mathrm{HL})$ that $lab_{H^\prime}(e)\triangleright_\pi f(e)$, and $\mathrm{HL}\vdash f(H^\prime)\to S$. Since the only type corresponding to $a_0$ in $HGr^\prime(\pi)$ is $p_0=lab(e_0)$ (recall that it is considered to be primitive), $f(e_0)=p_0$. 
	\\
	Similarly to the induction basis, we notice that $\mathrm{HL}\vdash F_i[e_i:=T_i]\to lab(f_i)$ where $e_i$ is the only edge of $F_i$, and $T_i=\div(lab(f_i)/F_i[e_i:=\$])$; note that $lab(e_i)\triangleright T_i$. Using $(\div\to)$ we also derive $\mathrm{HL}\vdash g(F)\to p_0$ where $g:E_F\to Tp(\mathrm{HL})$ is a relabeling acting as follows:
	\begin{enumerate}
		\item $g(e_i)=T_i$ whenever $i>0$;
		\item For the only terminal edge $f_0$ of $F_0$ we put $g(f_0)=\div(p_0/F_0[f_0:=\$])$ (note that $lab(f_0)\triangleright g(f_0)$).
	\end{enumerate}
	Applying the cut rule we combine the sequent $g(F)\to p_0$ with the sequent $f(H^\prime)\to S$ and obtain a new derivable sequent $h(H)\to S$ where $h$ coincides with $g$ on $E_F$ and with $f$ on $E_{H^\prime}\setminus\{e_0\}$. The last thing we should notice is that $h$ is a relabeling of $H$ such that $lab_H(e)\triangleright h(e)$ whenever $e\in E_H$. This finishes the first part.
	
	Secondly, we explain why $L(HGr^\prime)\subseteq L(HGr)$. Note that types in the dictionary of $HGr^\prime$ are simple; thus for each derivable sequent of the form $H\to S$ over this dictionary we can apply Theorem \ref{simple_der} and obtain a derivation where each premise except for, possibly, the first one is an axiom. Now we can transform a derivation tree of $\mathrm{HL}$ into a derivation tree in the HRG $HGr$, which concludes the proof. Formally, we have to use induction again.
	\qed
\end{proof}
Now we leave these boring technical results and turn to the most interesting aspects of HLGs. Namely, we present several languages generated by HLGs that cannot be generated by HRGs. Each example will be presented in a separate subsection.
\subsection{2-Graphs Without Isolated Nodes}\label{sec_lan_all}
Consider the language $\mathcal{L}_1$ of all 2-graphs (i.e. usual graphs with edges of type 2) without isolated nodes (the empty graph is not included in $\mathcal{L}_1$ as well) over the alphabet $\{a\}$ ($type(a)=2$) without external nodes. This language intuitively seems to be very simple, but, astonishingly, there is no HRG generating it. This follows from the pumping lemma for HRGs, which implies that graph context-free languages are of bounded connectivity (see \cite{Drewes}). However, it is not hard to present an HLG that generates $\mathcal{L}_1$. Let $s,p$ be primitive types ($type(s)=0,type(p)=1$).
\begin{itemize}
	\item $Q_1=p,\;
		Q_2=\div\left(p\middle/\mbox{
		{\tikz[baseline=.1ex]{
			\node[] (R) {};
			\node[node,below=1mm of R,label=below:{\scriptsize $(1)$}] (N1) {};
			\node[hyperedge,above=5mm of N1] (E1) {$\$$};
			\node[node,right=7mm of N1] (N2) {};
			\node[hyperedge,above=5mm of N2] (E2) {$p$};
			
			\draw[-,black] (E1) -- node[left] {\scriptsize 1} (N1);
			\draw[-,black] (E2) -- node[right] {\scriptsize 1} (N2);	
			}}}\right),\;
		Q_3=\div\left(s\middle/\mbox{
			{\tikz[baseline=.1ex]{
					\node[] (R) {};
					\node[node,below=1mm of R] (N1) {};
					\node[hyperedge,above=5mm of N1] (E1) {$\$$};
					\node[node,right=7mm of N1] (N2) {};
					\node[hyperedge,above=5mm of N2] (E2) {$p$};
					
					\draw[-,black] (E1) -- node[left] {\scriptsize 1} (N1);
					\draw[-,black] (E2) -- node[right] {\scriptsize 1} (N2);	
		}}}\right);$
	\item $M_{11}^{ij}= \times\left(\mbox{
		{\tikz[baseline=.1ex]{
				\node[] (R) {};
				\node[node,below=1mm of R,label=below:{\scriptsize $(1)$}] (N1) {};
				\node[hyperedge,above=5mm of N1] (E1) {$Q_i$};
				\node[node,right=7mm of N1,label=below:{\scriptsize $(2)$}] (N2) {};
				\node[hyperedge,above=5mm of N2] (E2) {$Q_j$};
				
				\draw[-,black] (E1) -- node[left] {\scriptsize 1} (N1);
				\draw[-,black] (E2) -- node[right] {\scriptsize 1} (N2);
		}}
	}\right),\;
	M_{12}^{i}= \times\left(\mbox{
		{\tikz[baseline=.1ex]{
				\node[] (R) {};
				\node[node,below=1mm of R,label=below:{\scriptsize $(1)$}] (N1) {};
				\node[hyperedge,above=5mm of N1] (E1) {$Q_i$};
				\node[node,right=7mm of N1,label=below:{\scriptsize $(2)$}] (N2) {};
				
				\draw[-,black] (E1) -- node[left] {\scriptsize 1} (N1);
		}}
	}\right),\\
	M_{21}^{j}= \times\left(\mbox{
		{\tikz[baseline=.1ex]{
				\node[] (R) {};
				\node[node,below=1mm of R,label=below:{\scriptsize $(1)$}] (N1) {};
				\node[node,right=7mm of N1,label=below:{\scriptsize $(2)$}] (N2) {};
				\node[hyperedge,above=5mm of N2] (E2) {$Q_j$};
				
				\draw[-,black] (E2) -- node[right] {\scriptsize 1} (N2);
		}}
	}\right),\;
	M_{22}= \times\left(\mbox{
		{\tikz[baseline=.1ex]{
				\node[] (R) {};
				\node[node,above=0mm of R,label=below:{\scriptsize $(1)$}] (N1) {};
				\node[node,right=7mm of N1,label=below:{\scriptsize $(2)$}] (N2) {};
		}}
	}\right).$
\end{itemize}
A desired grammar is of the form $HGr_1=\langle\{a\}, s, \triangleright \rangle$ and  $a\triangleright N$ whenever $N\in\{M_{11}^{ij},M_{12}^{i},M_{21}^{j},M_{22}|1\le i,j\le 3\}$.
\begin{proposition}\label{prop_lan_all}
	$L(HGr_1)=\mathcal{L}_1$.
\end{proposition}
\begin{proof}
To prove that $L(HGr_1)\subseteq\mathcal{L}_1$ it suffices to note that denominators of types in $dict(HGr_1)$ do not contain isolated nodes; since isolated nodes may appear only after applications of rules $(\div\to)$ or $(\to\times)$, all graphs in $L(HGr_1)$ do not have them.

The other inclusion $L(HGr_1)\supseteq\mathcal{L}_1$ is of central interest. An example of a specific derivation in this grammar is given in Appendix \ref{sec_lan_all_example}. Below we provide general reasonings of this inclusion, but we suppose that this example is enough to understand the construction of $HGr_1$.

Let $H$ be in $\mathcal{L}_1$. Since there are no isolated nodes in $H$ there exists a function $h:V_H\to E_H$ such that $h(v)$ is attached to $v$ whenever $v\in V_H$. We choose two arbitrary nodes $v_e$ and $v_b$ and a define a function $c:V_H\to \{1,2,3\}$ as follows: $c(v_b)=1$, $c(v_e)=3$, $c(v)=2$ whenever $v\not\in\{v_b,v_e\}$.

Now we present a relabeling $f_H:E_H\to Tp(\mathrm{HL})$. Let $e$ belong to $E_H$ and let $att_H(e)=v_1v_2$.
\begin{itemize}
	\item If $h(v_1)=h(v_2)=e$, then $f_H(e):=M_{11}^{c(v_1)c(v_2)}$;
	\item If $h(v_1)=e,h(v_2)\ne e$, then $f_H(e):=M_{12}^{c(v_1)}$;
	\item If $h(v_1)\ne e,h(v_2)= e$, then $f_H(e):=M_{21}^{c(v_2)}$;
	\item If $h(v_1)\ne e,h(v_2)\ne e$, then $f_H(e):=M_{22}$.
\end{itemize} 
Then we check derivability of the sequent $f_H(H)\to s$. Its derivation from bottom to top starts with rules $(\times\to)$ applied $|E_H|$ times to all types in the antecedent. It turns out that the  sequent standing above these applications of $(\times\to)$ has one edge labeled by $Q_1$, one edge labeled by $Q_3$ and the remaining edges labeled by $Q_2$; besides, for each node there is exactly one edge attached to it (this is satisfied by the definition of the function $h$). Then we apply (again from bottom to top) the rule $(\div\to)$ and using it ``reduce'' the only $Q_1$-labeled edge (recall that $Q_1=p$) with a $Q_2$-labeled edge; after this we obtain a new $p$-labeled edge and repeat the procedure. Thus we eliminate all nodes and edges one-by-one. Finally, we obtain a graph with two nodes, with a $Q_3$-labeled edge attached to the first one and a $p$-labeled edge attached to the second one. Applying $(\div\to)$ once more, we ``contract'' $Q_3$ with $p$ and obtain the sequent $\circledcirc(s)\to s$, which is an axiom.
\qed
\end{proof}
Therefore we have already shown that hypergraph Lambek grammars based on $\mathrm{HL}$ are stronger than HRGs (thus Pentus theorem cannot be generalized to HL) and that they moreover disobey the pumping lemma. One would say that the secret is in types with $\times$, which play a central role in $HGr_1$. However, we argue that $HGr_1$ can be modified in such a grammar $HGr_1^\prime$ that its types shall not contain $\times$, but $L(HGr_1^\prime)=\mathcal{L}_1$ as well. In order to do this we present a function $U$, which operates on types $T$ such that $type(T)=2$ as follows:
$$
U(T)=\div\left(s\middle/\mbox{
	{\tikz[baseline=.1ex]{
			\node[] (R) {};
			\node[node] (N1) {};
			\node[node,right=7mm of N1] (N2) {};
			
			\draw[>=stealth,->,black] (N1) to[bend left=45] node[above] {\$} (N2);
			\draw[>=stealth,->,black] (N1) to[bend right=45] node[below] {$T$} (N2);
	}}
}\right).
$$
Then $HGr_1^\prime=\langle\{a\},s,\triangleright^\prime\rangle$ is defined by the following relation: $a\triangleright T \Leftrightarrow a\triangleright^\prime U(U(T))$. Of course, this transformation itself does not eliminate $\times$, it just places $\times$ in denominators of types. Nevertheless, this transformation is a desired one since we can use Proposition \ref{elim_mult} and change types in $dict(HGr_1^\prime)$ with equivalent ones without multiplication. Now it remains to prove
\begin{proposition}
	$L(HGr_1^\prime)=\mathcal{L}_1$.
\end{proposition}
\begin{proof}
	Firstly we note that $\mathrm{HL}\vdash \circledcirc(T)\to U(U(T))$; thus if a sequent $G^\prime\to s$ over $dict(HGr_1^\prime)$ is derivable, then we can use the cut rule and derive a sequent $G\to s$ where $G$ is obtained from $G^\prime$ by changing each type of the form $U(U(T))$ with $T$. This justifies that $L(HGr_1^\prime)\subseteq L(HGr_1)$.
	
	To prove $L(HGr_1^\prime)\supseteq L(HGr_1)$ it is enough to recreate a derivation described in Proposition \ref{prop_lan_all} using new types. We remodel rule $(\times\to)$ applications within a derivation as follows:
	$$
	\infer[(\times\to)]{G\llbracket e_0, T/F\rrbracket\to s}{G\to s}
	\quad\rightsquigarrow\quad
	\infer[(\div\to)]{G\llbracket e_0,U(U(T))/F\rrbracket\to s}{s\to s & \infer[(\to\div)]{G^\prime\to U(T)}
	{
		G\to s
	}
	}
	$$
	Here $F$ is a subgraph of $G$, $T=\times(F)$, and if we denote $G\llbracket e_0,U(U(T))/F\rrbracket$ by $H$, then $G^\prime=\langle V_H, E_H\setminus\{e_0\},att_H|_{E_{G^\prime}},lab_H|_{E_{G^\prime}},att_H(e_0)\rangle$. It is not hard to see that a new derivation is correct.
	\qed
\end{proof}
Therefore, even $\mathrm{HL}(\div)$-grammars can produce non-context-free graph languages.
\subsection{Bipartite graphs}
Another example is the language $\mathcal{L}_2\subseteq \mathcal{L}_1$ of all bipartite 2-graphs without isolated nodes. In this example, we call a graph $H$ bipartite if its nodes can be divided into two subsets $V_1$ and $V_2$ in such a way that each edge of $H$ outgoes from a node belonging to $V_1$ to a node belonging to $V_2$. 

Let us define the following types (where $p,q$ are primitive, $type(p)=type(q)=1$):
\begin{itemize}
	\item $R_1(r):=r$;
	\item $R_2(r):=\div\left(r\middle/\mbox{
		{\tikz[baseline=.1ex]{
				\node[] (R) {};
				\node[node,below=0mm of R,label=below:{\scriptsize $(1)$}] (N1) {};
				\node[hyperedge,above left = 3mm and 2mm of N1] (E1) {\$};
				\node[hyperedge,above right = 3mm and 2mm of N1] (E2) {$r$};
				\draw[-,black] (E1) -- node[below left] {\scriptsize 1} (N1);
				\draw[-,black] (E2) -- node[below right] {\scriptsize 1} (N1);
		}}
	}\right)$;
	\item $R_3(r):=\div\left(r\middle/\mbox{
		{\tikz[baseline=.1ex]{
			\node[] (R) {};
			\node[node,below=1mm of R,label=right:{\scriptsize $(1)$}] (N1) {};
			\node[hyperedge,above = 3mm of N1] (E1) {\$};
			\node[node,right=10mm of N1] (N2) {};
			\node[hyperedge,above = 3mm of N2] (E2) {$r$};
			\draw[-,black] (E1) -- node[left] {\scriptsize 1} (N1);
			\draw[-,black] (E2) -- node[left] {\scriptsize 1} (N2);
		}}
	}\right)$;
	\item $M^{ij}:=\times\left(\mbox{
		{\tikz[baseline=.1ex]{
				\node[] (R) {};
				\node[node,below=1mm of R,label=below:{\scriptsize $(1)$}] (N1) {};
				\node[hyperedge,above = 3mm of N1] (E1) {$\,R_i(p)\,$};
				\node[node,right=10mm of N1,label=below:{\scriptsize $(2)$}] (N2) {};
				\node[hyperedge,above = 3mm of N2] (E2) {$\,R_j(q)\,$};
				\draw[-,black] (E1) -- node[left] {\scriptsize 1} (N1);
				\draw[-,black] (E2) -- node[left] {\scriptsize 1} (N2);
		}}
	}\right),\quad 1\le i,j\le 3;$
	\item $S:=\times\left(\mbox{
		{\tikz[baseline=.1ex]{
				\node[] (R) {};
				\node[node,below=1mm of R] (N1) {};
				\node[hyperedge,above = 3mm of N1] (E1) {$p$};
				\node[node,right=10mm of N1] (N2) {};
				\node[hyperedge,above = 3mm of N2] (E2) {$q$};
				\draw[-,black] (E1) -- node[left] {\scriptsize 1} (N1);
				\draw[-,black] (E2) -- node[left] {\scriptsize 1} (N2);
		}}
	}\right).$
\end{itemize}
We define $HGr_2:=\langle\{a\},S,\triangleright\rangle$ as follows: $a\triangleright M^{ij}$ for all $1\le i,j\le 3$.
\begin{proposition}
	$\mathcal{L}_2=L(HGr_2)$.
\end{proposition}
\subsection{Finite Intersections of HCFLs}\label{sec_FIHCFL}
Once Stepan Kusnetsov who is a Russian mathematician doing research regarding the Lambek calculus and its variants in Steklov Mathematical Institute and in Moscow State University delivered a talk where he mentioned the following concept: multiplication in L (i.e. an operation $A\cdot B$) may be considered as some kind of conjunction of $A$ and $B$ when we have both $A$ and $B$ combined in a single type. This analogy with conjunction led us to the following thought. In the graph case we can use multiplication (i.e. $\times$) in a more general way than for strings: any graph structure can be put inside $\times$. What if there is a way to use $\times$ as conjunction and thus model intersections of languages?

Firstly, we invented a way to show that any finite intersection of string context-free languages (considered as a graph language) can be generated by an HL-grammar; then we realized that this construction can be generalized to all hypergraph context-free languages. Below we present this construction.

\begin{definition}\label{def_ersatz}
	An ersatz conjunction $\wedge_E(T_1,\dots,T_k)$ of types $T_1,\dots,T_k\in Tp(\mathrm{HL})$ such that $type(T_1)=\dots=type(T_k)=m$ is the type $\times(H)$ where
	\begin{enumerate}
		\item $V_H=\{v_1,\dots,v_m\}$;
		\item $E_H=\{e_1,\dots,e_k\}$;
		\item $att_H(e_i)=v_1\dots v_m$;
		\item $lab_H(e_i)=T_i$;
		\item $ext_H=v_1\dots v_m$.
	\end{enumerate}
\end{definition}
\begin{example}
	Let $T_1,T_2,T_3$ be types with $type$ equal to 2. Then their ersatz conjunction equals $\wedge_E(T_1,T_2,T_3)=\times\left(\mbox{
		{\tikz[baseline=.1ex]{
				\node[node,label=left:{\scriptsize $(1)$}] (N1) {};
				\node[node,right=18mm of N1,label=right:{\scriptsize $(2)$}] (N2) {};
				\draw[>=stealth,->,black] (N1) to[bend left = 60] node[above] {$T_1$} (N2);
				\draw[>=stealth,->,black] (N1) -- node[above] {$T_2$} (N2);
				\draw[>=stealth,->,black] (N1) to[bend right = 60] node[above] {$T_3$} (N2);
		}}
	}\right)$.
\end{example}
Using ersatz conjunction we can prove the following
\begin{theorem}\label{Theorem_intersections}
	If $HGr^\prime_1,\dots,HGr^\prime_k$ are HRGs generating isolated-node bounded languages, then there is an HL-grammar $HGr$ such that $L(HGr)=L(HGr^\prime_1)\cap\dots\cap L(HGr^\prime_k)$.
\end{theorem}
This may be considered as the main result of this section.
\begin{proof}
	Using the construction from Theorem \ref{hrg_hlg} we construct an HL-grammar $HGr_i$ for each $i=1\dots,k$ such that $L(HGr_i^\prime)=L(HGr_i)$. We assume without loss of generality that types involved in $HGr_i$ and $HGr_j$ for $i\ne j$ do not have common primitive subtypes (let us denote the set of primitive subtypes of types in $HGr_i$ as $Pr_i$). Let us denote $HGr_i = \langle\Sigma, s_i,\triangleright_i\rangle$. Note that $type(s_1)=\dots=type(s_k)$ (otherwise $L(HGr_1)\cap\dots \cap L(HGr_k)=\emptyset$, and the theorem holds due to trivial reasons). The main idea then is to do the following: given $a\triangleright_i T_i,\; i=1,\dots, k$ we join $T_1,\dots,T_k$ using ersatz conjunction; we also join $s_1,\dots s_k$ using it. Then a derivation is expected to split into $k$ independent parts corresponding to derivations in grammars $HGr_1,\dots, HGr_k$. However, there is a nuance that spoils simplicity of this idea; it is related to the issue of isolated nodes. This nuance leads to a technical trick, which we call ``tying balloons''.
	
	Let us fix $k$ new primitive types $b_1,\dots,b_{k-1}$ (``balloon'' labels) such that $type(b_i)=1$. For $j<k$ let us define a function $\varphi_j:dict(HGr_j)\to Tp(\mathrm{HL})$ as follows: $\varphi_j(p)=p$ whenever $p\in Pr$; $\varphi_j(\div(p/D))=\div(\times(M)/D^\prime)$ where
	\begin{enumerate}
		\item $D^\prime=\langle V_D,E_D,att_D,lab_D,ext_Dw\rangle$ where $[w]=V_D\setminus [ext_D]$ (that is, $w$ consists of nodes that are not external in $D$).
		\item Denote $m=|w|=|V_D|-|ext_D|$, and $t = type(p)$. Then $M=\langle \{v_1,\dots,v_{t+m}\},\\\{e_0,e_1,\dots,e_m\}, att,lab, v_1,\dots v_{t+m} \rangle$ where $att(e_0)=v_1\dots v_t$, $lab(e_0)=p$; $att(e_i)=v_{t+i}$, $lab(e_i)=b_j$ whenever $i=1\dots, m$.
	\end{enumerate}
	Informally, we thus make all nodes in the denominator $D$ external, while $\times(M)$ ``ties a balloon'' labeled $b_j$ to each node corresponding to a nonexternal one in $D$. Presence of these ``balloon edges'' is compensated by modified types of the grammar $HGr_k$. Namely, we define a function $\varphi_k:dict(HGr_k)\to Tp(\mathrm{HL})$ as follows: $\varphi_k(p)=p$ whenever $p\in Pr$; $\varphi_k(\div(p/D))=\div(p/D^\prime)$ where $D^\prime=\langle V_D,E_D\cup \{e_1,\dots, e_{(k-1)m}\},att,lab,ext_D\rangle$  such that:
	\begin{enumerate}
		\item $m=|V_D|-|ext_D|$;
		\item $e_1,\dots,e_{(k-1)m}$ are new edges;
		\item $att|_{E_D}=att_D$;
		\item If $v_1,\dots,v_m$ are all nonexternal nodes of $D$, then $att(e_i)=v_{\lceil i/(k-1)\rceil}$ for $i=1,\dots, (k-1)m$. In other words, we attach $(k-1)$ new edges to each nonexternal node of $D$.
		\item $lab(e_i)=b_{g(i)},\; i=1,\dots, (k-1)m$ where $g(i) = i \mod (k-1)$ if $(k-1)\nmid i$ and $g(i)=k-1$ otherwise. That is, for each $b_i,i=1,\dots,(k-1)$ and for each nonexternal node there is a $b_i$-labeled edge attached to it.
	\end{enumerate}
	Now we are ready to introduce $HGr$: $HGr=\langle\Sigma,S,\triangleright\rangle$ where
	\begin{itemize}
		\item $a\triangleright T \Leftrightarrow T=\wedge_E(\varphi_1(T_1),\dots, \varphi_k(T_k)) \mbox{ and } \forall i =1,\dots,k\; a\triangleright_i T_i$;
		\item $S=\wedge_E(s_1,\dots,s_k)$.
	\end{itemize}
	The proof of $L(HGr)=L(HGr_1)\cap\dots\cap L(HGr_k)$ is divided into two parts: the $\subseteq$-inclusion proof and the $\supseteq$-inclusion proof.
	
	\textbf{Proof of $L(HGr)\supseteq L(HGr_1)\cap\dots\cap L(HGr_k)$.} A hypergraph $H\in\mathcal{H}(\Sigma)$ belongs to $L(HGr_1)\cap\dots\cap L(HGr_k)$ if and only if there are relabeling functions $f_i:E_{H}\to Tp(\mathrm{HL})$ such that $lab_H(e)\triangleright_i f_i(e)$ for all $e\in E_H$, and $\mathrm{HL}\vdash f_i(H)\to s_i$. Using these relabelings we can construct a relabeling $f:E_H\to Tp(\mathrm{HL})$ as follows: if $f_i(e)=T_i$, then $f(e):=\wedge_E(\varphi_1(T_1),\dots,\varphi_k(T_k))$. It follows directly from the definition that $lab_H(e)\triangleright f(e)$. Now we construct a derivation of $f(H)\to \wedge_E(s_1,\dots,s_k)$ from bottom to top:
	\begin{enumerate}
		\item We apply rules $(\times\to)$ to all ersatz conjunctions in the antecedent (this yields a graph with $k$ ``layers'');
		\item\label{FIT_remodel} We remodel a derivation of $f_1(H)\to s_1$, which consists of $(\div\to)$-applications only, using types of the form $\varphi_1(f_1(e)),e\in E_H$ that are present in $f(H)$; the only difference now is that external nodes do not ``disappear'' (recall that a derivation is considered from bottom to top), and edges labeled by types with $\times$ appear. Every time when $\times$ appears in the left-hand side we immediately apply $(\times\to)$, which results in adding an edge labeled by a primitive type and in adding balloon edges to all nodes that would disappear in the derivation of $f_1(H)\to s_1$. 
		
		The result of this procedure is that now all types corresponding to $HGr_1$ left the antecedent, except for the only $s_1$-labeled edge attached to external nodes in the right order; besides, for each nonexternal node of the antecedent there is now a balloon edge labeled by $b_1$ attached to it.
		\item We perform $(k-2)$ more steps similarly to Step \ref{FIT_remodel} using types of the form $\varphi_i(f_i(e)), 1<i<k$ and thus remodeling a derivation $f_i(H)\to s_i$. Upon completion of all these steps the antecedent contains:
		\begin{itemize}
			\item Types of the form $\varphi_k(f_k(e)),e\in E_H$;
			\item $(k-1)$ edges labeled by $s_1,\dots,s_{k-1}$ resp. and attached to external nodes of the graph;
			\item Balloon edges such that for each $j\in\{1,\dots,k-1\}$ and for each nonexternal node there is a $b_j$-labeled edge attached to it.
		\end{itemize}
		\item We remodel a derivation of $f_k(H)\to s_k$ using types of the form $\varphi_k(f_k(e))$; a situation differs from previous ones because now nonexternal nodes do disappear, and each time when this happens all balloon edges attached to a nonexternal node disappear as well. 
		
		After this step, all balloon edges are removed, and we obtain a graph with $type(s_1)$ nodes such that all of them are external, and with $k$ edges labeled by $s_1,\dots, s_k$ such that their attachment nodes coincide with external nodes of the graph. This ends the proof since $\wedge_E(s_1,\dots,s_k)$ is exactly this graph standing under $\times$.
	\end{enumerate}
	\textbf{Proof of $L(HGr)\subseteq L(HGr_1)\cap\dots\cap L(HGr_k)$.}
	Let $H$ be in $L(HGr)$; then there is a function $\Phi:E_H\to Tp(\mathrm{HL})$ such that $\Phi(e)=\wedge_E(\varphi_1(T_1(e)), ,\dots,\varphi_k(T_k(e)))$ (whenever $e\in E_H$), $lab(e)\triangleright_i T_i(e)$, and $\Phi(H)\to S$ is derivable in HL. Our desire is to decompose this derivation into $k$ ones in grammars $HGr_1,\dots,HGr_k$. In order to do this we transform the derivation in stages:
	
	\textbf{Stage 1.} Using Proposition \ref{reversibility} we can replace every edge in $\Phi(H)$ labeled by a type of the form $\times(M)$ with $M$. A new sequent (denote it by $H^\prime\to S$) is derivable as well. Let us fix some its derivation.
	
	\textbf{Stage 2.} The sequent $H^\prime\to S$ fits in Theorem \ref{simple_der}; hence there exists its simple derivation. Let us fix some simple derivation of $H^\prime\to S$ and call it $\Delta$.
	
	Furthermore we consider $\Delta$ from bottom to top.
	
	\textbf{Stage 3.} Design of types $\varphi_i(T)$ differs in the case $i<k$ and $i=k$. Consequently, if $\varphi_i(T)$ for $i<k$ participates in the rule $(\div\to)$ in $\Delta$, this affects only primitive types from $Pr_i$; on the contrary, participating of $\varphi_k(T)$ in $(\div\to)$ affects types from $Pr_k$ but also balloon types $b_1,\dots,b_{k-1}$, which appear after rule applications of $(\div\to)$ and $(\times\to)$ to several types of the form $\varphi_i(T),\;i<k$. This allows us to come up with the following conclusion: if a rule application $(\div\to)$ to a type of the form $\varphi_k(T)$ preceeds (from bottom to top) a rule application of $(\div\to)$ to a type of the form $\varphi_i(T)$ for $i<k$, then we can change their order (note also that all nodes in the denominator of $\varphi_i(T)$ are external). Thus $\Delta$ can be remade in such a way that all rules affecting $\varphi_k(T)$ will occur upper than rules affecting $\varphi_i(T),i<k$ in a derivation. Let us call a resulting derivation $\Delta^\prime$.
	
	\textbf{Stage 4.} A denominator of a type $\varphi_i(T)$ for $i<k$ contains edges labeled by elements of $Pr_i$ only. Since $\Delta^\prime$ is simple, applications of the rule $(\div\to)$ to types of the form $\varphi_i(T)$ and $\varphi_j(T^\prime)$ for $i\ne j$ are independent, and their order can be changed. This means that we can reorganize $\Delta^\prime$ in the following way (from bottom to top):
	\begin{enumerate}
		\item\label{delta_1} Set $i=1$;
		\item\label{delta_2} Perform applications of the rule $(\div\to)$ to types of the form $\varphi_i(T)$ and right away of the rule $(\times\to)$ to their numerators;
		\item\label{delta_3} If $i=k-1$, go forward; otherwise, set $i=i+1$ and go back to the previous step;
		\item\label{delta_4} Perform applications of the rule $(\div\to)$ to types of the form $\varphi_k(T)$;
		\item\label{delta_5} Now an antecedent of the major sequent (denote this sequent as $G\to S$) does not include types with $\div$ or $\times$. $S$ is of the form $\times(M_S)$, and Theorem \ref{simple_der} provides that the last rule applied has to be $(\to\times)$; therefore, $G=M_S$ and we reach the sequent $M_S\to S$. Consequently, $G=M_S$ consists of $k$ edges labeled by $s_1$, $\dots$, $s_k$ resp.
	\end{enumerate}
	Let us call this derivation $\Delta_0$. Observe that after steps \ref{delta_1}-\ref{delta_3} of the above description balloon edges with all labels $b_1,\dots, b_{k-1}$ may occur in the antecedent of a sequent (denote this sequent as $G^\prime\to S$). There is only one way for them to disappear: they have to participate in the rule $(\div\to)$ with a type of the form $\varphi_k(T)$ (since the denominator of such a type generally contains balloon edges). Note however that balloon edges within the denominator of $\varphi_k(T)$ may be attached only to nonexternal nodes. Therefore balloon edges in $G^\prime$ can be attached only to nonexternal nodes as well. Besides, if some balloon edge labeled by $b_i$ is attached to a node $v\in V_{G^\prime}\setminus [ext_{G^\prime}]$, then the set of balloon edges attached to $v$ has to consist of exactly $k-1$ edges labeled by $b_1,\dots,b_{k-1}$ (because in the denominator of $\varphi_k(T)$ exactly such edges are attached to each nonexternal node). Finally, note that after step \ref{delta_4} all nonexternal nodes disappear since $M_S$ contains exactly $type(S)$ nodes, all of which are external, therefore balloon edges have to be present on all nonexternal nodes (otherwise, a nonexternal node cannot go away interacting with a type of the form $\varphi_k(T)$). 
	
	Summarizing all the above observations, we conclude that after steps \ref{delta_1}-\ref{delta_3} there is exactly one balloon edge labeled by $b_i$ on each nonexternal node of $G^\prime$ for all $i=1,\dots,k-1$ (and no balloon edge is attached to some external node of $G^\prime$). Since the only way for $b_i$ to be attached to a node is to participate in the rule $(\div\to)$ applied to a type of the form $\varphi_i(T)$, now it is quite clear how to decompose this derivation into $k$ ones:
	\begin{itemize}
		\item For $1\le i < k$ we consider step \ref{delta_2} of $\Delta_0$ with that only difference that we disregard balloon edges. Then the combination of rules $(\div\to)$ and $(\times\to)$ applied to a type $\varphi_i(T)$ turns into an application of the rule $(\div\to)$ to $T$ in the $HGr_i$. Take into account that the only type that is built of elements of $Pr_i$ and remains to step \ref{delta_5} is $s_i$ attached to external nodes in the right order; therefore, if we remove from $H^\prime$ all edges not related to $HGr_i$ and relabel ech edge labeled by $\varphi_i(T)$ with $T$ (call the resulting graph $H^\prime_i$), then $H^\prime_i\to s_i$ is derivable.
		\item For $i=k$ everything works similarly; however, instead of step \ref{delta_2} we have to look at step \ref{delta_4} and again not to consider balloon edges. Then the application of $(\div\to)$ to $\varphi_k(T)$ transforms into the similar application of $(\div\to)$ to $T$ in $HGr_k$. After the whole process only $s_k$ remains, so if $H^\prime_k$ is a graph obtained from $H^\prime$ by removing edges not related to $HGr_k$ and changing each label of the form $\varphi_k(T)$ by $T$, then $H^\prime_k\to s_k$ is derivable.
	\end{itemize}
	Finally note that $H^\prime_i=\Phi_i(H)$ where $\Phi_i(e)=T_i(e)$. The requirement $lab(e)\triangleright_i T_i(e)$ completes the proof.
	\qed
\end{proof}
This theorem has a number of important consequences:
\begin{corollary}
	There is an HL-grammar generating the language of string graphs $\{(a^{2n^2})^\bullet,n>1\}$.
\end{corollary}
\begin{proof}
	The string language $L_1=\{(a^nb^n)^k|n,k>0\}$ is context-free: it can be generated by a grammar with productions $S\to SS$, $S\to T$, $T\to aTb$, $T\to ab$. The string language $L_2=\{a^k(b^na^n)^lb^k|k,l,n>0\}$ is context-free as well: it can be generated by a grammar with productions $S\to aSb$, $S\to Q$, $Q\to QQ$, $Q\to T$, $T\to bTa$, $T\to ba$. Consequently, languages $L_1^\bullet=\{w^\bullet|w\in L_1\}$ and $L_2^\bullet$ are generated by some HRGs. The language $L_3=L_1\cap L_2$ equals $L_3=\{((a^nb^n)^n)^\bullet|n>1\}$, so $L_3^\bullet$ is a finite intersection of HCFLs and can be generated by some HL-grammar. Using Proposition \ref{closure_relabeling} (where the relabeling is $f:\{a,b\}\to \{a\}$) we conclude that $L=\{(a^{2n^2})^\bullet,n>1\}$ can be generated by an HL-grammar.
	\qed
\end{proof}
\begin{corollary}
	The pumping lemma and the Parikh theorem do not hold for languages generated by HL-grammars. If we take the set $NE=\{n|\exists H\in L(HGr): |E_H|=n\}$ for $HGr$ being an HL-grammar and introduce an increasing sequence $a_k,k\ge 0$ such that $a_k<a_{k+1}$, and $NE=\{a_k,k\ge 0\}$), then $a_k=O(k)$ does not hold in general.
\end{corollary}
\begin{proof}
	The language $\{(a^{2n^2})^\bullet,n>1\}$ is a counterexample to the pumping lemma, the Parikh theorem and to the statement that $a_k=O(k)$.
	\qed
\end{proof}
Since the issue of intersections is raised, one would ask whether every language generated by an HL-grammar can be obtained from HCFLs by intersections and relabelings (and, possibly, some other set-theoretical operations). However, the example from Section \ref{sec_lan_all} contradicts this. Any (possibly not finite) intersection of HCFLs, which are languages of bounded connectivity, is a language of bounded connectivity as well; obviously, relabelings, graph-for-symbol substitutions and finite unions also preserve bounded connectivity. However, $\mathcal{L}_1$ from Section \ref{sec_lan_all} is of unbounded connectivity, so it cannot be obtained after any sequence of the abovementioned operations.

It is also interesting to answer the question whether languages generated by HL-grammars are closed under intersections (Theorem \ref{Theorem_intersections} gives us a hope that this could be true) but there is no clear way how to prove this (there are problems with $(\to\div)$ rules).
\section{Algorithmic Complexity}\label{sec_complex}
A series of fundamental questions we have to answer is related to algorithmic complexity of $\mathrm{HL}$ and of HLGs. How difficult is it to check derivability of a sequent? to check whether a given graph belongs to a given grammar?

In the string case the following theorem was proved by Pentus in \cite{Pentus_complexity}:
\begin{theorem}
	The problem of whether a given sequent $\Gamma\to A$ is derivable in the Lambek calculus is NP-complete.
\end{theorem}
Using this theorem it is easy to show that a problem of whether a given word $w$ belongs to the language generated by a given Lambek grammar $Gr$ is also NP-complete.

In the graph case the derivability problem and all the more the membership problem seem to be much harder: if we search for a derivation of a given graph sequent we have to choose a type to which the rule is applied and choose several subgraphs that will go to antecedents of premises (when we try to apply one of the rules $(\div\to)$ or $(\to\times)$). However, it is not hard to prove the following
\begin{theorem}\label{NP_der}
	The problem of whether a given graph sequent $H\to A$ is derivable in the hypergraph Lambek calculus is NP-complete.
\end{theorem}
\begin{proof}
	This problem is in NP: if $H\to A$ is derivable, then a certificate of derivability is a derivation tree of $H\to A$. This derivation tree has to include all steps of the derivation starting with axioms, and all isomorphisms between graphs in premises and in a conclusion that justify that a replacement (or a compression) is done correctly. Such a certificate has polynomial size w.r.t. size of $H\to A$ since the sum of sizes of all premises is strictly less than the size of a sequent in a conclusion (isomorphisms make it larger, but since each isomorphism can be represented as a list of correspondences between edges in graphs in premises and in a conclusion, their total size can be estimated by the size of a conclusion as well).
	
	NP-completeness directly follows from Theorem \ref{embed_lambek}: since the Lambek calculus is NP-complete, and it is embedded in $\mathrm{HL}$ (in polynomial time), the latter is NP-complete as well.
	\qed
\end{proof}
\begin{theorem}\label{NP_mem}
	The problem of whether a given graph $H\in\mathcal{H}(\Sigma)$ belongs to a given HLG $HGr=\langle\Sigma, S, \triangleright\rangle$ is NP-complete.
\end{theorem}
\begin{proof}
	This problem is in NP: if the answer is ``YES'', then its certificate is the function $f_H:E_H\to Tp(\mathrm{HL})$ such that $lab_H(e)\triangleright f_H(e)$ (the size of a description of this function can be estimated as $O(|E_H|)$) and a derivation of the sequent $f_H(H)\to S$, which also has polynomial size w.r.t. size of $H$ (see Theorem \ref{NP_der}).
	
	In order to check that this problem is NP-complete we reduce the derivability problem from Theorem \ref{NP_der} to it. If $G\to A$ is a graph sequent, and $E_G=\{e_1,\dots,e_n\}$, then we introduce a grammar $HGr=\langle\{a_1,\dots,a_n\},A,\triangleright\rangle$ where $a_i\triangleright lab_G(e_i)$, and a graph $H=g(G)$ where $g(e_i):=a_i$ for $i=1,\dots,n$. Clearly, $HGr$ and $H$ can be constructed in linear time w.r.t. the size of $G\to A$. Then $H\in L(HGr)\Leftrightarrow \mathrm{HL}\vdash G\to A$, which finishes the proof.
	
	Another way to observe NP-hardness is to note that HRGs can generate an NP-complete language $\mathcal{L}$ without isolated nodes and the empty graph (see \cite{Drewes}). Accordingly to Theorem \ref{hrg_hlg}, $\mathcal{L}$ can be also generated by a hypergraph Lambek grammar as well.
	\qed
\end{proof}
Both theorems themselves do not look unusual: we just notice that both problems can be certified in polynomial time and that they contain formalisms for which NP-completeness is well known. However, there are two important remarks that make these results more amazing:
\begin{enumerate}
	\item Lambek grammars have the same power with context-free grammars; however, the membership problem for the former is NP-complete while for the latter it is in P. In the graph case everything is different: the membership problems for HRGs and HLGs have the same algorithmic complexity (they are both NP-complete) while hypergraph Lambek grammars are much more powerful than hyperedge replacement grammars (see Section \ref{sec_power}).
	\item The Lambek calculus is NP-complete (in the sense of derivability problem), and so is the hypergraph Lambek calculus; however, the former can be embedded in the latter, and the latter deals with much more general structures than just strings.
\end{enumerate}
Hence, $\mathrm{HL}$ is more powerful than string formalisms from which it arises; however, it has the same algorithmic complexity, which is a great pleasure.
\section{Models for $\mathrm{HL}$}
Extending the string case again, we introduce algebraic models for $\mathrm{HL}$; however, in order to do this we firstly need to generalize some basic algebraic notions. Following \cite{Pentus_models}, we focus here on semigroup models.
\begin{definition}
	Let us fix a symbol $\verb*| |$ which does not occur in all other considered sets (formally, we fix a countable set of symbols of the form $\verb*| |_n$ and set $type(\verb*| |_n)=n$, $n\ge 0$; compare with the $\$$ symbol). A graph $H\in\mathcal{H}(\{\verb*| |\})$ is called \emph{unlabeled}. We fix an arbitrary edge ordering on unlabeled graphs: if $H$ is unlabeled, then $E_H=\{e_1(H),\dots,e_{|E_H|}(H)\}$.
\end{definition}
\begin{definition}\label{def_semigroup}
	A \emph{hypergraph semigroup} over the set $U\subseteq\mathcal{H}(\{\verb*| |\})$ of unlabeled graphs is a structure $\langle M,type,\{\circ[H]\}_{H\in U}\rangle$ where
	\begin{enumerate}
		\item $M$ is a carrier set.
		\item $type:M\to\mathbb{N}$ is a ranking function; denote by $M_k$ the set $\{a\in M|type(a)=k\}$.
		\item $\circ[H]:M_{i_1}\times\dots\times M_{i_n}\to M_{j}$ is an $n$-ary operation where $n=|E_H|$ and $i_k=type(e_k(H)),1\le k \le n$, $j=type(H)$.
	\end{enumerate}
	Regarding the set of operations we firstly impose the handle identity property: whenever we consider $H=\circledcirc(\verb*| |)$, $\circ[H](a)=a$ for all $a\in M_{type(H)}$. The next requirement is the following associativity property: if $G^i, H^i_1,\dots,H^i_{|E_{G^i}|}$, $i=1,2$ are unlabeled graphs such that each $G^i$ is from $U$ (denote $|E_{G^i}|$ as $L^i$), each $H^i_l$ is either from $U$ or equal to $\circledcirc(\verb*| |)$ and
	\begin{equation}\label{ass_req}
	\begin{split}
	K^1:=G^1[H^1_1/e_1(G^1),H^1_2/e_2(G^1),\dots,H^1_{L^1}/e_{L^1}(G^1)]
	=\\=
	G^2[H^2_1/e_1(G^2),H^2_2/e_2(G^2),\dots,H^2_{L^2}/e_{L^2}(G^2)]=:K^2
	\end{split}
	\end{equation}
	(here it is important to emphasize that graphs $K^1$ and $K^2$ are equal only up to isomorphism), then the following equality holds for elements $a_1,\dots,a_N$ of $M$ with appropriate types:
	\begin{multline}\label{ass_stat}
	\circ[G^1]\left(\circ[H^1_1]\left(a_1,\dots,a_{l^1_1}\right),\dots, \circ[H^1_q]\left(a_{N^1_q+1},\dots,a_{N^1_{q+1}}\right), \dots,\right.\\\left. \circ[H^1_{L^1}]\left(a_{N^1_{L^1}+1},\dots,a_{N^1_{L^1+1}}\right)\right)
	=
	\circ[G^2]\left(\circ[H^2_1]\left(a_{\pi(1)},\dots,a_{\pi(l^2_1)}\right),\dots, \right.\\\left.\circ[H^2_q]\left(a_{\pi(N^2_q+1)},\dots,a_{\pi(N^2_{q+1})}\right), \dots, \circ[H^2_{L^2}]\left(a_{\pi(N^2_{L^2}+1)},\dots,a_{\pi(N^2_{L^2+1})}\right)\right).
	\end{multline}
	Here $l^i_p=|E_{H^i_p}|, 1\le p\le L^i$, $N^{i}_{q+1}=l^i_1+\dots+l^i_q, 1\le q \le L_i$, $N^i_1=0$. Note that $N^{1}_{L^1+1}=N^{2}_{L^2+1}=N$. $\pi$ is a permutation from $S_N$, which is defined right below. Let $f^i:\{1,\dots, N\}\to E_{K^i}$ be the following function: $f^i(N^i_q+r)$ equals an edge $e_r(H^i_q)$ considered as a part of $K^i$ (since $K^i$ is obtained from $G^i$ by replacements with graphs $H^i_1,\dots,H^i_{L^i}$). One observes that $f^1$ and $f^2$ are bijective. Let $\mathcal{V}: V_{K^1}\to V_{K^2}$, $\mathcal{E}: E_{K^1}\to E_{K^2}$ be an isomorphism between $K_1$ and $K_2$. Then $\pi=(f^2)^{-1}\circ \mathcal{E} \circ f^1$.
\end{definition}
An important feature of this definition is that the set of operations here is indexed by graphs. 
\begin{example}
	Let $U=\{Str\}$ where $Str=\langle\{v_1,v_2,v_3\},\{e_1,e_2\},att,lab,v_1v_3\rangle$, $att(e_1)=v_1v_2$, $att(e_2)=v_2v_3$, $lab(e_1)=lab(e_2)=\verb*| |$; that is, $Str=(\verb*|  |)^\bullet$. Let $e_i(Str)=e_i,\;i=1,2$. Then each semigroup $\langle M,\circ\rangle$ in the common sense can be considered as a hypergraph semigroup $\langle M,type,\{\circ[Str]\}\rangle$ over $U$ where $type(a)=2$ for all $a\in M$, and $\circ[Str]=\circ$.
\end{example}
Let us introduce several other generalizations of notions regarding semigroup theory.
\begin{definition}\label{def_partial_semigroup}
	A \emph{partial hypergraph semigroup} over a set $U$ of unlabeled graphs is defined similarly to a hypergraph semigroup with that only difference that $\circ[H]$ are partial functions; the handle identity property remains the same (particularly, $\circ[\circledcirc(\verb*| |_k)]$ is defined on all elements of $M_k$); the associativity property now states that if all the operations in the left-hand side of (\ref{ass_stat}) are defined, then they are defined in its right-hand side, and (\ref{ass_stat}) holds.
\end{definition}
\begin{definition}\label{def_po_semigroup}
	A hypergraph semigroup $\langle M,type,\{\circ[H]\}_{H\in U},\{\le_n\}_{n\ge 0}\rangle$ is \emph{partially ordered} if $\{\le_n\}_{n\ge 0}$ is a set of binary relations such that $\le_n$ is defined on $M_n\times M_n$, $\langle M_n,\le_n\rangle$ is partially ordered, and the following monotonicity property holds: if $H$ belongs to $U$, $|E_H|=m$, and $a_1,\dots,a_m,b_1,\dots,b_m$ are such elements of $M$ that $a_i\le_{t_i} b_i$ ($1\le i \le m$, $t_i=type(a_i)=type(b_i)=type(e_i(H))$), then $\circ[H](a_1,\dots,a_m)\le_t\circ[H](b_1,\dots,b_m)$ for $t=type(H)$.
\end{definition}
\begin{definition}\label{def_resid_semigroup}
	A \emph{residuated hypergraph semigroup} is a partially ordered hypergraph semigroup $\langle M,type,\{\circ[H]\}_{H\in U},\{\le_n\}_{n\ge 0}\rangle$ such that for each $H\in U$ (denote $|E_H|=m$), $i\in\{1,\dots,m\}$, $a_1,\dots,a_{i-1},a_{i+1},\dots,a_m,b\in M$ with $type(a_j)=type(e_j(H))=t_j$ ($j\ne i$), $t=type(b)=type(H)$ there exists such an element $d\in M$, $type(d)=type(e_i(H))=t_i$ that for all $c\in M$, $type(c)=type(d)$ the following inequalites are equivalent:
	$$
	\circ[H](a_1,\dots,a_{i-1},c,a_{i+1},\dots,a_m)\le_{t} b \Leftrightarrow c\le_{t_i} d
	$$
	Such an element $d$ is unique (see the proof below); we denote it as $$d=\oslash[H](b/a_1,\dots,a_{i-1},\$,a_{i+1},\dots,a_m).$$
\end{definition}
\begin{proposition}
	The element $d$ from Definition \ref{def_resid_semigroup} is defined uniquely.
\end{proposition}
\begin{proof}
	Let $d_1,d_2$ be two elements satisfying requirements of Definition \ref{def_resid_semigroup}. Since $d_1\le_{t_i} d_1$, $d_2\le_{t_i} d_2$, the following holds for $k=1,2$:
	$$
	\circ[H](a_1,\dots,a_{i-1},d_k,a_{i+1},\dots,a_m)\le_{t} b
	$$
	Taking $c=d_1$ and $d=d_2$ in Definition \ref{def_resid_semigroup} yields that $d_1\le_{t_i} d_2$; taking $c=d_2$ and $d=d_1$ yields that $d_2\le_{t_i} d_1$. Thus $d_1=d_2$.
	\qed
\end{proof}
\begin{definition}
	A (partial, partially ordered, residuated) \emph{all-hypergraph semigroup} is a (partial, partially ordered, residuated) hypergraph semigroup over the set $\mathcal{H}(\{\verb*| |\})$ of all unlabeled graphs.
\end{definition}
Similarly to the notion of hypergraph furthermore we may omit the prefix \emph{hyper-} in all the terms defined above.

Generalizing the string case we consider residuated semigroups as models for the hypergraph Lambek calculus. It is not hard to see that $\circ$ from Definition \ref{def_po_semigroup} is somehow related to $\times$ in HL, and $\oslash$ from Definition \ref{def_resid_semigroup} has something in common with $\div$ in HL.
\begin{definition}\label{def_model_rsg}
	A \emph{residuated all-hypergraph semigroup model} $\langle RSG,w\rangle$ is a residuated all-graph semigroup $RSG=\langle M,type,\{\circ[H]\},\{\le_n\}_{n\ge 0}\rangle$ along with a valuation $w:Tp(\mathrm{HL})\to M$ satisfying the following reguirements:
	\begin{enumerate}
		\item\label{def_mrsg_0} $type(T)=type(w(T))$ for all types $T$;
		\item\label{def_mrsg_1} Let $\times(M)$ be a type, and let $u:E_M\to \{\verb*| |_n\}_{n\ge 0}$ be an unlabeling function ($u(e)=\verb*| |$). $u(M)$ is an unlabeled graph, so there is a fixed order on its edges (and, consequently, on edges of $M$): $E_M=E_{u(M)}=\{e_1(u(M)),\dots,e_{m}(u(M))\}$ where $m=|E_M|$; denote by $e_i\in E_M$ the edge $e_i(u(M))$. Then $$w(\times(M))=\circ[u(M)]\left(w(lab_M(e_1)),\dots,w(lab_M(e_m))\right).$$
		\item\label{def_mrsg_2} Let $\div(N/D)$ be a type. Let $u:E_D\to \{\verb*| |_n\}_{n\ge 0}$ be an unlabeling function. $u(D)$ is unlabeled, so there is a fixed order on its edges (and, consequently, on edges of $D$): $E_D=E_{u(D)}=\{e_1(u(D)),\dots,e_{l}(u(D))\}$ where $0<l=|E_D|$; denote by $f_j\in E_D$ the edge $e_j(u(D))$. Let finally $lab_D(f_i)=\$$ for some $1\le i \le l$. Then 
		\begin{multline*}
		w(\div(N/D))=\oslash[u(D)]\left(w(N)/w(lab_D(f_1)),\dots,w(lab_D(f_{i-1})),\$,\right.\\\left.w(lab_D(f_{i+1})),\dots, w(lab_D(f_{l}))\right).
		\end{multline*}
	\end{enumerate}
\end{definition}
Note that it suffices to define $w$ on primitive types, and conditions \ref{def_mrsg_1} and \ref{def_mrsg_2} allow one to extend $w$ to all types.
\begin{definition}
	A graph sequent $H\to A$ is true in a model $\langle RSG,w\rangle$ if $w(\times(H))\le w(A)$ where $\le$ stands for $\le_{type(A)}$.
\end{definition}
Below we prove correctness and completeness theorems. Their proof is similar to those in the string case.
\begin{theorem}[correctness]\label{th_correct}
	If $\mathrm{HL}\vdash H\to A$, then $H\to A$ is true in all residuated all-graph semigroup models.
\end{theorem}
\begin{proof}
	Induction on length of a derivation. 
	
	Induction basis: if $H=\circledcirc(A)$, then $w(\times(H))=\circ[\circledcirc(\verb*| |)](w(A))=w(A)\le w(A)$ (here we use the handle identity property).
	
	Induction step. There are four cases depending on the last rule applied in a derivation. Furthermore in this proof, we denote by $u(H)$ an unlabeled graph isomorphic to $H$; we also borrow some notations from Definition \ref{def_model_rsg}.
	\\
	Case $(\times\to)$ follows from associativity of $\circ[H]$ operations: if the last rule is of the form
	$$
	\infer[(\times\to)]{G\to A}{G[F/e]\to A}
	$$
	where $e\in E_G$ is labeled by $\times(F)$, then the valuation of the antecedent of the below sequent is expressed as $\circ[u(G)]$ with $|E_G|$ arguments where one of arguments is the result of the operation $\circ[u(F)]$; using associativity we reduce this composition to a single operation $\circ[u(G[F/e])]$ and apply the induction hypothesis.
	\\
	Case $(\to\times)$ follows from monotonicity of partially-ordered graph semigroups. If the last rule is of the form
	$$
	\infer[(\to\times)]{M[H_1/m_1,\dots,H_l/m_l]\to\times(M)}{H_1\to lab(m_1) & \dots & H_l\to lab(m_l)}
	$$
	then $w(\times(H_i))\le w(lab(m_i))$ by the induction hypothesis, and due to monotonicity and associativity $w(\times(M[H_1/m_1,\dots,h_l/m_l]))\le w(\times(M))$.
	\\
	Case $(\div\to)$ follows from monotonicity and from conditions on residuated semigroups. If $\div(N/D)$ is a type (we take notations from Definition \ref{def_model_rsg}), then the following sequent is derivable:
	$$
	D[f_i:=\div(N/D)]\to N
	$$
	This sequent is also true in all residuated all-graph semigroup models: it suffices to take $H=u(D)$, $a_j=w(lab_D(f_j))$ for $j\ne i$, $c=w(\div(N/D))$, $b=w(N)$; then the first inequality of Definition \ref{def_resid_semigroup} since the second one is of the form $w(\div(N/D))\le w(\div(N/D))$ so it obviously holds.
	\\
	Now, if $\mathrm{HL}\vdash H_j\to lab_D(f_j),j\ne i$, then due to monotonicity and associativity $$D[f_i:=\div(N/D)][H_1/f_1,\dots,H_{i-1}/f_{i-1},H_{i+1}/f_{i+1},\dots,H_l/f_l]\to N$$ is also true in all residuated all-graph semigroup models. Using monotonicity and associativity again, we complete this case for a general form of the $(\div\to)$ rule as in Section \ref{subsec_div_to}.
	\\
	Case $(\to\div)$ directly follows from the condition from Definition \ref{def_resid_semigroup}.
	\qed
\end{proof}
\begin{theorem}[completeness]
	If $H\to A$ is true in all residuated all-graph semigroup models, then $\mathrm{HL}\vdash H\to A$, and there exists a universal model (i.e. a model such that $H\to A$ is derivable if and only if $H\to A$ is true in this model).
\end{theorem}
\begin{proof}
	In Section \ref{sec_equiv} we defined the relation $\sim$. Consider the set $M$ of equivalences classes w.r.t. $\sim$ as a carrier set (i.e. $M=\{[A]|A\in Tp(\mathrm{HL})\}$ where $[A]=\{B\in Tp(\mathrm{HL})|B\sim A\}$). The function $type$ is defined on $M$ as follows: $type([A]):=type(A)$. For each unlabeled graph $H$ we define $\circ[H](A_1,\dots,A_n):=\times(H[e_1(H):=A_1]\dots[e_n(H):=A_n])$. We say that $[A] \le_n [B]$ if $type(A)=type(B)=n$, and $\mathrm{HL}\vdash \circledcirc(A)\to B$. $\le_n$ is a partial order: reflexivity follows from Proposition \ref{TtoT}, antisymmetry directly follows from the definition of $\sim$, and transitivity can be easily proven using the cut rule.
	
	We claim that $URSG=\langle M,type,\{\circ[H]\}_{H\in\mathcal{H}(\{\verb*| |\})}, \{\le_n\}_{n\ge 0}\rangle$ is a residuated all-graph semigroup. Indeed, we can define
	\begin{multline*}
	\oslash[G]([N]/[T_1],\dots, [T_{i-1}],\$,[T_{i+1}],\dots, [T_m]):=\big[\div(N/G[e_1(G):=T_1]\dots\\\dots[e_{i-1}(G):=T_{i-1}][e_i(G):=\$][e_{i+1}(G):=T_{i+1}]\dots[e_m(G):=T_m])\big].
	\end{multline*}
	The requirement imposed by Definition \ref{def_resid_semigroup} then follows from the rule $(\div\to)$ and Proposition \ref{reversibility}.
	
	Finally, we define a valuation $w$ as follows: $w(A):=[A]$. This valuation obviously satisfies all the required conditions, hence $\langle URSG, w\rangle$ is a residuated all-graph semigroup model. If $H\to A$ is true in this model, then $\mathrm{HL}\vdash \circledcirc(H)\to A$; using Proposition \ref{reversibility} we obtain $\mathrm{HL}\vdash H\to A$.
	\qed
\end{proof}
Therefore, the hypergraph Lambek calculus may be considered as a logic of residuated all-graph semigroups.

The next question is the following: can we restrict the class of residuated all-graph semigroup models to a weaker one but preserve completeness? In the string case this question is extensively studied; particularly, \cite{Pentus_models} is devoted to models of L based on semigroups and especially to so-called language models (or L-models) and to relational models (or R-models). Pentus proved in \cite{Pentus_models} that the Lambek calculus is complete w.r.t. L-models and w.r.t. R-models. 

It appears that all definitions of these models can be lifted to HL; however, completeness or incompleteness results are sometimes unexpected. We start with discussing how to construct a residuated graph semigroup over $U$ having a partial graph semigroup over $U$.
\begin{definition}
	Given a set $M$ with a function $type:M\to\mathbb{N}$ acting on it we define a typed powerset $\mathcal{P}_{type}(M)$ of $M$ as follows: $$\mathcal{P}_{type}(M)=\{A\subseteq M|\exists t\in\mathbb{N}:\forall a\in A \;type(a)=t\}$$
\end{definition}
Recall that if $f:A_1\times\dots\times A_n\to A$ is some partial function defined on sets $A_1,\dots,A_n$, then it can be generalized to a (total) function $F:\mathcal{P}(A_1)\times\dots\times \mathcal{P}(A_n)\to \mathcal{P}(A)$ in a natural way: $F(B_1,\dots,B_n)=\{f(b_1,\dots,b_n)|b_1\in B_1,\dots, b_n\in B_n\}$. Such a function $F$ is usually denoted by the same symbol $f$.
Note that $F$ is an increasing function of all its arguments: if $B_i\subseteq B^\prime_i$, then $$F(B_1,\dots,B_{i-1},B_i,B_{i+1},\dots,B_n)\subseteq F(B_1,\dots,B_{i-1},B^\prime_i,B_{i+1},\dots,B_n).$$
\begin{proposition}\label{prop_sem_to_res}
	Let $\langle M,type,\{\circ[H]\}_{H\in U}\rangle$ be a partial semigroup over $U$. Then $\langle \mathcal{P}_{type}(M),type,\{\circ[H]\}_{H\in U},\{\le_n\}_{n\ge 0}\rangle$ is a residuated graph semigroup over $U$ where 
	\begin{enumerate}
		\item $type(A)=type(a)$ for all $a\in A$ (this definition is correct since we consider a typed powerset);
		\item For each $H\in U$ $\circ[H]$ denotes a natural extension of the operation of the same name to powersets;
		\item $A\le_n B$ if $A\subseteq B$.
	\end{enumerate}
\end{proposition}
\begin{proof}
	$\{\circ[H]\}_{H\in U}$ defined on sets rather on elements of $M$ are total functions. The handle identity property obviously holds for these functions. A more interesting question is why the associativity property holds. This follows from Definition \ref{def_partial_semigroup}: it says that if a composition of operations is defined and there is another composition such that corresponding graphs obtained after replacements are isomorphic, then the result of another composition is also defined and their results coincide. 
	
	To check that this graph semigroup is partially ordered it suffices to notice that $\circ[H]$ defined on powersets is an increasing function of all its parameters w.r.t. $\le_n$.
	
	To show that this graph semigroup is residual we define
	\begin{multline*}
		\oslash[H](B/A_1,\dots,A_{i-1},\$,A_{i+1},\dots,A_n):=\\=\{d\in M_t|\circ[H](A_1,\dots,A_{i-1},\{d\},A_{i+1},\dots,A_n)\subseteq B\}
	\end{multline*}
	where $t=type(e_i(H))$. Let us denote $\oslash[H](B/A_1,\dots,A_{i-1},\$,A_{i+1},\dots,A_n)$ as $D$. Then the condition of Definition \ref{def_resid_semigroup} is obviously satisfied: $C\subseteq D$ if and only if $\circ[H](A_1,\dots,A_{i-1},C,A_{i+1},\dots,A_n)\subseteq B$.
	\qed
\end{proof}
Therefore, Proposition \ref{prop_sem_to_res} is a tool allowing us to build a residuated graph semigroup on the basis of a partial graph semigroup.

Let us consider two examples of partial all-graph semigroups:
\begin{enumerate}
	\item\label{ex_L_gsg} Let $C$ be a set of labels (not necessarily finite). Then $$\langle\mathcal{H}(C),type,\{\circ[H]\}_{H\in\mathcal{H}(\{\verb*| |\})}\rangle$$
	is an all-graph semigroup where $\circ[H](G_1,\dots,G_n)=H[G_1/e_1(H),\dots,G_n/e_n(H)]$.
	\item\label{ex_R_gsg} Let $M$ be a set. Then $$\langle M^\circledast,type,\{\circ[H]\}_{H\in\mathcal{H}(\{\verb*| |\})}\rangle$$
	is a partial all-graph semigroup where
	\begin{itemize}
		\item $type(x)=|x|$ whenever $x\in M^\circledast$;
		\item $\circ[H](x_1,\dots,x_n)=x$, if there is an injective function $f:V_H\to M$ such that $f(att_H(e_i(H)))=x_i$ for all $i=1,\dots, n$, and $f(ext_H)=x$. If such a function does not exist, then $\circ[H](x_1,\dots,x_n)$ is undefined.
	\end{itemize}
\end{enumerate}
Using Proposition \ref{prop_sem_to_res} we can perform the following scheme with all these partial all-graph semigroups: partial graph semigroup $\mapsto$ residuated graph semigroup $\mapsto$ residuated graph semigroup model. Models based on graph semigroups as in \ref{ex_L_gsg} are called language models or L-models; models based on partial graph semigroups as in \ref{ex_R_gsg} are called relational models or R-models (similarly to the string case). 

Why are they called so? Let us look closer how valuations in both cases work. 
\begin{enumerate}
	\item Each L-model includes a valuation $w:Tp(\mathrm{HL})\to \mathcal{P}_{type}(\mathcal{H}(C))$, that is, we assign a graph language to each type.
	\item A valuation $w$ in an R-model assigns a subset $w(T)\subseteq M^k$ for some $k\ge 0$ to each type $T$. Therefore, $w(T)$ can be considered as a $k$-ary relation on $M$. Note that if $type(T)=0$, then $w(T)=\emptyset$ or $w(T)=\{\Lambda\}$ (it is a nullary relation).
\end{enumerate}
Since both L- and R- models are residuated graph semigroup models correctness for them is known. Finishing this section we prove two results regarding completeness.
\begin{theorem}
	$\mathrm{HL}(\div)$ (that is, the fragment of HL with division only) is complete w.r.t. L-models.
\end{theorem}
Note that all the above definitions were given for HL, but they can be restricted to some its fragments, e.g. to $\mathrm{HL}(\div)$.
\begin{proof}
	Let $H\to A$ be true in all L-models. Consider the model based on the set $C=Tp(\mathrm{HL})$ of labels, and introduce a valuation $w$ such that $w(T)=\{G\in\mathcal{H}(Tp(\mathrm{HL}(\div)))|\mathrm{HL}\vdash G\to T\}$ whenever $T\in\mathrm{HL}(\div)$. Note that if $T=\div(N/D)$ and $lab_D(d_0)=\$$, then $\mathrm{HL}\vdash G\to T$ if and only if $\mathrm{HL}\vdash D[G/d_0]\to N$. This observation and the cut rule imply that $w$ satisfies condition \ref{def_mrsg_2} of Definition \ref{def_model_rsg} (and condition \ref{def_mrsg_0} clearly holds as well) so it is a correct valuation function.
	
	If $H\to A$ is true in this model, then $w(\times(H))\subseteq w(A)$; since $H$ belongs to $w(\times(H))$ ($\mathrm{HL}\vdash H\to \times(H)$) $H$ belongs to $w(A)$ as well; thus $\mathrm{HL}\vdash H\to A$.
	\qed
\end{proof}
Note that we cannot directly generalize this simple proof to HL: the problem is that $\mathrm{HL}\vdash H\to \times(M)$ does not imply that $H=M[H_1/e_1,\dots,H_m/e_m]$ where $E_M=\{e_1,\dots,e_m\}$ and $H_1,\dots,H_m$ are such graphs that $\mathrm{HL}\vdash H_i\to lab_M(e_i)$. A similar problem arises in the string case, and its solution is a difficult problem (in the string case it was solved by Pentus, see \cite{Pentus_models}). We have not studied yet whether the proof of Pentus can be lifted for graphs so completeness of HL w.r.t. L-models is still an open question.

Regarding R-models we establish the following result:
\begin{proposition}\label{prop_R_incompl}
	$\mathrm{HL}$ is not complete w.r.t. R-models.
\end{proposition}
\begin{proof}
	Fix some primitive type $p$ ($type(p)=1$). Consider the sequent $Y\to p$ where $Y=\langle\{v_1,v_2\},\{e_1,e_2\}, att,lab,v_1\rangle$ ($att(e_i)=v_i$, $lab(e_1)=lab(e_2)=p$):
	\begin{center}
	\begin{tikzpicture}
			\node[hyperedge] (E1) {$p$};
			\node[node,left=4mm of E1, label=left:{\scriptsize $(1)$}] (N1) {};
			\node[hyperedge,right=5mm of E1] (E2) {$p$};
			\node[node,right=4mm of E2] (N2) {};
			\node[right=2.5mm of N2] (To) {$\to p$};
			
			\draw[] (N1) -- (E1);
			\draw[] (E2) -- (N2);
	\end{tikzpicture}
	\end{center}
	Obviously, it is not derivable. Assume that there is some R-model based on a partial graph semigroup $$\langle M^\circledast,type,\{\circ[H]\}_{H\in\mathcal{H}(\{\verb*| |\})}\rangle$$ with a valuation $w$
	where $Y\to p$ is not true. But $w(\times(Y))=\circ[u(Y)](w(p),w(p))=\{a\in M|\exists b\ne a, b\in M:\; a,b\in w(p)\}\subseteq\{a\in w(p)\}=w(p)$. This leads to a contradiction.
	\qed
\end{proof}
This very simple proposition is thought-provoking: what is the reason of such incompleteness? Clearly, there is a very wide variety of sequents similar to one considered in this proposition. The question of how to overcome incompleteness w.r.t. R-models (by generalizing R-models? by adding more rules to $\mathrm{HL}$?) remains open. See more about this in Section \ref{sec_wc}.
\section{Further Modifications of HL}
Some other features of the Lambek calculus can be extended to hypergraphs; since we do not have much to say about them we will just present them below with some examples.
\subsection{Hypergraph Multiplicative-Additive Lambek Calculus}
In the string case, L can be extended by two operations $\wedge$ and $\vee$ called conjunction and disjunction respectively; this leads to the definition of the multiplicative-additive Lambek calculus (MALC). Rules designed for MALC do not actually exploit string nature so they can be easily lifted to HL. This leads to an extension of HL which we call \emph{hypergraph multiplicative-additive Lambek calculus (HMALC)}.

Types in HMALC are constructed using $\div,\times$ and also using $\wedge$ and $\vee$: if $A,B$ are types, then $A\wedge B, A\vee B$ are types as well. The following four rules for them are presented:
$$
\infer[(\wedge_i\to)]{H(e_0:A_1\wedge A_2)\to B}{H(e_0:A_i)\to B},\;i=1,2
\qquad
\infer[(\to\wedge)]{H\to A_1\wedge A_2}{H\to A_1 & H\to A_2}
$$
$$
\infer[(\vee\to)]{H(e_0:A_1\vee A_2)\to B}{H(e_0:A_1)\to B & H(e_0:A_2)\to B}
\qquad
\infer[(\to\vee_i)]{H\to A_1\vee A_2}{H\to A_i}
,\;i=1,2
$$
\begin{example}
	Consider the following types of HMALC:
	\begin{itemize}
		\item $T_1=\times\left(\mbox{
			{\tikz[baseline=.1ex]{
					\node[node,label=left:{\scriptsize $(1)$}] (N1) {};
					\node[node,right=10mm of N1,label=right:{\scriptsize $(2)$}] (N2) {};
					\draw[>=stealth,->,black] (N2) -- node[above] {$p$} (N1);
			}}
		}\right)\vee p;$
		\item $T_2=\div\left(q\middle/\mbox{
			{\tikz[baseline=.1ex]{
					\node[] (R1) {};
					\node[node,above=-2mm of R1, label=left:{\small $(1)$}] (N1) {};
					\node[node,right=10mm of N1] (N2) {};
					\node[node,right=10mm of N2, label=right:{\small $(2)$}] (N3) {};
					\draw[>=stealth,->,black] (N1) -- node[above] {$p$} (N2);
					\draw[>=stealth,->,black] (N2) -- node[above] {$\$$} (N3);
			}}
		}\right);$
		\item $T_3=\times\left(\mbox{
			{\tikz[baseline=.1ex]{
					\node[node] (N1) {};
					\node[node,right=10mm of N1] (N2) {};
					\draw[>=stealth,->,black] (N1) to[bend left=45] node[above] {$q$} (N2);
					\draw[>=stealth,->,black] (N1) to[bend right=45] node[below] {$T_2$} (N2);
			}}
		}\right).$
	\end{itemize}
	Then we can derive the sequent
	
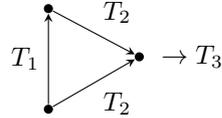
\begin{center}
	{\tikz[baseline=.1ex]{
			\node[node] (N1) {};
			\node[node,above=12mm of N1] (N2) {};
			\node[node,above right=6mm and 11mm of N1] (N3) {};
			\draw[>=stealth,->,black] (N1) -- node[left] {$T_1$} (N2);
			\draw[>=stealth,->,black] (N2) -- node[above right] {$T_2$} (N3);
			\draw[>=stealth,->,black] (N1) -- node[below right] {$T_2$} (N3);
			\node[right=1mm of N3] (To) {$\to T_3$};
	}}
	\end{center}
	as follows (draw your attention to arrow directions):
	\\
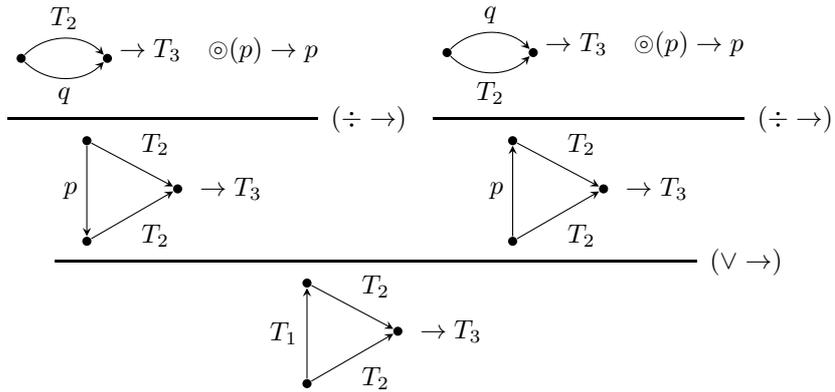
	$$
	\infer[(\vee\to)]{	
		\mbox{{\tikz[baseline=.1ex]{
					\node[node] (N1) {};
					\node[node,above=12mm of N1] (N2) {};
					\node[node,above right=6mm and 11mm of N1] (N3) {};
					\draw[>=stealth,->,black] (N1) -- node[left] {$T_1$} (N2);
					\draw[>=stealth,->,black] (N2) -- node[above right] {$T_2$} (N3);
					\draw[>=stealth,->,black] (N1) -- node[below right] {$T_2$} (N3);
					\node[right=1mm of N3] (To) {$\to T_3$};
	}}}}{
	\infer[(\div\to)]
	{\mbox{{\tikz[baseline=.1ex]{
				\node[node] (N1) {};
				\node[node,above=12mm of N1] (N2) {};
				\node[node,above right=6mm and 11mm of N1] (N3) {};
				\draw[>=stealth,->,black] (N2) -- node[left] {$p$} (N1);
				\draw[>=stealth,->,black] (N2) -- node[above right] {$T_2$} (N3);
				\draw[>=stealth,->,black] (N1) -- node[below right] {$T_2$} (N3);
				\node[right=1mm of N3] (To) {$\to T_3$};
	}}}}{\mbox{
	{\tikz[baseline=.1ex]{
	\node[node] (N1) {};
	\node[node,right=10mm of N1] (N2) {};
	\draw[>=stealth,->,black] (N1) to[bend left=45] node[above] {$T_2$} (N2);
	\draw[>=stealth,->,black] (N1) to[bend right=45] node[below] {$q$} (N2);
	}}}\to T_3
	&
	\circledcirc(p)\to p
	}
	&
	\infer[(\div\to)]
	{\mbox{{\tikz[baseline=.1ex]{
					\node[node] (N1) {};
					\node[node,above=12mm of N1] (N2) {};
					\node[node,above right=6mm and 11mm of N1] (N3) {};
					\draw[>=stealth,->,black] (N1) -- node[left] {$p$} (N2);
					\draw[>=stealth,->,black] (N2) -- node[above right] {$T_2$} (N3);
					\draw[>=stealth,->,black] (N1) -- node[below right] {$T_2$} (N3);
					\node[right=1mm of N3] (To) {$\to T_3$};
	}}}}{
		\mbox{
				{\tikz[baseline=.1ex]{
						\node[node] (N1) {};
						\node[node,right=10mm of N1] (N2) {};
						\draw[>=stealth,->,black] (N1) to[bend left=45] node[above] {$q$} (N2);
						\draw[>=stealth,->,black] (N1) to[bend right=45] node[below] {$T_2$} (N2);
			}}}\to T_3
		&
		\circledcirc(p)\to p
	}
	}
	$$
	The sequent 
	$\mbox{{\tikz[baseline=.1ex]{
				\node[node] (N1) {};
				\node[node,right=10mm of N1] (N2) {};
				\draw[>=stealth,->,black] (N1) to[bend left=45] node[above] {$q$} (N2);
				\draw[>=stealth,->,black] (N1) to[bend right=45] node[below] {$T_2$} (N2);
	}}}\to T_3$ is obviously derivable since the succedent equals the antecedent under $\times$.
	\qed
\end{example}
For HMALC we can reformulate a number of the above definitions and theorems such as embedding of MALC, the cut elimination, Lambek grammars. One difference is that now we cannot prove in the same simple way that the sequent derivability problem is NP-complete: a derivation tree can possibly be of exponential size w.r.t. a sequent one derives (due to rules $(\to\wedge)$ and $(\vee\to)$).
\subsection{Structural Rules}\label{sec_wc}
The Lambek calculus both in the string and in the graph cases lacks structural rules (except for the cut rule, which is admissible). In the string case there are extensions of L with such structural rules as weakening, contraction, permutation (the latter is presentes in Section \ref{sec_embed_lp}). It is known that MALC enriched with these three rules turns into the intuitionistic logic without negation. Therefore, it is decidedly interesting to answer a question whether these structural rules may be somehow generalized to HL (or to HMALC). Here we provide some ideas how this can be done (focusing on weakening and contraction rules).

In the string case weakening and contraction look as follows:
$$
\infer[(\mathrm{w})]{\Gamma,A,\Delta\to B}{\Gamma,\Delta\to B}\qquad
\infer[(\mathrm{c})]{\Gamma,A,\Delta\to B}{\Gamma,A,A,\Delta\to B}
$$
Weakening allows us to freely add types in an antecedent, and contraction allows us to remove a type from an antecedent if it appears twice in a row.

Our suggestion as to how these rules can be formulated in the graph case are the following.
\subsubsection{Weakening.}
Let $G\to A$ be a graph sequent, and let $G^\prime$ be such a graph that $G$ is its subgraph and $ext_G=ext_{G^\prime}$. Then
$$
\infer[(\mathrm{w_H})]{G^\prime\to A}{G\to A}
$$
Informally, $G^\prime$ is obtained from $G$ by adding some nodes and edges but without changing external nodes; if this is the case, and one derives $G\to A$, then $(\mathrm{w_H})$ allows one to derive $G^\prime\to A$ as well.
\subsubsection{Contraction.}
Let $H\to A$ be a graph sequent, and let $e_1,e_2\in E_H$ be such edges that $e_1\ne e_2$, $att_H(e_1)=att_H(e_2)$, $lab_H(e_1)=lab_H(e_2)$. Let $H^\prime$ be obtained from $H$ by removing $e_2$ (i.e. $H^\prime=\langle V_H,E_H\setminus\{e_2\},att_H|_{E_{H^\prime}},lab_H|_{E_{H^\prime}},ext_H\rangle$). Then
$$
\infer[(\mathrm{c_H})]{H^\prime\to A}{H\to A}
$$
That is, contraction allows one to remove multiple hyperedges with the same label.

If we add weakening (contraction/both) to HL, we denote this as $\mathrm{HL+w}$ ($\mathrm{HL+c}$, $\mathrm{HL+wc}$).
\begin{example}
	The sequent from Proposition \ref{prop_R_incompl} is derivable in $\mathrm{HL+w}$:
	$$
		\infer[(\mathrm{w_H})]{\mbox{{\tikz[baseline=.1ex]{
		\node[hyperedge] (E1) {$p$};
		\node[node,left=4mm of E1, label=left:{\scriptsize $(1)$}] (N1) {};
		\node[hyperedge,right=5mm of E1] (E2) {$p$};
		\node[node,right=4mm of E2] (N2) {};
		\node[right=2.5mm of N2] (To) {$\to p$};
		\draw[] (N1) -- (E1);
		\draw[] (E2) -- (N2);
		}}}}{\circledcirc(p)\to p}
	$$
	Here we just add extra $p$-labeled edge in the left-hand side.
\end{example}
\begin{example}
	The following sequent is derivable in $\mathrm{HL+c}$:
	$$
	\infer[(\mathrm{w_H})]{\circledcirc(p)\to \times\left(\mbox{{\tikz[baseline=.1ex]{
					\node[] (R) {};
					\node[hyperedge, above=-3mm of R] (E1) {$p$};
					\node[node,right=4mm of E1, label=below:{\scriptsize $(1)$}] (N1) {};
					\node[hyperedge,right=4mm of N1] (E2) {$p$};
					\draw[] (N1) -- (E1);
					\draw[] (E2) -- (N1);
	}}}\right)}{
	\infer[(\to\times)]{
	\mbox{{\tikz[baseline=.1ex]{
				\node[] (R) {};
				\node[hyperedge, above=-3mm of R] (E1) {$p$};
				\node[node,right=4mm of E1, label=below:{\scriptsize $(1)$}] (N1) {};
				\node[hyperedge,right=4mm of N1] (E2) {$p$};
				\draw[] (N1) -- (E1);
				\draw[] (E2) -- (N1);}}}
	\to \times\left(\mbox{{\tikz[baseline=.1ex]{
				\node[] (R) {};
				\node[hyperedge, above=-3mm of R] (E1) {$p$};
				\node[node,right=4mm of E1, label=below:{\scriptsize $(1)$}] (N1) {};
				\node[hyperedge,right=4mm of N1] (E2) {$p$};
				\draw[] (N1) -- (E1);
				\draw[] (E2) -- (N1);
}}}\right)}
{\circledcirc(p)\to p & \circledcirc(p)\to p}
}
	$$
	The last step is performed as follows: one of two $p$-labeled edges is removed from the left-hand side of a sequent.
\end{example}
An important remark regarding rules $(\mathrm{w_H})$ and $(\mathrm{c_H})$ is that they do not represent generalizations of rules $(\mathrm{w})$ and $(\mathrm{c})$ resp. in the sense that the former restricted to string graphs work in exactly the same way as the latter. Indeed, in the string case we can apply the rule $(\mathrm{w})$ as follows:
$$
	\infer[(\mathrm{w})]{p,p\to p}{p\to p}
$$
If we transform strings into string graphs in this one-step derivation, we obtain
$$
\infer[]{\mbox{
		{\tikz[baseline=.1ex]{
				\node[] (R1) {};
				\node[node,above=-2mm of R1, label=left:{\small $(1)$}] (N1) {};
				\node[node,right=10mm of N1] (N2) {};
				\node[node,right=10mm of N2, label=right:{\small $(2)$}] (N3) {};
				\draw[>=stealth,->,black] (N1) -- node[above] {$p$} (N2);
				\draw[>=stealth,->,black] (N2) -- node[above] {$p$} (N3);
		}}
	}\to p}{\mbox{
	{\tikz[baseline=.1ex]{
			\node[] (R1) {};
			\node[node,above=-2mm of R1, label=left:{\small $(1)$}] (N1) {};
			\node[node,right=10mm of N1, label=right:{\small $(2)$}] (N3) {};
			\draw[>=stealth,->,black] (N1) -- node[above] {$p$} (N3);
	}}
}\to p}
$$
However, this step is not an application of $(\mathrm{w_H})$: antecedents of these sequents violate the condition on external nodes. A correct application would be of the form
$$
\infer[(\mathrm{w_H})]{\mbox{
		{\tikz[baseline=.1ex]{
				\node[] (R1) {};
				\node[node,above=-2mm of R1, label=below:{\small $(1)$}] (N1) {};
				\node[node,right=10mm of N1, label=below:{\small $(2)$}] (N2) {};
				\node[node,right=10mm of N2] (N3) {};
				\draw[>=stealth,->,black] (N1) -- node[above] {$p$} (N2);
				\draw[>=stealth,->,black] (N2) -- node[above] {$p$} (N3);
		}}
	}\to p}{\mbox{
		{\tikz[baseline=.1ex]{
				\node[] (R1) {};
				\node[node,above=-2mm of R1, label=left:{\small $(1)$}] (N1) {};
				\node[node,right=10mm of N1, label=right:{\small $(2)$}] (N3) {};
				\draw[>=stealth,->,black] (N1) -- node[above] {$p$} (N3);
		}}
	}\to p}
$$
but in this derivation the antecedent of the below sequent is not a string graph. Furthermore, the rule $(\mathrm{w_H})$ is not applicable to string graphs at all.

This discrepancy between string and graph cases questions the way we introduced $(\mathrm{w_H})$ and $(\mathrm{c_H})$. To reason our definitions we provide two arguments in their favour. They are presented as propositions.
\begin{proposition}\label{prop_hlwc_correct}
	$\mathrm{HL+wc}$ is sound w.r.t. R-models.
\end{proposition}
\begin{proof}
	We prove by induction on length of a derivation that if $H\to A$ is derivable in $\mathrm{HL+wc}$, then it is true in all R-models. 
	
	The axiom case is similar to that of in Theorem \ref{th_correct} as well as cases corresponding to rules of HL. It remains to consider weakening and contraction as last rule applied in a derivation of $H\to A$. 
	
	\textbf{The $(\mathrm{w_H})$ case.} Let $G$ be a subgraph of $H$ such that $ext_G=ext_H$. Let the last step be of the form
	$$
	\infer[(\mathrm{w_H})]{H\to A}{G\to A}
	$$
	Let $E_G=\{e_1,\dots,e_m\}$, $V_G=\{v_1,\dots,v_k\}$ and let $E_H=\{e_1,\dots,e_m,e_{m+1},\dots,e_n\}$, $V_H=\{v_1,\dots,v_k,v_{k+1},\dots,v_l\}$; let $ext_G=ext_H=v_1\dots v_t$, $t\le m\le n$, $k\le l$. Without loss of generality we may assume that edges of unlabeled graphs $u(G)$ and $u(H)$ corresponding to $G$ and $H$ respectively are ordered in such a way that $e_i(u(G))=e_i$ ($i=1,\dots,m$), $e_i(u(H))=e_i$ ($i=1,\dots,n$). 
	
	Consider an R-model based on a set $M$ with a valuation $w$:
	\begin{equation*}
	\begin{split}
		w(\times(H))=\circ[u(H)](w(lab(e_1)),\dots,w(lab(e_n)))=
		\{a_1\dots a_t\in M^\circledast|\\\exists \mbox{ injective }f:V_H\to M:f(att_H(e_i))\in w(lab(e_i)),i=1,\dots,n;\\f(v_j)=a_j,j=1,\dots,t\}
		\subseteq
		\{a_1\dots a_t\in M^\circledast|\\\exists \mbox{ injective }f:V_G\to M:f(att_H(e_i))\in w(lab(e_i)),i=1,\dots,m;\\f(v_j)=a_j,j=1,\dots,t\}=w(\times(G))
	\end{split}
	\end{equation*}
	This inclusion reflects the fact that $H$ has more edges and nodes than $G$: $att_G=att_H|_{E_G}$, $V_G\subseteq V_H$, $E_G\subseteq E_H$.
	By the induction hypothesis $w(\times(G))\subseteq w(A)$; therefore, $w(\times(H))\subseteq w(A)$, which completes this case.
	
	\textbf{The $(\mathrm{c_H})$ case.} Let $G$ be a graph such that there are two edges (say $e_1$ and $e_2$) with the same label $T$ and the same ordered set of attachment nodes $\alpha$. Let $H$ be obtained from $G$ by removing $e_2$. Let the last rule be of the form
	$$
	\infer[(\mathrm{w_H})]{H\to A}{G\to A}
	$$
	Without loss of generality we may assume that $e_1(u(H))=e_1$, $e_1(u(G))=e_1$, $e_2(u(G))=e_2$, and $e_i(u(H))=e_{i+1}(u(G)),i=2,\dots,n$ (where $u(G),u(H)$ are unlabeled graphs corresponding to $G,H$ resp. and $n=|E_H|$). Then in an R-model based on a set $M$ with a valuation $w$ we have
	\begin{equation*}
	\begin{split}
	w(\times(H))=\circ[u(H)](w(lab(e_1)),w(lab(e_2(H))),\dots,w(lab(e_n(H))))=\\=
	\{v\in M^\circledast|\exists \mbox{ injective }f:V_H\to M:f(att_H(e_1))=f(\alpha)\in w(T);\\f(att_H(e_i(H)))\in w(lab_H(e_i(H))),i=2,\dots,n;f(ext_H)=v\}
	=\\=
	\{v\in M^\circledast|\exists \mbox{ injective }f:V_G\to M:f(att_G(e_1))=f(\alpha)\in w(T),\\f(att_G(e_2))=f(\alpha)\in w(T);f(att_G(e_i(G)))\in w(lab_G(e_i(G))),i=3,\dots,n+1;\\f(ext_G)=v\}
	=w(\times(G))
	\end{split}
	\end{equation*}
	Note here that $V_G=V_H$ and that $ext_G=ext_H$. Therefore, $w(\times(H))=w(\times(G))$. By the induction hypothesis, $w(\times(H))\subseteq w(A)$, which finishes the proof.
	\qed
\end{proof}
The proof basically generalizes that of Proposition \ref{prop_R_incompl}. Proposition \ref{prop_hlwc_correct} shows unexpected difference with the string Lambek calculus, which is complete w.r.t. R-models. An interesting open question naturally arises: is $\mathrm{HL+wc}$ complete w.r.t. R-models?

Another simple proposition shows connection between ersatz conjunction (see Definition \ref{def_ersatz}) and conjunction of HMALC. In the string case product $A\cdot B$ and conjunction $A\wedge B$ behave differently; particularly, neither $p\wedge q\to p\cdot q$ nor $p\cdot q\to p\wedge q$ is derivable in MALC. However, if we add weakening and contraction to MALC, then both sequents become derivable, and consequently $p\wedge q$ and $p\cdot q$ become equivalent. In the graph case similar things happen:
\begin{proposition}
	In $\mathrm{HMALC}$ enriched with $(\mathrm{w_H})$ and $(\mathrm{c_H})$ $\wedge_E(T_1,\dots,T_k)$ is equivalent to $T_1\wedge\dots\wedge T_k$.
\end{proposition}
Brackets in $T_1\wedge\dots\wedge T_k$ can be placed in any order since $\wedge$ is assosiative: $(A\wedge B)\wedge C\sim A\wedge(B\wedge C)$).
\begin{proof}
	$\circledcirc(\wedge_E(T_1,\dots,T_k))\to T_1\wedge\dots\wedge T_k$ is derived from bottom to top as follows: we apply $(\to\wedge)$ $(k-1)$ times and obtain $k$ sequents $\circledcirc(\wedge_E(T_1,\dots,T_k))\to T_i,\;i=1,\dots,k$. Each of these sequents can be derived using $(\mathrm{w_H})$: we remove all edges except for the one with the label $T_i$.
	
	$\circledcirc(T_1\wedge\dots\wedge T_k)\to \wedge_E(T_1,\dots,T_k)$ is derived from bottom to top as follows: we apply $(\mathrm{c_H})$ $(k-1)$ times and make $k$ copies of $T_1\wedge\dots\wedge T_k$ in the antecedent; thus we obtain a sequent $$\circledcirc(\wedge_E(T_1\wedge\dots\wedge T_k,\dots,T_1\wedge\dots\wedge T_k))\to \wedge_E(T_1,\dots,T_k).$$
	It is derivable using $(\times\to)$ and rules $(\wedge_i\to)$.
	\qed
\end{proof}
Summing up, weakening and contraction introduced in our way for HL are connected to a number of notions studied in this work, hence it seems that they are defined in a right way.
\section{Conclusion}
	Our goal to present a natural extension of the Lambek calculus to graphs in a way, which is somehow dual to hyperedge replacement grammars, is reached. The hypergraph Lambek calculus we have presented in this work seems to be an appopriate formalism satisfying all our requirements. This is justified by the fact that most of notions and results that exist for L can be naturally generalized to HL (with similar proofs). Moreover, we discovered that many fragments of L can be considered as fragments of HL. Unfortunately, definitions presented in this work are more cumbersome than those of L, but it is important to understand that they work in essentially the same way.
	
	In contrast to many established connections, in some aspects HL works in a different way than L. Below we list such cases, which are considered to be of interest:
	\begin{enumerate}
		\item Hypergraph Lambek grammars generate more languages than hyperedge replacement grammars; the class of languages generated by HL-grammars includes the set of all 2-graphs without isolated nodes; the set of all bipartite graphs without isolated nodes; most important, finite intersections of languages generated by hyperedge replacement grammars. Thus the famous theorem of Pentus about equivalence of context-free grammars and Lambek grammars in the string case cannot be generalized to graphs. What matters is that despite the fact that hypergraph Lambek grammars are more powerful than HRGs they are also NP-complete; thus we increase capabilities of grammars without increasing complexity.
		\item Languages generated by HL-grammars satisfy neither the pumping lemma for graph languages formulated in \cite{Drewes} nor the Parikh theorem; number of edges in such languages can grow nonlinearly.
		\item While a fundamental result regarding soundness and completeness w.r.t. residuated semigroup models for the Lambek calculus can be directly lifted to HL along with some other results regarding specific models (like L-models), it appears that HL is not complete w.r.t. R-models generalized to graphs. On the one hand, this may be caused by an incorrect definition of R-models; on the other hand, this can be an interesting result showing difference between string and graph cases. In Section \ref{sec_wc} we moreover show that one can add some kind of weakening and contraction rules preserving soundness of the hypergraph Lambek calculus.
	\end{enumerate}
	There is still much work to do. Throughout this work we mentioned several questions that would be interesting to investigate in the future. Some of open questions are listed below:
	\begin{enumerate}
		\item Studying properties of the Lambek calculus with weights.
		\item Further investigations of power of HL-grammars (e.g. whether they can generate the language of complete graphs, the language of grids and so on) and, as far as possible, describing the class of languages generated by HL-grammars. The same with grammars based on HMALC or on $\mathrm{HL}(\div)$.
		\item Finding a fragment of HL where the membersip property is in P (has polynomial-time complexity).
		\item Is HL complete w.r.t. L-models?
		\item Is $\mathrm{HL+wc}$ complete w.r.t. R-models?
	\end{enumerate}
	From our point of view, the hypergraph Lambek calculus is a nice generalization of the Lambek calculus. Unfortunately, due to generality of graph structures, constructions and reasonings in HL are sometimes complex and cumbersome; however, most of definitions and results are based on natural and simple ideas. HL provides a fresh look at the Lambek calculus and at the graph grammars, and we hope that it will be useful in further theoretical and practical investigations.

\newpage
\appendix
\addappheadtotoc
\noindent
{\bf\LARGE Appendices}
\section{Proofs}
\subsection{Theorem \ref{embed_lambek}}\label{embed_lambek_proof}
\begin{proof}
	The first statement is proved by a straghtforward remodelling of a derivation as well as in Theorems \ref{embed_nld}, \ref{embed_lp}; here we prove it in detail while omitting the proof for the rest of the abovementioned theorems.
	
	The proof is by induction on the size of derivation of $\Gamma\to C$ in $\mathrm{L}$. If $p\to p$ is an axiom, then $tr(p\to p)=\circledcirc(p)\to p$ is an axiom of $\mathrm{HL}$. To prove the induction step, consider the last step of a derivation:
	\begin{itemize}
		\item Case $(/\to)$:
		$$
		\infer[(/\to)]{\Psi, B / A, \Pi, \Delta \to C}{\Pi \to A & \Psi, B, \Delta \to C}
		$$
		By the induction hypothesis, $tr(\Pi \to A)$ and $tr(\Psi, B, \Delta \to C)$ are derivable in $\mathrm{HL}$. Note that $SG=tr(\Psi, B, \Delta)=(tr(\Psi)tr(B)tr(\Delta))^\bullet$ is a string graph with an edge (call it $e_0$) labeled by $tr(B)$. Then we can construct the following derivation:
		$$
		\infer[(\div\to)]{SG[D/e_0][d_0:=\div(tr(B)/D)][tr(\Pi)^\bullet/d_1] \to tr(C)}{SG \to tr(C) & tr(\Pi)^\bullet \to tr(A)}
		$$
		Here $D$ is the denominator of the type $tr(B/A)$, $E_D=\{d_0,d_1\}$ and $lab(d_0)=\$$, $lab(d_1)=tr(A)$. Finally note that $SG[D/e_0][d_0:=\div(tr(B)/D)][tr(\Pi)^\bullet/d_1]=tr(\Psi, B / A, \Pi, \Delta)$.
		\item Case $(\to/)$:
		$$
		\infer[(\to/)]{\Gamma \to B / A}{\Gamma, A \to B}
		$$
		Here $C=B/A$. By the induction hypothesis, $\mathrm{HL}\vdash tr(\Gamma,A\to B)$. Denote $SG=tr(\Gamma, A)^\bullet$ and $SF=tr(\Gamma)^\bullet$ its subgraph. Then
		$$
		\infer[(\to\div)]{SF \to \div(tr(B)/SG\llbracket \$/SF \rrbracket)}{SG \to tr(B)}
		$$
		Finishing the proof, we note that $\div(tr(B)/SG\llbracket \$/SF \rrbracket)=tr(B/A)$.
		\item Cases $(\backslash\to)$ and $(\to\backslash)$ are treated similarly.
		\item Case $(\cdot\to)$
		$$
		\infer[(\cdot\to)]{\Psi, A \cdot B, \Delta \to C}{\Psi, A, B, \Delta \to C}
		$$
		is remodeled (applying the induction hypothesis) as follows:
		$$
		\infer[(\times\to)]{tr(\Psi, A, B, \Delta)^\bullet\llbracket \times(SF)/SF \rrbracket \to tr(C)}{tr(\Psi, A, B, \Delta)^\bullet \to tr(C)}
		$$
		Here $SF=tr(A,B)^\bullet$; thus, $\times(SF)=tr(A\cdot B)$.
		\item Case $(\to\cdot)$:
		$$
		\infer[(\to\cdot)]{\Psi, \Delta \to A \cdot B}{\Psi \to A & \Delta \to B}
		$$
		is converted into
		$$
		\infer[(\to\times)]{tr(\Psi, \Delta)^\bullet \to tr(A \cdot B)}{tr(\Psi \to A) & tr(\Delta \to B)}
		$$
	\end{itemize}
	The second statement is of more interest since we know nothing about $G$ at first. The proof again is by induction on length of the derivation. If $G=\circledcirc(p)$ and $C=p$, then obviously $G\to C=tr(p\to p)$.
	\\
	For the induction step consider the last step of a derivation in $\mathrm{HL}$. Below $A,B$ are some types belonging to $Tp(\mathrm{L})$.
	\begin{itemize}
		\item Case $(\div\to)$: after application of this rule a type of the form $\div(tr(A)/D)$ has to appear. Note that $D$ is either of the form $(\$tr(B))^\bullet$ or of the form $(tr(B)\$)^\bullet$ for some $B$. Let $E_D=\{d_0,d_1\}$ and let $lab(d_0)=\$,lab(d_1)=tr(B)$. Then the application of this rule is of the form
		$$
		\infer[(\div\to)]{H[D/e][d_0:=\div(tr(A)/D)][H_1/d_1]\to T}{H\to T & H_1\to tr(B)}
		$$
		By the induction hypothesis, $H\to T=tr(\Gamma,A,\Delta\to C)$ (since $lab(e)=A$) and $H_1\to tr(B)=tr(\Pi\to B)$. Therefore, depending on structure of $D$, we obtain that $H[D/e][d_0:=\div(tr(A)/D)][H_1/d_1]\to T$ equals either $tr(\Gamma,\Pi,B\backslash A, \Delta\to C)$ or $tr(\Gamma,A/B,\Pi,\Delta\to C)$, which completes this case.
		\item Case $(\to\div)$: 
		$$
		\infer[(\to\div)]{G\to \div(tr(A)/H\llbracket \$/G\rrbracket)}{H\to tr(A)}
		$$
		By the induction hypothesis, $H\to tr(A)$ corresponds to a sequent of the Lambek calculus via $tr$. Note that $H\llbracket \$/G\rrbracket$ is of one of the following forms: $(\$tr(B))^\bullet$ or $(tr(B)\$)^\bullet$. Then the only possibility for $G$ is to be a string graph $(tr(\Pi))^\bullet$. Thus, $H$ equals either $tr(\Pi,B)^\bullet$ or $tr(B,\Pi)^\bullet$, and we can model this step in the Lambek calculus by means of $(\to/)$ or $(\to\backslash)$ resp.
		\item Case $(\times\to)$: by the induction hypothesis, a premise has to be of the form $tr(\Pi\to C)$. Then a conclusion is obtained from $tr(\Pi\to C)$ by compressing a subgraph $F$ of $tr(\Pi)$ into a type of the form $tr(A\cdot B)$. This implies that $F=tr(A,B)^\bullet$ and that this step can be modeled in $\mathrm{L}$ with the rule $(\cdot\to)$.
		\item Case $(\to\times)$: by the induction hypothesis, all antecedents of premises in this rule are string graphs. Since a succedent of the conclusion has the form $tr(A\cdot B)$, there are two premises, they have antecedents $tr(\Gamma)^\bullet$ and $tr(\Delta)^\bullet$ resp. and they are substituted in $tr(A\cdot B)$. This yields that $G\to T=tr(\Gamma,\Delta\to A\cdot B)$ and that this rule corresponds to $(\to\cdot)$ as expected.
	\end{itemize}
\qed
\end{proof}
\subsection{Theorem \ref{embed_nld}}
\begin{proof}[sketch]\label{embed_nld_proof}
	The first statement is simple, and it is proved by a straightforward induction.
	
	The second statement is also proved by induction on length of a derivation in $\mathrm{HL}$, and in general it is similar to the proof of Theorem \ref{embed_lambek} (see \ref{embed_lambek_proof}). The axiom case is the same. 
	\\
	To prove the induction step, we consider the last rule applied in a derivation of $G\to T$. We note that all premises of this rule have to correspond to sequents in $\mathrm{NL\diamondsuit}$ by the induction hypothesis; then it suffices to note that this last step transforms premises in such a way that the resulting sequent ($G\to T$) also corresponds to a derivable sequent in $\mathrm{NL\diamondsuit}$, i.e. $G\to T=tr_\diamond(\Gamma\to C)$.
	
	However, one difficulty arises. Let the last rule be, for instance, $(\div\to)$ and let, e.g., $tr_\diamond(A/B)$ appear after its application. Then one of premises has to be of the form $F\to p_{br}$ where $F$ is a subgraph of $G$. Unfortunately, we cannot apply the induction hypothesis to this premises since $p_{br}$ is not a type. However, we do not need this; instead we apply the wolf lemma. Note that for each type $T$ in the set $\mathcal{T}:=tr_\diamond(Tp(\mathrm{NL\diamondsuit}))\cup\{p_{br},p_\diamond\}$ it is true that $T$ does not have skeleton subtypes and that $p_{br}$ is lonely in $T$. Thus we can apply Corollary \ref{wolf} to $F\to p_{br}$ and obtain that $F=\circledcirc(p_{br})$. This is a desired result: $p_{br}$-labeled edges can interact only with $p_{br}$-labeled edges, hence they work in a way which corresponds to rules of $\mathrm{NL\diamondsuit}$. The same reasoning works with $p_\diamond$.
	\qed
\end{proof}
\subsection{Theorem \ref{cut}}\label{cut_proof}
\begin{proof}
	We prove that if $\mathrm{HL}\vdash H\to A$ and $\mathrm{HL}\vdash G\to B$, then $\mathrm{HL}\vdash G[H/e_0]\to B$ where $e_0\in E_G$ and $lab(e_0)=A$ by induction on $|H\to A|+|G\to B|$.
	
	Case 1: $H\to A$ is an axiom $\circledcirc(p)\to p$. Then $G[H/e_0]=G$, so the replacement changes nothing. 
	
	Case 2: $G\to B$ is an axiom $\circledcirc(p)\to p$. Then $A=lab(e_0)=B$, and $G[H/e_0]\to B = H\to A$, so the conclusion coincides with one of the premises.
	
	Let us further call the distinguished type $\div(N/D)$ in rules $(\div\to)$ and $(\to\div)$, and the distinguished type $\times(M)$ in rules $(\times\to)$ and $(\to\times)$ (see Section \ref{sec_axioms_rules}) the \emph{major type of the rule}.
	
	Case 3: in $H\to A$, the type $A$ is not the major type of the last rule applied. There are two subcases depending on the type of this rule.
	
	Case 3a. $(\div\to)$:
	$$
	\infer[(\mathrm{cut})]{G[H/e_0]\to B}{
		\infer[(\div\to)]{H\to A}{
			K\to A & H_1\to T_1 & \dots & H_k\to T_k
		}
		&
		G\to B
	}
	$$
	Here $H$ is obtained from $K$ by replacements using $H_1,\dots,H_k$ as the rule $(\div\to)$ prescribes. Note that we omit some details of rule applications that are not essential here.
	\\
	This derivation is transformed as follows:
	$$
	\infer[(\div\to)]{G[H/e_0]\to B}{
		H_1\to T_1 & \dots & H_k\to T_k &
		\infer[(\mathrm{cut})]{G[K/e_0]\to B}{
			K\to A & G\to B
		}
	}
	$$
	Now we apply the induction hypothesis to the premises and obtain a $(\mathrm{cut})$-free derivation for $G[H/e_0]\to B$. Further the induction hypothesis will be applied to the premises appearing in the new derivation process as well. Sometimes the induction hypothesis will be applied several times (from top to bottom, see Cases 5 and 6); however, this will be always legal.
	
	Case 3b. $(\times\to)$: let $H=K\llbracket \times(L)/L\rrbracket$ where $L$ is a subgraph of $K$. Then
	$$
	\infer[(\mathrm{cut})]{G[H/e_0]\to B}{
		\infer[(\times\to)]{K\llbracket \times(L)/L\rrbracket\to A}{
			K\to A
		}
		&
		G\to B
	}
	\rightsquigarrow
	\infer[(\times\to)]{G[H/e_0]\to B}{
		\infer[(\mathrm{cut})]{G[K/e_0]\to A}{
			K\to A & G\to B
		}
	}
	$$
	
	Case 4. The type $A$ labeling $e_0$ is not the major type in the last rule in the derivation of $G\to B$. Then one repeats the last step of the derivation of $G\to B$ in $G[H/e_0]\to B$ considering $H$ to be an atomic structure acting as $e_0$. There are five subcases corresponding to the type of the last rule:
	\begin{enumerate}[wide, labelwidth=!, labelindent=0pt]
		\item $(\div\to)$ if one of the invloved subgraphs contains $e_0$:
		$$
		\infer[(\mathrm{cut})]{G[H/e_0]\to B}{
			H\to A
			&
			\infer[(\div\to)]{G\to B}{
				K\to B & H_1\to T_1\;\dots & H_i\to T_i &\dots\; H_k\to T_k
			}
		}
		$$
		Let $H_i$ contain an edge $e_0$; then this derivation is remodeled as follows:
		$$
		\infer[(\div\to)]{G[H/e_0]\to B}{
			K\to B & H_1\to T_1\;\dots & \infer[(\mathrm{cut})]{H_i[H/e_0]\to T_i}{H\to A & H_i\to T_i} &\dots\; H_k\to T_k
		}
		$$
		
		\item $(\div\to)$ if $e_0$ is not contained in any $H_i$ (then $e_0$ belongs to $E_K$):
		$$
		\infer[(\mathrm{cut})]{G[H/e_0]\to B}{
			H\to A
			&
			\infer[(\div\to)]{G\to B}{
				K\to B & H_1\to T_1\; & \dots & \; H_k\to T_k
			}
		}
		$$
		\begin{center}
			$\downsquigarrow$
		\end{center}
		$$
		\infer[(\div\to)]{G[H/e_0]\to B}{
			\infer[(\mathrm{cut})]{K[H/e_0]\to B}
			{H\to A & K\to B} &
			H_1\to T_1\; & \dots & \; H_k\to T_k
		}
		$$
		\item $(\times\to)$: 
		$$
		\infer[(\mathrm{cut})]{G[H/e_0]\to B}{
			H\to A
			&
			\infer[(\times\to)]{K\llbracket \times(L)/L\rrbracket\to B}{K\to B}
		}
		$$
		Here $G=K\llbracket \times(L)/L\rrbracket$ and $e_0$ is not the edge obtained after this compression.
		$$
		\infer[(\to\times)]{G[H/e_0]\to B}{
			\infer[(\mathrm{cut})]{K[H/e_0]\to B}{
				H\to A
				&
				K\to B
			}
		}
		$$
		\item $(\to\div)$:
		$$
		\infer[(\mathrm{cut})]{G[H/e_0]\to B}
		{
			H\to A
			&
			\infer[(\to\div)]{G\to \div(N/K\llbracket \$/G\rrbracket)}{K\to N}
		}
		$$
		Here $B=\div(N/K\llbracket \$/G\rrbracket)$ and $G$ is considered to be a subgraph of $K$. Then
		$$
		\infer[(\to\div)]{G[H/e_0]\to \div(N/K\llbracket \$/G\rrbracket)}
		{
			\infer[(\mathrm{cut})]{K[H/e_0]\to N}{
				H\to A
				&
				K\to N
			}
		}
		$$
		\item $(\to\times)$: 
		$$
		\infer[(\mathrm{cut})]{G[H/e_0]\to \times(M)}{
			H\to A
			&
			\infer[(\to\times)]{G\to \times(M)}{
				H_1\to T_1\;\dots & H_i\to T_i &\dots\; H_k\to T_k 
			}
		}
		$$
		Here $G$ is composed of copies of $H_1,\dots,H_k$ by means of $M$. Since $e_0\in E_G$, there is such a graph $H_i$ that $e_0\in E_{H_i}$. Then we can remodel this derivation as follows:
		$$
		\infer[(\to\times)]{G[H/e_0]\to \times(M)}{
			H_1\to T_1\;\dots & \infer[(\mathrm{cut})]{H_i[H/e_0]\to T_i}
			{H\to A & H_i\to T_i} &\dots\; H_k\to T_k
		}
		$$
	\end{enumerate}
	
	Case 5: $A=\times(M)$ is major in both $H\to A$ and $G\to B$.
	$$
	\infer[(\mathrm{cut})]{G[H/e_0]\to B}
	{
		\infer[(\to\times)]{H\to\times(M)}{H_1\to T_1 & \dots & H_k\to T_k}
		&
		\infer[(\times\to)]{K\llbracket \times(M)/M\rrbracket\to B}{K\to B}
	}
	$$
	Here $G=K\llbracket \times(M)/M\rrbracket$, and $e_0$ is the edge that appears after this compression; let also denote $E_M=\{e_1,\dots,e_k\}$ and $lab_M(e_i)=T_i$. Note that $M$ is considered to be a subgraph of $K$, so particularly $E_M\subseteq E_K$. Now we are ready to remodel this derivation as follows:
	$$
	\infer[(\mathrm{cut})]{K[H_1/e_1]\dots[H_k/e_k]\to B}
	{
		H_k\to T_k
		&
		\infer[]{K[H_1/e_1]\dots[H_{k-1}/e_{k-1}]\to B}
		{
		\infer[(\mathrm{cut})]{\dots}
		{
			H_2\to T_2
			&
			\infer[(\mathrm{cut})]{K[H_1/e_1][H_2/e_2]\to B}
			{
				\infer[(\mathrm{cut})]{K[H_1/e_1]\to B}
				{
					H_1\to T_1
					&
					K\to B
				}
			}
		}
		}
	}
	$$
	Finally, note that $K[H_1/e_1]\dots[H_k/e_k]=G[H/e_0]$. The induction hypothesis applied several times from top to bottom of this new derivation completes the proof.
	
	Case 6: $A=\div(N/D)$ is major in both $H\to A$ and $G\to B$.
	$$
	\infer[(\mathrm{cut})]{G[H/e_0]\to B}
	{
		\infer[(\to\div)]{H\to\div(N/K\llbracket \$/H\rrbracket)}{K\to N}
		&
		\infer[(\div\to)]{G\to B}{L\to B & H_1\to T_1 & \dots & H_k\to T_k}
	}
	$$
	Here $D=K\llbracket \$/H\rrbracket$. We denote edges in $E_D$ except for the one labeled by \$ as $e_1,\dots,e_k$; let $lab_D(e_i)=T_i$ (from above). Note that $e_1,\dots,e_k$ can be considered as edges of $K$ as well. Observe that $L$ has to contain an edge labeled by $N$ that participates in $(\div\to)$; denote this edge by $\widetilde{e}_0$. Then the following remodeling is done:
	$$
	\infer[(\mathrm{cut})]{L[K/\widetilde{e}_0][H_1/e_1]\dots[H_k/e_k]\to B}
	{
		H_1\to T_1
		&
		\infer[]{L[K/\widetilde{e}_0][H_1/e_1]\dots[H_{k-1}/e_{k-1}]\to B}
		{
		\infer[(\mathrm{cut})]{\dots}
		{
			H_2\to T_2
			&
			\infer[(\mathrm{cut})]{L[K/\widetilde{e}_0][H_1/e_1][H_2/e_2]\to B}
			{
				\infer[(\mathrm{cut})]{L[K/\widetilde{e}_0][H_1/e_1]\to B}
				{
					H_1\to T_1
					&
					\infer[(\mathrm{cut})]{L[K/\widetilde{e}_0]\to B}
					{
						K\to N
						&
						L\to B
					}
				}
			}
		}
		}
	}
	$$
	As a final note, we observe that $L[K/\widetilde{e}_0][H_1/e_1]\dots[H_k/e_k]=G[H/e_0]$. This completes the proof.
	\qed
\end{proof}
\section{Examples}
\subsection{Derivation in the grammar $HGr_1$ from Section \ref{sec_lan_all}}\label{sec_lan_all_example}
Consider the 2-graph
$$
H=\mbox
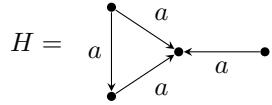
{
	{\tikz[baseline=.1ex]{
			\node[] (R) {};
			\node[node,above=4mm of R] (N1) {};
			\node[node,below=4mm of R] (N2) {};
			\node[node,right=6.93mm of R] (N3) {};
			\node[node,right=10mm of N3] (N4) {};
			
			%%\node[hyperedge,right=5mm of N3] (E3) {$Q_i$};
			
			\draw[>=stealth,->,black] (N1) -- node[left] {$a$} (N2);
			\draw[>=stealth,->,black] (N1) -- node[above right] {$a$} (N3);
			\draw[>=stealth,->,black] (N2) -- node[below right] {$a$} (N3);
			\draw[>=stealth,->,black] (N4) -- node[below] {$a$} (N3);
	}}
}
$$
In order to check that $H$ belongs to $L(HGr_1)$ we relabel it by corresponding types as follows:
$$
f_H(H)=\mbox
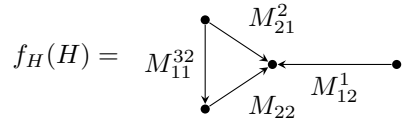
{
	{\tikz[baseline=.1ex]{
			\node[] (R) {};
			\node[node,above=4mm of R] (N1) {};
			\node[node,below=4mm of R] (N2) {};
			\node[node,right=6.93mm of R] (N3) {};
			\node[node,right=15mm of N3] (N4) {};
			
			%%\node[hyperedge,right=5mm of N3] (E3) {$Q_i$};
			
			\draw[>=stealth,->,black] (N1) -- node[left] {$M_{11}^{32}$} (N2);
			\draw[>=stealth,->,black] (N1) -- node[above right] {$M_{21}^{2}$} (N3);
			\draw[>=stealth,->,black] (N2) -- node[below right] {$M_{22}$} (N3);
			\draw[>=stealth,->,black] (N4) -- node[below] {$M_{12}^{1}$} (N3);
	}}
}
$$
Then we check derivability of $f_H(H)\to s$:
$$
\infer[(\times\to)]{
	\mbox{{\tikz[baseline=.1ex]{
				\node[] (R) {};
				\node[node,above=4mm of R] (N1) {};
				\node[node,below=4mm of R] (N2) {};
				\node[node,right=6.93mm of R] (N3) {};
				\node[node,right=15mm of N3] (N4) {};				
				\draw[>=stealth,->,black] (N1) -- node[left] {$M_{11}^{32}$} (N2);
				\draw[>=stealth,->,black] (N1) -- node[above right] {$M_{21}^{2}$} (N3);
				\draw[>=stealth,->,black] (N2) -- node[below right] {$M_{22}$} (N3);
				\draw[>=stealth,->,black] (N4) -- node[below] {$M_{12}^{1}$} (N3);
		}}
	}\to s}{
	\infer[(\times\to)]{
		\mbox{
			{\tikz[baseline=.1ex]{
					\node[] (R) {};
					\node[node,above=4mm of R] (N1) {};
					\node[node,below=4mm of R] (N2) {};
					\node[node,right=6.93mm of R] (N3) {};
					\node[node,right=15mm of N3] (N4) {};
					\node[hyperedge,left=4mm of N1] (E1) {$Q_3$};
					\node[hyperedge,left=4mm of N2] (E2) {$Q_2$};					
					\draw[>=stealth,->,black] (N1) -- node[above right] {$M_{21}^{2}$} (N3);
					\draw[>=stealth,->,black] (N2) -- node[below right] {$M_{22}$} (N3);
					\draw[>=stealth,->,black] (N4) -- node[below] {$M_{12}^{1}$} (N3);
					\draw[-,black] (E1) -- node[below] {\scriptsize 1} (N1);
					\draw[-,black] (E2) -- node[above] {\scriptsize 1} (N2);
			}}
		}\to s
	}{
		\infer[(\times\to)]{
			\mbox{
				{\tikz[baseline=.1ex]{
						\node[] (R) {};
						\node[node,above=4mm of R] (N1) {};
						\node[node,below=4mm of R] (N2) {};
						\node[node,right=6.93mm of R] (N3) {};
						\node[node,right=15mm of N3] (N4) {};
						\node[hyperedge,left=4mm of N1] (E1) {$Q_3$};
						\node[hyperedge,left=4mm of N2] (E2) {$Q_2$};					
						\node[hyperedge,above=4mm of N3] (E3) {$Q_2$};					
						\draw[>=stealth,->,black] (N2) -- node[below right] {$M_{22}$} (N3);
						\draw[>=stealth,->,black] (N4) -- node[below] {$M_{12}^{1}$} (N3);
						\draw[-,black] (E1) -- node[below] {\scriptsize 1} (N1);
						\draw[-,black] (E2) -- node[above] {\scriptsize 1} (N2);
						\draw[-,black] (E3) -- node[right] {\scriptsize 1} (N3);
				}}
			}\to s}{
			\infer[(\times\to)]{
				\mbox{
					{\tikz[baseline=.1ex]{
							\node[] (R) {};
							\node[node,above=4mm of R] (N1) {};
							\node[node,below=4mm of R] (N2) {};
							\node[node,right=6.93mm of R] (N3) {};
							\node[node,right=15mm of N3] (N4) {};
							\node[hyperedge,left=4mm of N1] (E1) {$Q_3$};
							\node[hyperedge,left=4mm of N2] (E2) {$Q_2$};					
							\node[hyperedge,above=4mm of N3] (E3) {$Q_2$};					
							\draw[>=stealth,->,black] (N4) -- node[below] {$M_{12}^{1}$} (N3);
							\draw[-,black] (E1) -- node[below] {\scriptsize 1} (N1);
							\draw[-,black] (E2) -- node[above] {\scriptsize 1} (N2);
							\draw[-,black] (E3) -- node[right] {\scriptsize 1} (N3);
					}}
				}\to s}{
				\infer[(\div\to)]{
					\mbox{
						{\tikz[baseline=.1ex]{
								\node[] (R) {};
								\node[hyperedge] (E1) {$Q_3$};
								\node[node,right=5mm of E1] (N1) {};
								\node[hyperedge,right=5mm of N1] (E2) {$Q_2$};					
								\node[node,right=5mm of 
								E2] (N2) {};
								\node[hyperedge,right=5mm of N2] (E3) {$Q_2$};					
								\node[node,right=5mm of E3] (N3) {};
								\node[hyperedge,right=5mm of N3] (E4) {$Q_1$};
								\node[node,right=5mm of E4] (N4) {};
								\draw[-,black] (E1) -- node[above] {\scriptsize 1} (N1);
								\draw[-,black] (E2) -- node[above] {\scriptsize 1} (N2);
								\draw[-,black] (E3) -- node[above] {\scriptsize 1} (N3);
								\draw[-,black] (E4) -- node[above] {\scriptsize 1} (N4);
						}}
					}\to s}{
					\infer[(\div\to)]{
						\mbox{
							{\tikz[baseline=.1ex]{
									\node[] (R) {};
									\node[hyperedge] (E1) {$Q_3$};
									\node[node,right=5mm of E1] (N1) {};
									\node[hyperedge,right=5mm of N1] (E2) {$Q_2$};					
									\node[node,right=5mm of 
									E2] (N2) {};
									\node[hyperedge,right=5mm of N2] (E3) {$p$};					
									\node[node,right=5mm of E3] (N3) {};
									\draw[-,black] (E1) -- node[above] {\scriptsize 1} (N1);
									\draw[-,black] (E2) -- node[above] {\scriptsize 1} (N2);
									\draw[-,black] (E3) -- node[above] {\scriptsize 1} (N3);
							}}
						}\to s}{
						\infer[(\div\to)]{
							\mbox{
								{\tikz[baseline=.1ex]{
										\node[] (R) {};
										\node[hyperedge] (E1) {$Q_3$};
										\node[node,right=5mm of E1] (N1) {};
										\node[hyperedge,right=5mm of N1] (E2) {$p$};					
										\node[node,right=5mm of 
										E2] (N2) {};
										\draw[-,black] (E1) -- node[above] {\scriptsize 1} (N1);
										\draw[-,black] (E2) -- node[above] {\scriptsize 1} (N2);
								}}
							}\to s}{\circledcirc(s)\to s & \circledcirc(p)\to p} & \circledcirc(p)\to p
					} & \circledcirc(p)\to p
				}
			}
		}
	}
}
$$

\end{document}